\documentclass{article}

\usepackage{xcolor}

\definecolor{linkblue}{RGB}{0,90,180}
\definecolor{lightlinkblue}{RGB}{70,150,230} 

\usepackage[
    pdftex,
    colorlinks=true,
    bookmarksopen=true,
    linkcolor=lightlinkblue,
    citecolor=linkblue,      
    urlcolor=linkblue       
]{hyperref}
\usepackage{xspace}
\usepackage[
    a4paper,
    left=3cm,
    right=3cm,
    top=3cm,
    bottom=3cm
]{geometry}

\usepackage{booktabs}
\usepackage[T1]{fontenc}
\usepackage[utf8]{inputenc}
\usepackage{lmodern}
\usepackage{libertine}
\usepackage{makecell}
\usepackage{bbm}
\usepackage{enumitem}
\usepackage{mathrsfs}
\usepackage{amssymb}
\usepackage{amsmath}
\usepackage{upgreek}
\usepackage{mathtools}
\usepackage{multirow,multicol}
\usepackage{amsthm}
\usepackage{comment}
\usepackage[font=small,labelfont=bf]{caption}
\usepackage{cleveref}
\theoremstyle{plain}
\newtheorem{theorem}{Theorem}[section]
\newtheorem{definition}{Definition}
\newtheorem{proposition}[theorem]{Proposition}
\newtheorem{lemma}[theorem]{Lemma}
\newtheorem{corollary}[theorem]{Corollary}

\newtheorem{remark}[theorem]{Remark}

\usepackage{graphicx} 
\usepackage{quoting}
\usepackage{subcaption}
\usepackage[numbers,sort&compress]{natbib}

\usepackage{lineno}

\DeclareMathOperator*{\argmin}{arg\,min}

\def\1{\mathbf{1}}

\newcommand{\Lip}{\mathrm{Lip}}

\newcommand{\eps}{\varepsilon}

\renewcommand{\beta}{\upbeta}
\renewcommand{\eta}{\upeta}
 
\renewcommand{\leq}{\leqslant}
\renewcommand{\geq}{\geqslant}

\newcommand{\diam}{\operatorname{diam}}
\newcommand{\dist}
{\operatorname{dist}}

\newcommand{\supp}{\operatorname{supp}}

\newcommand{\diff}{\,\mathrm{d}}

\newcommand{\bfa}{\mathbf{a}}
\newcommand{\bfb}{\mathbf{b}}
\newcommand{\bfc}{\mathbf{c}}

\newcommand{\bfe}{\mathbf{e}}

\newcommand{\bfr}{\mathbf{r}}
\newcommand{\bfs}{\mathbf{s}}

\newcommand{\bfu}{\mathbf{u}}
\newcommand{\bfv}{\mathbf{v}}
\newcommand{\bfw}{\mathbf{w}}
\newcommand{\bfx}{\mathbf{x}}
\newcommand{\bfy}{\mathbf{y}}
\newcommand{\bfz}{\mathbf{z}}

\newcommand{\bfS}{\mathbf{S}}

\newcommand{\bfX}{\mathbf{X}}
\newcommand{\bfY}{\mathbf{Y}}

\newcommand{\Pac}{\mathcal{P}_{\textrm{ac}}}

\newcommand{\N}{\mathbb{N}}

\newcommand{\R}{\mathbb{R}}

\newcommand{\PP}{\mathcal{P}}

\newcommand{\UIP}{\hyperref[def:uip]{UIP}\xspace}
\newcommand{\SCC}{\hyperref[def:cell_ctrl]{SCC}\xspace}

\usepackage{xcolor}
\definecolor{gblue}{HTML}{418BC0}
\definecolor{gpink}{HTML}{F17EA7}
\definecolor{ggreen}{HTML}{BDE0D2}
\definecolor{gpeach}{HTML}{F3CCB7}

\usepackage{pifont}

\title{Constructive interpolation and 
generalization rates for neural ODEs: a control perspective}

\author{
Antonio Álvarez-López\textsuperscript{1,2,*} \qquad
Lorenzo Liverani\textsuperscript{1} \qquad
Enrique Zuazua\textsuperscript{1,2,3}
}

\date{\vspace{-1cm}}

\begin{document}

\maketitle

\footnotetext[1]{
Chair for Dynamics, Control, Machine Learning, and Numerics
(Alexander von Humboldt Professorship), Department of Mathematics,
Friedrich-Alexander-Universität Erlangen-Nürnberg,
91058 Erlangen, Germany.
}

\footnotetext[2]{
Departamento de Matemáticas, Universidad Autónoma de Madrid,
28049 Madrid, Spain.
}

\footnotetext[3]{
Chair of Computational Mathematics, Universidad de Deusto,
Av. de las Universidades 24, 48007 Bilbao, Basque Country, Spain.
}

\renewcommand{\thefootnote}{\fnsymbol{footnote}}
\footnotetext[1]{Corresponding author. Email: \texttt{antonio.alvarezl@uam.es}}
\renewcommand{\thefootnote}{\arabic{footnote}}

\begin{abstract}
We study supervised regression with neural ODEs (NODEs) from a control-theoretic perspective to derive explicit population-risk bounds. We focus on a widely used class of non-autonomous models with constant parameters and explicit time dependence, which we call semi-autonomous NODEs (SA-NODEs). We constructively prove that SA-NODEs are capable of \emph{exact} interpolation of admissible finite datasets, and even satisfy a stronger property that we call \emph{simultaneous cell controllability} (SCC): their flows can map prescribed disjoint cells into arbitrarily small target balls. This property is the mechanism that upgrades interpolation into quantitative generalization, by allowing SA-NODEs to emulate piecewise-constant nonparametric estimators. Consequently, our risk bounds recover the rates of histogram and nearest-neighbor estimators, provided the network width satisfies a conservative scaling with the sample size. Numerical experiments show that trained SA-NODEs achieve competitive---often lower---test errors than these baselines. Finally, we show that the explicit time dependence is essential. Although two-layer autonomous NODEs can interpolate geometrically nondegenerate datasets, structural obstructions prevent them from achieving SCC. These limitations, further confirmed numerically, support the view that SA-NODEs provide a minimal effective architecture for learning.
\end{abstract}

\setcounter{tocdepth}{2}
\tableofcontents
\vspace{0.5cm}


\section{Introduction}\label{s:intro}

\subsection{Expressivity and generalization}

Supervised learning is one of the central paradigms in machine learning, with successful applications across a wide range of problems in science and engineering  \cite{LeCun2015Deep}
. The standard framework relies on a dataset of size $N \geq 1$,  consisting of pairwise distinct inputs $\bfx_i \in \mathcal{X}$ and corresponding labels $\bfy_i \in \mathcal{Y}$,
\begin{equation}\label{eq:dataset}
\mathcal{D}_N = \{(\bfx_i, \bfy_i)\}_{i=1}^N, 
\qquad \text{with } \bfx_i \neq \bfx_j \text{ for } i \neq j.
\end{equation}
The goal is to approximate the unknown relationship between inputs and labels by fitting a predictor to the sample $\mathcal{D}_N$, typically via a parametric family $\{F_\theta\}_{\theta}$, where $\theta$ is the learnable parameter. The choice of $\mathcal{Y}$ determines the task. In this work, we focus on regression in a deterministic, noiseless setting. Accordingly, we assume throughout that $\mathcal{X}\subseteq\mathbb{R}^d$ and $\mathcal{Y}=\mathbb{R}^d$ for $d\geq 2$, and that there exists a ground-truth map $y:\mathcal{X} \to \mathcal{Y}$ such that $y(\mathbf{x}_i)=\mathbf{y}_i$ for all $i$.

Of course, fitting $\mathcal D_N$ alone does not constitute learning: the true goal is to approximate $y$ accurately across $\mathcal{X}$. To assess performance beyond the training data, we fix a test probability measure $\mu \in \PP(\mathcal{X})$, often viewed as the (unknown) data-generating distribution, and define the population (true) risk by 
\begin{equation}\label{eq.pop.risk}
\mathcal{R}(\theta) \coloneqq \int \ell\left(F_\theta(\bfx), y(\bfx)\right)\diff\mu(\bfx),
\end{equation}
where $\ell \colon \mathcal{Y} \times \mathcal{Y} \to \R_{\geq 0}$ is a chosen loss function. Since $\mu$ is unknown, the population risk $\mathcal{R}(\theta)$ cannot be computed, let alone minimized, directly from~$\mathcal{D}_N$. One therefore introduces the empirical risk
\begin{equation}\label{eq.emp.risk}
\mathcal{R}_N(\theta) \coloneqq  \frac{1}{N}\sum_{i=1}^N \ell\left(F_\theta(\bfx_i), \bfy_i\right) = \int \ell\left(F_\theta(\bfx), y(\bfx)\right)\diff\mu_N(\bfx), \qquad \mu_N \coloneqq \frac{1}{N}\sum_{i=1}^N \delta_{\bfx_i} \in \PP(\mathcal{X}).
\end{equation}
Training then amounts to selecting parameters $\theta$ that minimize the empirical risk. This approach introduces a fundamental tension, which is revealed by the following error decomposition:
\begin{equation}\label{eq.err.decomp}
\mathcal{R}(\theta) \leq \mathcal{R}_N(\theta) + \left|\mathcal{R}(\theta)-\mathcal{R}_N(\theta)\right|.
\end{equation}
The first term measures the fit on the dataset; driving it to zero (exact interpolation) requires enough model \emph{expressivity}. The second term is the \emph{generalization gap}, measuring the discrepancy between the expected risk under $\mu$ and its empirical proxy $\mu_N$. Standard approaches usually control this gap by constraining the parameter space. This strategy is not designed for the interpolation regime, where models are inherently overparameterized but may still generalize, in tension with classical capacity-based theory \cite{Vapnik1998Statistical,Zhang2021Understanding,Belkin2019Reconciling,Bartlett2020Benign}. Our purpose is to exhibit a different mechanism, based on controllability of macroscopic regions, by which interpolation and generalization can coexist.



\subsubsection*{Controlled flows and neural ODEs}

We study expressivity and generalization within a class of predictors derived from dynamical systems. More precisely, we cast supervised learning under the guise of \emph{controlled flows}, modeled by the ODE
\begin{equation}\label{eq:gen_model}
\begin{dcases}
\dot{\bfx}(t) = v_{\theta(t)} (\bfx(t)), \qquad t \in (0,T], \\
\bfx(0) = \bfx_0\in\R^d,
\end{dcases}
\end{equation}
with $T>0$ fixed, and the vector field $v_{\theta(\cdot)}(\cdot):[0,T]\times\R^d\to\R^d$ that we assume to be measurable in $t$ and uniformly Lipschitz in $\bfx$. Under these conditions, \eqref{eq:gen_model} induces a unique flow 
\[
\Phi_t^\theta : \R^d \to \R^d, \qquad t \in [0,T],
\]
and the predictor is $F_\theta = \Phi_T^\theta$. This formulation offers a useful control-theoretic interpretation of deep learning \cite{E2017Proposal,Haber2018Stable,Lu2018Beyond}. The parameters $\theta(\cdot)$ act as \emph{controls} steering the dynamics, population risk minimization translates into an optimal control problem \cite{E2018Mean, Li2018Maximum, Esteve2020Large, CiprianiFornasierScagliotti2024}, and empirical risk minimization becomes, in essence, a problem of \emph{simultaneous controllability} of finitely many data points.


From a machine learning perspective, it is natural to parameterize $v_\theta$ as a neural network. With~this choice, \eqref{eq:gen_model} is termed a neural ODE \cite{Chen2018Neural}, with many variants differing in parameterizations, readouts or discretizations  \cite{Dupont2019Augmented,Rubanova2019Latent,Davis2020Time,Kidger2020Neural,Massaroli2020Dissecting,Celledoni2021Structure}. 
For a single-layer ReLU network of constant width $p\in\N$, we~obtain:
\begin{equation}
    \label{eq:NODE}\tag{\textsc{node}}
 \dot{\bfx}(t) = \sum_{i=1}^p \bfw_i(t)\big( \bfa_i(t)\cdot \bfx(t) + b_i(t)\big)_+, \qquad t \in [0,T].
\end{equation}
Here, $\theta=(\bfw_i,\bfa_i,b_i)_{i=1}^p : [0,T]\to(\R^d\times\R^d\times\R)^p$ are the controls, and $z_+\coloneqq\max\{z,0\}$ is the ReLU.

While \eqref{eq:NODE} is highly flexible, its effective parameter count scales with temporal discretization. Rather than mitigating this by prescribing flexible temporal parameterizations for all controls \cite{Davis2020Time,Massaroli2020Dissecting}, we seek simpler architectures that preserve expressivity. We focus on the \emph{semi-autonomous neural ODE}:
\begin{equation}
    \label{eq:saNODE}\tag{\textsc{sanode}}
        \dot{\bfx}(t) = \sum_{i=1}^p \bfw_i \big( \bfa_i\cdot \bfx(t) + b_i t + c_i \big)_+, \qquad t \in [0,T], 
\end{equation}
derived from \eqref{eq:NODE} by restricting $\bfw_i$ and $\bfa_i$ to be time-independent, and $b_i(t) = b_i t + c_i$. Equivalently, this corresponds to parameterizing $v_\theta$ via a time-invariant network applied to the augmented state $(\bfx,t)$. This design, standard in practice \cite{Chen2018Neural,Grathwohl2019Ffjord,Rubanova2019Latent}, makes \eqref{eq:saNODE} an appealing middle ground: its relative simplicity eases implementation compared to \eqref{eq:NODE} while retaining high expressivity \cite{Li2024Universal}.

\subsection{Main results}\label{ss:main_results}

Our goal is to establish expressivity and generalization guarantees for neural ODEs. We demonstrate that \eqref{eq:saNODE} achieves both, avoiding structural limitations of purely autonomous models.

\subsubsection{Expressivity.}

The study of approximation and controllability of neural ODEs has garnered significant attention \cite{Sontag1997rnns,Cuchiero2020Deep,Agrachev2022Control,Li2022Deep,Tabuada2023Universal,RuizBalet2023Neural,Li2024Deep,AlvarezLopez2024Interplay,Cheng2025Interpolation}. A central concept is the \emph{universal interpolation property} (\UIP, \cite{Cuchiero2020Deep,Cheng2025Interpolation}): given finitely many pairwise distinct inputs $\bfx_i$ and targets $\bfy_i$, does there exist a control $\theta(\cdot)$ such that $\Phi^\theta_T(\bfx_i) \approx \bfy_i$ for all $i$?  While \emph{approximate} interpolation (up to arbitrary precision) holds for various models, the exact~\UIP ($\Phi^\theta_T(\bfx_i) = \bfy_i$ for all $i$) is much more rigid. For time-dependent controls as in \eqref{eq:NODE}, constructive proofs typically introduce temporal discontinuities \cite{AlvarezLopez2024Interplay,RuizBalet2023Neural}. However, this machinery breaks down for (semi-)autonomous systems, leaving their exact \UIP as a challenging open question. 

Beyond exact interpolation of points, we introduce the notion of \emph{simultaneous cell controllability} (SCC) to control entire regions of the input space. As illustrated in \Cref{fig:process}, this requires the flow to map disjoint convex cells (e.g., a Voronoi partition) entirely into target balls $B(\bfr_k, \eta)$. Conceptually, this requirement sits between the \UIP and the more rigid notion of continuous ensemble controllability \cite{Agrachev2020Control,Scagliotti2025Minimax}, and related distributional questions arise in the study of normalizing flows 
\cite{Elamvazhuthi2022Neural,RuizBalet2024Control,
AlvarezLopez2025Constructive,Geshkovski2026Constructive}.

With these notions in place, our expressivity results for \eqref{eq:saNODE} establish the following:

\begin{itemize}[leftmargin=*, topsep=3pt, itemsep=2pt] 
    \item We prove the exact \UIP for \eqref{eq:saNODE} in \Cref{thm:SANODE_exact}. In the spirit of \cite{RuizBalet2023Neural,RuizBalet2022Interpolation,RuizBalet2024Control,AlvarezLopez2024Interplay,AlvarezLopez2025Cluster,AlvarezLopez2025Constructive,Geshkovski2026Constructive}, our approach is constructive and yields explicit controls rather than merely proving their existence.
    \item We formalize simultaneous cell controllability and prove that \eqref{eq:saNODE} exhibits this property in \Cref{thm:cell_ctrl}. Via Barron constant estimation \cite{Barron1993Universal,Klusowski2018Approximation}, we derive explicit bounds on the required width.
\end{itemize}



\begin{figure}[t]
    \centering
    \begin{subfigure}[t]{0.3\linewidth}
        \centering
        \includegraphics[width=\linewidth]{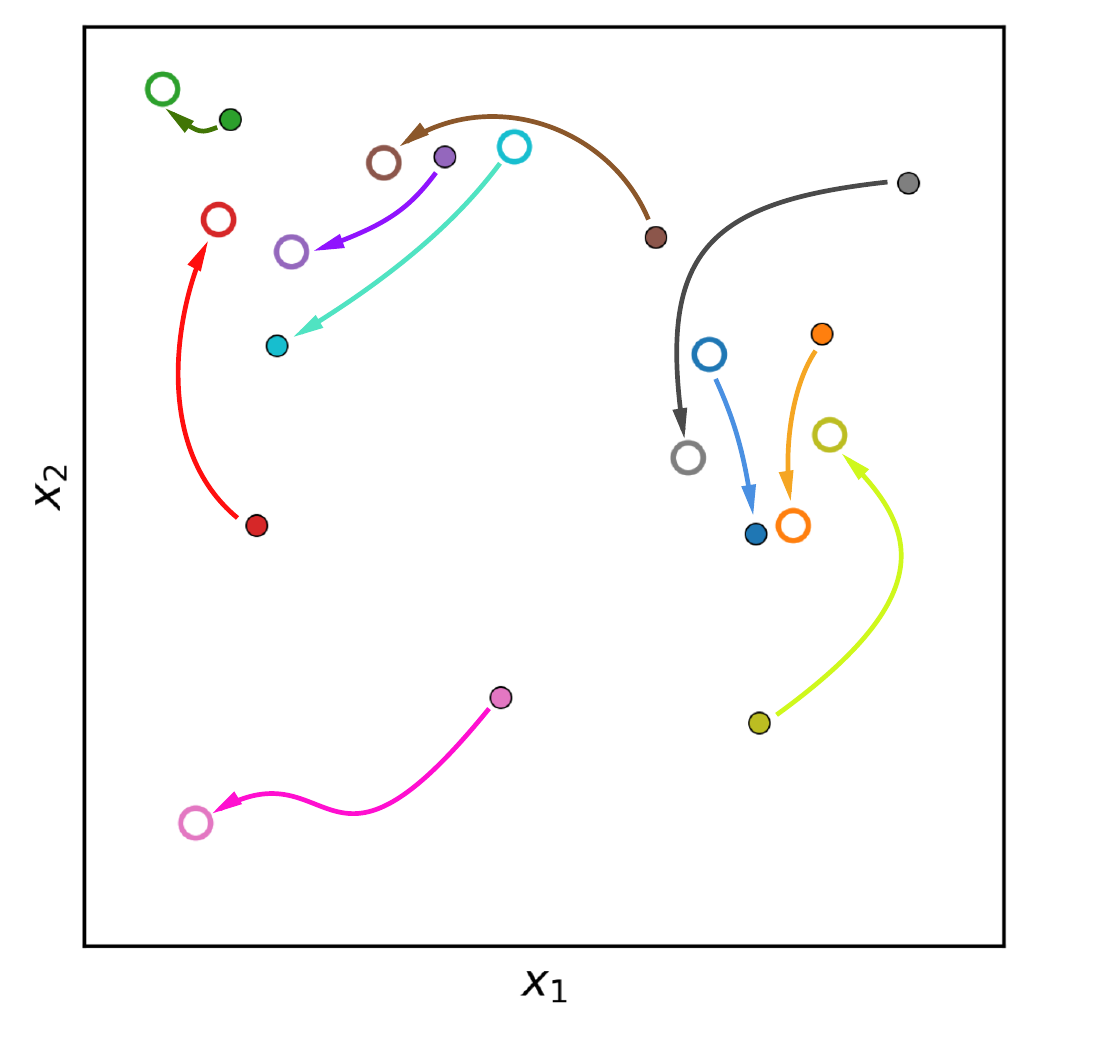}
    \end{subfigure}
    \hfill
    \begin{subfigure}[t]{0.3\linewidth}
        \centering
        \includegraphics[width=\linewidth]{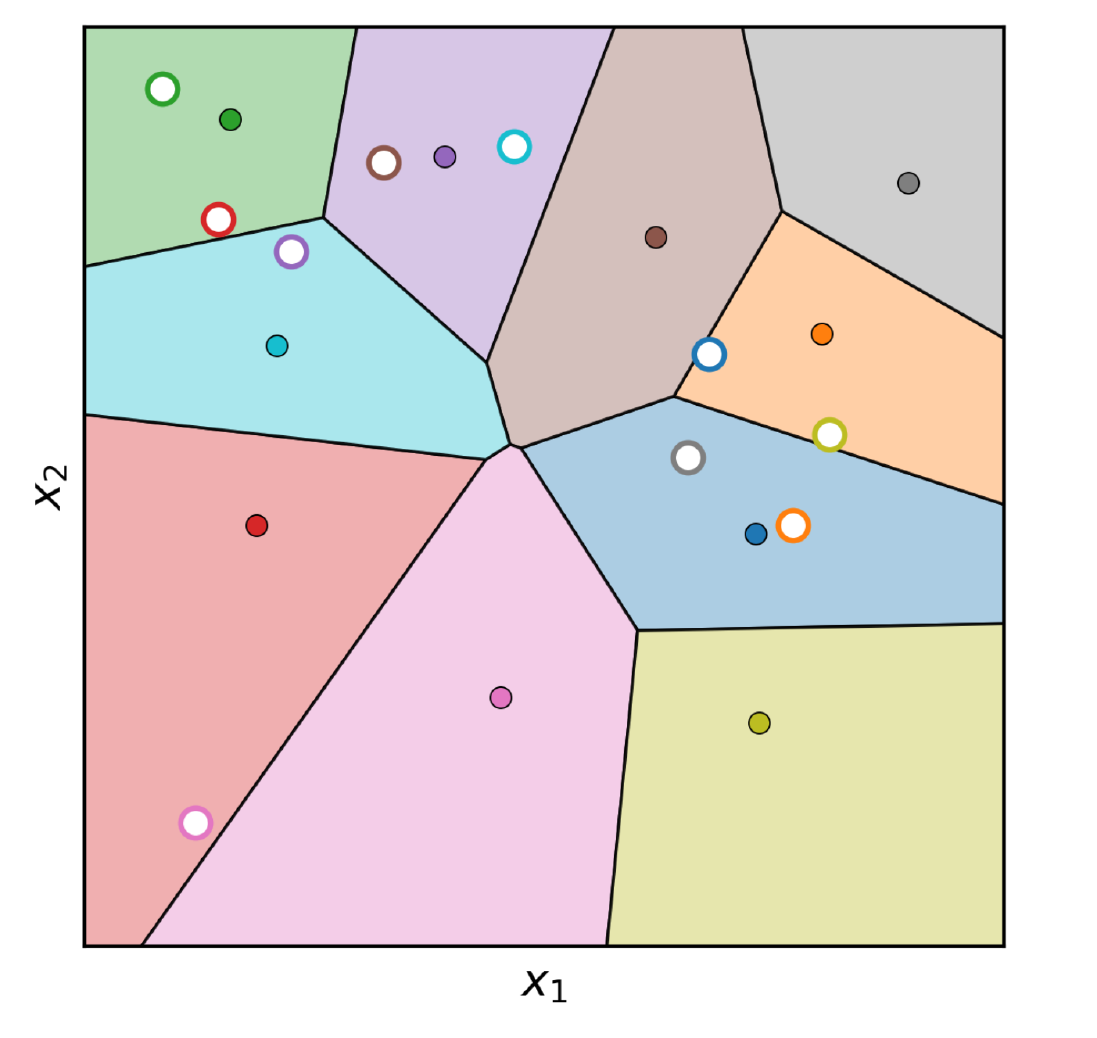}
    \end{subfigure}
    \hfill
    \begin{subfigure}[t]{0.3\linewidth}
        \centering
        \includegraphics[width=\linewidth]{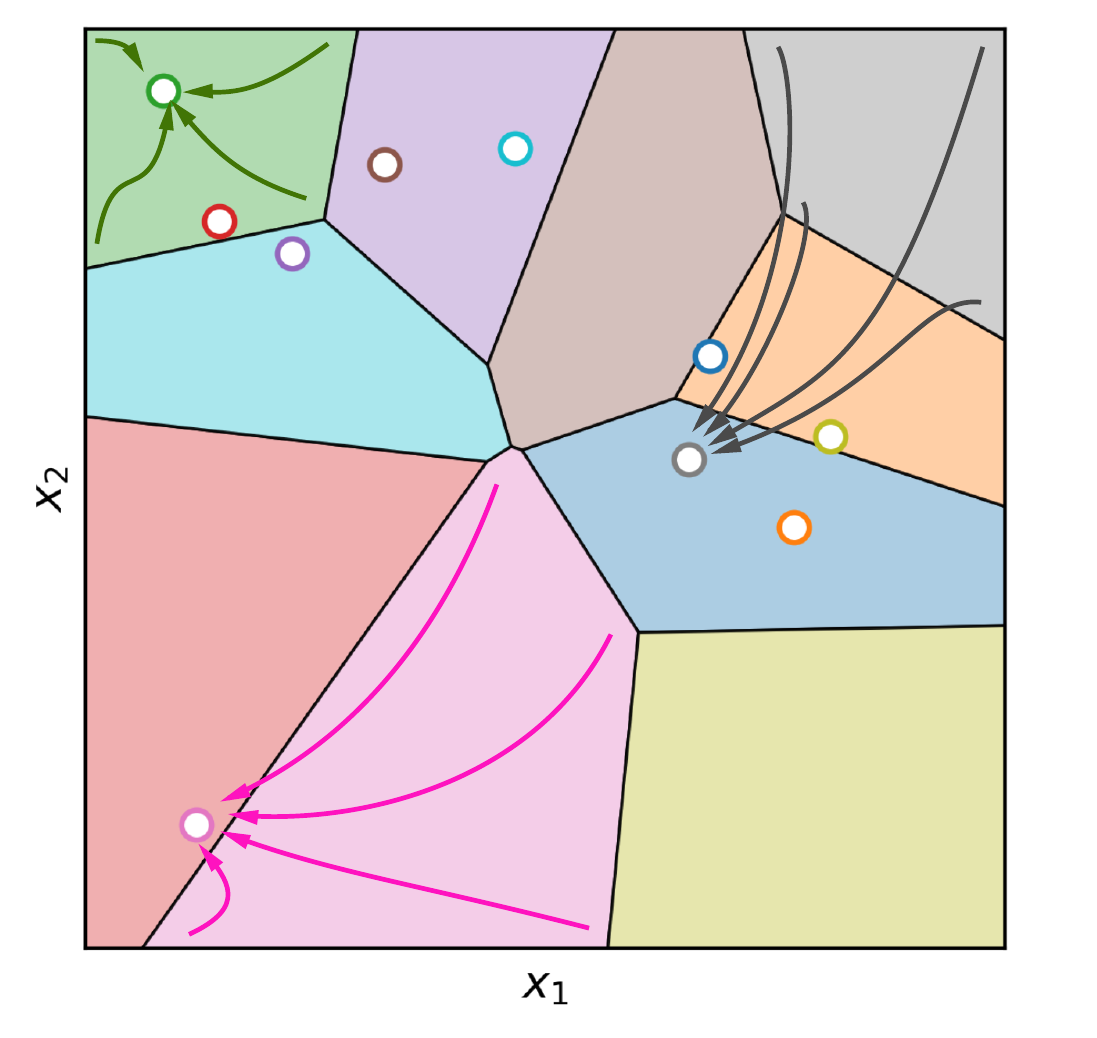}
    \end{subfigure}
   \caption{Overview of the proposed approach. \textbf{Left:} Interpolation of the training data from inputs (solid circles) to targets (hollow circles) via learned trajectories (arrows). \textbf{Center:} Voronoi partition of the input space, where each cell corresponds to a training sample. \textbf{Right:} The flow compresses entire cells toward their respective targets, illustrating the simultaneous cell controllability from \Cref{def:cell_ctrl}.}
    \label{fig:process}
\end{figure}

\subsubsection{Generalization.}
Under the setting of \eqref{eq:gen_model} and the squared Euclidean loss, the population and empirical risks \eqref{eq.pop.risk}--\eqref{eq.emp.risk} become:
\begin{equation}\label{eq:risks}
\mathcal{R}(\theta)=\int \left\|\Phi_T^\theta(\bfx)-y(\bfx)\right\|^2\diff\mu(\bfx),
\qquad
\mathcal{R}_N(\theta)=\frac1N\sum_{i=1}^N \left\|\Phi_T^\theta(\bfx_i)-\bfy_i\right\|^2.
\end{equation}
The central challenge remains bounding the decomposition \eqref{eq.err.decomp}. Existing theoretical approaches typically bound the generalization gap uniformly over a restricted parameter space $\Theta$. For instance, \cite{Marion2023Generalization} proves that for the linear-in-control version of \eqref{eq:NODE} under bounded Lipschitz controls, with probability $1-\delta$:
\[
\mathcal{R}(\theta)
\leq
\mathcal{R}_N(\theta)
+
\sup_{\theta\in\Theta} |\mathcal R(\theta)-\mathcal R_N(\theta)|
\lesssim
\mathcal{R}_N(\theta)
+
\sqrt{\frac{(m+1)\log(R_\Theta mN)}{N}}
+
\frac{m\sqrt{R_\Theta}}{N^{1/4}}
+
\sqrt{\frac{\log(1/\delta)}{N}}
\]
where $m$ is the number of scalar control functions and $R_\Theta\ge0$ encapsulates the bounds on $\Theta$. If the controls are time-independent, this bound improves by removing the $O(N^{-1/4})$ term.

While informative, such bounds highlight an inherent tension: keeping the complexity ($m$ and $R_\Theta$) fixed as $N \to \infty$ shrinks the generalization gap but may preclude exact interpolation, leaving a large empirical risk. Conversely, driving $\mathcal{R}_N(\theta)$ to zero requires increasing complexity when $N$ grows large, which deteriorates the bound. This motivates a natural \textbf{question}:
\vspace{-1mm}
\begin{quote}
   \emph{Can we design data-dependent controls $\theta_N$, while appropriately scaling the network width, to guarantee that the population risk $\mathcal{R}(\theta_N)$ vanishes at a quantifiable rate as $N \to \infty$?}
\end{quote}
\vspace{-1mm}
We provide an affirmative answer by connecting \eqref{eq:saNODE} to nonparametric statistics. Rather than assuming a fixed parametric form, nonparametric methods---such as kernel estimators \cite{Nadaraya1964EstimatingRegression,Watson1964SmoothRegression}, histograms or $k$-nearest-neighbors \cite{Gyorfi2002DistributionFreeRegression,Tsybakov2008Introduction}---construct local approximations directly from the data. For Hölder spaces, these procedures yield explicit, minimax-optimal convergence rates \cite{Stone1982Optimal}. This establishes a natural benchmark: any learning method claiming to generalize should ideally match these rates. We show that \eqref{eq:saNODE} achieves this benchmark by using simultaneous cell controllability to emulate these estimators.


\vspace{2mm}
\noindent\textsc{Informal statement} (\Cref{prop:sanode-rate,prop:sanode-voronoi}). Given a dataset $\mathcal D_N$, a compactly supported probability measure $\mu$, and an $\alpha$-Hölder continuous target $y$, there exists a control $\theta_N$ such that
\[
\mathcal R(\theta_N)
\lesssim
\mathsf{Err}_{\mathrm{np}}+\mathsf{Err}_{\mathrm{node}},
\]
where $\mathsf{Err}_{\mathrm{np}}$ is the statistical error of a piecewise-constant nonparametric estimator built from $\mathcal D_N$, and $\mathsf{Err}_{\mathrm{node}}$ measures how accurately the flow of \eqref{eq:saNODE} realizes this estimator. Provided the network width $p$ grows sufficiently fast with $N$ ($p \ge p_N$), the realization error is absorbed. Taking the expectation over the random draw of $\mathcal{D}_N$ (if the inputs $\bfx_i$ are sampled i.i.d.\ from $\mu$), we recover the classical rates:
\[
\mathbb E_{\mathcal D_N}\big[\mathcal R(\theta_N)\big]
\lesssim
\begin{dcases}
N^{-\frac{2\alpha}{2\alpha+d}} & \quad \text{for histogram estimators}, \\[6pt]
\left(\frac{\log N}{N}\right)^{\frac{2\alpha}{d}} & \quad \text{for nearest-neighbor estimators}.
\end{dcases}
\]
Our width requirement $p_N$---which depends on $N$ and the geometry of the data partition---acts as an achievability guarantee rather than a tight bound. 
It shows that the architecture can support generalization: given sufficient width, there exists a control whose true risk vanishes at an explicit rate. This circumvents the inherent tension of uniform bounds: exact interpolation and optimal statistical consistency can mathematically coexist.

\paragraph*{Related work.} 
Recent literature has explored generalization across various continuous-time models. Building on the uniform bounds of \cite{Marion2023Generalization} for the linear-in-control version of \eqref{eq:NODE}, \cite{Verma2025Analysis} extended this capacity-constrained approach to broader nonlinear neural ODEs. In parallel, other works have studied complementary mechanisms: \cite{Jia2025Feedback} showed that incorporating feedback enhances robustness, while \cite{Bleistein2024Generalization} derived generalization bounds for neural controlled differential equations driven by irregular time series. Complementary to our viewpoint, \cite{Marzouk2024Distribution} develops a
nonparametric statistical theory for likelihood-based distribution learning
with neural ODEs. Closest in spirit is \cite{Scagliotti2023Diffeomorphisms}, which studies the
interpolation--generalization trade-off for learning diffeomorphisms via an $L^2$-regularized optimal control formulation.

\subsubsection{The limits of autonomous models}
Having established the capabilities of \eqref{eq:saNODE}, it is natural to ask whether explicit time dependence is truly necessary. We may remove it entirely, leading to the \emph{autonomous neural ODE}:
\begin{equation}
    \label{eq:aNODE}\tag{\textsc{anode}}
    \dot{\bfx}(t) = \sum_{i=1}^p \bfw_i \big( \bfa_i\cdot \bfx(t) + b_i \big)_+, \qquad t \in [0,T].
\end{equation}
To compensate for the resulting structural constraints, we also consider increasing the depth of the neural network, yielding the \emph{two-layer autonomous neural ODE} (\textsc{2anode}):
\begin{equation}
\label{eq:2layerANODE}\tag{\textsc{2anode}}
\dot{\bfx}(t) = \sum_{j=1}^{p_2} \bfw_j \left( \sum_{i=1}^{p_1} u_{ji} \left( \bfa_{ij}\cdot \bfx(t) + b_{ij}\right)_+ + c_j \right)_+, \qquad t \in [0,T],
\end{equation}
with $\bfw_j, \bfa_{ij} \in \R^d$ and $u_{ji}, b_{ij}, c_j \in \R$ for $p_1,p_2\in\N$. 

While we prove in \Cref{thm:exact_2ANODE} that \eqref{eq:2layerANODE} can achieve the exact \UIP (provided the data satisfy the geometric condition \eqref{eq:interp:segments_disjoint}), we also discuss obstructions to simultaneous cell controllability in autonomous models like \eqref{eq:aNODE} and \eqref{eq:2layerANODE}. Topological constraints prevent autonomous flows from being universal approximators of arbitrary continuous maps \cite{Dupont2019Augmented}; however, our study reveals a more precise gap: they can satisfy the \UIP but fail at cell routing. This distinction, illustrated experimentally in \Cref{ss:necessity.timedep}, confirms that among the models we consider, \eqref{eq:saNODE} provides the simplest setting in which both properties coexist.

\subsection{Organization}
\Cref{s:expressivity} introduces the core controllability notions and provides constructive proofs of the exact \UIP for \eqref{eq:saNODE} and \eqref{eq:2layerANODE}. \Cref{s:nonparam} reviews the nonparametric estimators that serve as our benchmarks. \Cref{s:generalization} then states and discusses our main generalization theorems. In \Cref{s:num}, we numerically evaluate the capacity of neural ODEs to match or outperform these benchmarks on several regression tasks. Finally, \Cref{s:conclusions} offers concluding remarks and future directions. Most proofs are deferred to \Cref{appendix}.

\subsection*{Notation}
For $n \in \N$, we use $[n] = \{1, \dots, n\}$. Vectors are bold (e.g. $\bfu, \bfv$) and we denote the Euclidean inner product of $\bfu, \bfv \in \R^d$ by $\bfu \cdot \bfv$. The open ball in $\R^d$ of radius $r>0$ centered at $\bfx$ is  $B(\bfx, r)$, and the hypercube is $I_R \coloneqq [-R, R]^d$ for $R > 0$. Let $\mathcal{P}_{\mathrm{ac}}(\R^d)$ be the space of absolutely continuous probability measures on $\R^d$. We set
\begin{equation}\label{eq:log+}
\log_+(r)\coloneqq\max\{\log(r),0\}\quad \text{for }r>0,\qquad \text{and}\qquad    \log_+(0)=0.
\end{equation}
Concerning inequalities, we write $a \lesssim b$ (and $a \gtrsim b$) to indicate that $a \leq C b$ (and $a \geq C b$) for some constant $C > 0$ whose independence from the parameters of interest (such as $N$, $\varepsilon$, or $p$) will be clear from the context. If $a \lesssim b$ and $b \lesssim a$, we write $a \asymp b$.


\section{Controllability of neural ODEs}
\label{s:expressivity}


\subsection{Simultaneous point controllability}

We begin with our central notion of expressivity. Consider the general model \eqref{eq:gen_model} with flow map $\Phi_t^\theta$. Since $\Phi_T^\theta:\R^d\to\R^d$ is a homeomorphism, and hence injective, exact interpolation is impossible if distinct inputs $\bfx_i\neq \bfx_j$ are assigned the same target. To avoid this, we restrict our attention to datasets satisfying
\begin{equation}\label{eq:interp:distinct_targets}
\bfx_i\neq \bfx_j
\quad\text{and}\quad
\bfy_i\neq \bfy_j
\qquad\text{for all }i\neq j.
\end{equation}
Any dataset satisfying \eqref{eq:interp:distinct_targets} will be called \emph{admissible}.

\begin{definition}
\label{def:uip}
We say that the control system associated with \eqref{eq:gen_model} possesses the \emph{approximate universal interpolation property (\UIP)} if for every $\eps > 0$, every  $N \in \N$, and every dataset $\{(\bfx_i, \bfy_i)\}_{i=1}^N\subset\R^d\times\R^d$ admissible in the sense of \eqref{eq:interp:distinct_targets}, there exist $T>0$ and an admissible control $\theta$ such that the flow map $\Phi^\theta_T$ generated by \eqref{eq:gen_model} satisfies
\begin{equation}\label{eq:uip}
\|\Phi^\theta_T(\bfx_i) - \bfy_i\| \le \eps, \qquad \text{for all } i\in[N].
\end{equation}
If \eqref{eq:uip} holds with $\eps=0$, we say that the system possesses the \emph{(exact) \UIP}.
\end{definition}

When the vector field is a neural network, the admissible controls in \Cref{def:uip} are understood to range over all widths. Thus, the \UIP is a property of the whole family, not of a fixed-width model. 

\begin{remark}[Time rescaling]\label{rem:time_resc}
In the definition above, the time horizon $T$ is allowed to be chosen freely. However, this is equivalent to a fixed-horizon formulation whenever the family of vector fields under consideration is closed under positive scalar multiplication and time reparameterization. Indeed, for any $T_\ast>0$, the rescaled curve $\bfy(s) \coloneqq \bfx(sT/T_\ast)$ solves an analogous system on $[0,T_\ast]$ driven by the vector field $v_{\theta(sT/T_\ast)}T/T_\ast$, which remains in the same class. All models considered in this work satisfy this property, provided there are no uniform bounds imposed on the parameter space.
\end{remark}

Before detailing our contributions, we briefly review known results for the general model \eqref{eq:NODE}. The \UIP was proved in \cite{RuizBalet2023Neural} using $p=1$ neurons. We state its generalization to arbitrary $p\ge1$ from \cite{AlvarezLopez2024Interplay}.

\begin{theorem}[Simultaneous point controllability for \eqref{eq:NODE} {\citep[Theorem 1]{AlvarezLopez2024Interplay}}]\label{thm:trad_exact}
For every $d\ge2$, $T>0$ and $p\in\N$, the system \eqref{eq:NODE} possesses the exact \UIP. Furthermore, for any admissible dataset of size $N$ (in the sense of \eqref{eq:interp:distinct_targets}), the control can be chosen piecewise constant in time with $2\left\lceil N/p \right\rceil-1$ jumps.
\end{theorem}

The constructive proof of this theorem exploits the freedom afforded by time-dependent parameters. The key device is the orthogonality constraint $\bfw_i\cdot \bfa_i = 0$: geometrically, this forces the flow generated by each neuron to act as a shear transformation parallel to its own activation hyperplane. This decoupling makes the dynamics tractable and allows the control strategy to split naturally into two phases separated by a single switch. In the first interval, $N$ parallel hyperplanes steer $d-1$ coordinates of the data simultaneously; after the switch, the hyperplanes reorient to handle the remaining coordinate.

By contrast, for \eqref{eq:saNODE} the vectors $\bfw_i$ and $\bfa_i$ are time-independent, so the reorientation mechanism described above is no longer available. The compensating source of flexibility is the explicit time dependence in the argument $ \bfa_i\cdot\bfx+b_i t+c_i$, which makes the activation hyperplanes translate at a constant velocity given by $b_i$. Our first new result---whose proof is deferred to \Cref{ss:saNODE}---shows that this single degree of freedom is sufficient to recover exact controllability without requiring any temporal switching.

\begin{theorem}[Simultaneous point controllability for \eqref{eq:saNODE}]
\label{thm:SANODE_exact}
For every $d\ge2$ and $T>0$, the system \eqref{eq:saNODE} possesses the exact \UIP. Moreover, for any admissible dataset of size $N$ (in the sense of \eqref{eq:interp:distinct_targets}), exact interpolation can be achieved with width $p=2N$.
\end{theorem}

\subsection{Simultaneous cell controllability}

The \UIP captures the capacity to interpolate any admissible finite dataset. To control the population risk, however, one needs the flow to behave well across entire regions of the input space, rather than only at finitely many points. This leads to the following stronger controllability notion.

\begin{definition}\label{def:cell_ctrl}
We say that the control system associated with \eqref{eq:gen_model} possesses the property of \emph{simultaneous cell controllability} (\SCC) if, for every finite family of pairwise disjoint compact convex sets $\{A_k\}_{k=1}^K \subset \R^d$, every collection of target points $\{\bfr_k\}_{k=1}^K \subset \R^d$, and every $\eta > 0$, there exist $T>0$ and an admissible control $\theta$ such that the flow map $\Phi^\theta_T$ generated by \eqref{eq:gen_model} satisfies
\begin{equation}\label{eq:cell-wise-eq}
\Phi_T^\theta(A_k) \subset B(\bfr_k, \eta) \qquad \text{for all } k\in[K].
\end{equation}
\end{definition}

\begin{remark}
    Although framed here in the context of flows, simultaneous cell controllability is an architecture-independent notion. It applies to any parametric map $F_\theta$, from neural networks to classical regression models.
\end{remark}

The property of \SCC is stronger than the approximate \UIP, as setting $A_k=\{\bfx_k\}$ in \eqref{eq:cell-wise-eq} directly recovers condition \eqref{eq:uip}. The fundamental shift is that the flow must now simultaneously route whole macroscopic regions rather than isolated points. Despite this added complexity, the following theorem establishes that \eqref{eq:saNODE} exhibits this property, and provides a sufficient width threshold. This is derived in \Cref{ss:cell.ctrl} by constructing a reference vector field and applying quantitative Barron-type estimates (with the geometric constants detailed in \Cref{rem:Cgeo_cell_ctrl}).

\begin{theorem}[Simultaneous cell controllability for \eqref{eq:saNODE}]\label{thm:cell_ctrl}
Let $d,K\ge 2$ and $R,T>0$. Let $\{A_k\}_{k=1}^K$ be a family of pairwise disjoint, compact, and convex subsets of $I_R\coloneqq[-R,R]^d$, and let $\{\bfr_k\}_{k=1}^K \subset \R^d$ be pairwise distinct target vectors. Then there exist
\[
\mathfrak{C}=\mathfrak{C}\left(d,T,R,(\bfr_k)_k,(A_k)_k\right) >0,\qquad\mathfrak{M}=\mathfrak{M}\left(d,T,R,(\bfr_k)_k\right)>0
\]
and $\eta_0\in(0,1]$ such that for every $\eta\in(0,\eta_0]$ and every integer
\begin{equation}\label{eq:p_cell_ctrl_thm}
p\ge \mathfrak{C}\,\eta^{-2-\mathfrak{C}}\left(1-\log\eta\right)^2,
\end{equation}
there exists a control $\theta$ of width $p$ for \eqref{eq:saNODE} such that
\begin{enumerate}
    \item $\Phi_T^\theta(A_k) \subset B(\bfr_k,\eta)$ for all $k\in[K]$;
    \item $\|\Phi_T^\theta\|_{L^\infty(I_R)} \le \mathfrak M$.
\end{enumerate}
If the targets are not necessarily pairwise distinct, the same conclusions hold for each fixed $\eta>0$; however, the constant $\mathfrak{C}$ is no longer uniform in $\eta$.
\end{theorem}


\begin{remark}[Explicit geometric constants]
\label{rem:Cgeo_cell_ctrl}
The constants $\mathfrak C$, $\mathfrak M$, and $\eta_0$ in \Cref{thm:cell_ctrl} may be chosen explicitly. A possible choice, obtained by combining \Cref{lem:steering-estimates} with \Cref{lem:barron_approximation}, is governed by the geometry of the cells and their targets. Assume that the target vectors $\bfr_1,\dots,\bfr_K$ are pairwise distinct, let
\[
D_* \coloneqq \max_{k\in[K]}\diam(A_k),
\qquad
s_* \coloneqq \min_{i\neq j}\dist(A_i,A_j),
\]
and let $\gamma_k:[0,T]\to\R^d$ be smooth pairwise disjoint curves such that $\gamma_k(0)\in A_k$, $\gamma_k(T)=\mathbf{r}_k$, $\gamma_k$ is constant on $[0,2T/3]$, and $\|\gamma_k(t)\|<1+\max_{k\in[K]}\|\mathbf{r}_k\|+\sqrt{d}\,R$ for all $t\in[0,T]$---their existence is guaranteed for $d\ge2$, see \Cref{lem:steering-corridors}. Define
\[
m \coloneqq \min_{t\in[0,T]}\min_{i\neq j}\|\gamma_i(t)-\gamma_j(t)\|,
\qquad
G \coloneqq 1+\max_{k\in[K]}\max_{1\le j\le d+5} \left\|\frac{\diff^j\gamma_k}{\diff t^j}\right\|_{L^\infty(0,T)},
\]
and let
\[
L
\coloneqq
C_{d,T}\left[
\frac{G}{m}
+
\frac{1+D_*/s_*}{T}\left(1+\log_+(4D_*)\right)
\right],
\]
and
\[
B
\coloneqq
C_{d,T}\,K\left[
T G^{d+5}(m^d+m^{-4})
+
\frac{(D_*+s_*)^{d+1}(1+s_*^{-(d+4)})}{T}
\left(1+\log_+(4D_*)\right)
\right].
\]
Exploiting these quantities---see the proofs of \Cref{lem:steering-estimates,lem:barron_approximation} for a derivation---one may take
\begin{align}\label{eq:defMp}
\eta_0 \coloneqq \min\left\{1,\frac{s_*}{2},\frac{m}{2}\right\},\quad \mathfrak M \coloneqq 2+\max\left\{T,\max_{k\in[K]}\|\mathbf{r}_k\|+\frac32\sqrt{d}\,R\right\},\\
\mathfrak{C} \coloneqq  
\max\left\{
3,\,
2T L,\,
4C_d^2\,\mathfrak M^4\,T^2\,B^2 e^{2T L}
\right\}.\notag
\end{align}
Thus, while $\mathfrak M$ depends solely on the domain size and target locations, $\mathfrak{C}$ captures the geometric bottleneck of the partition: the term $1/s_*$ in $L$ forces an exponential penalty $\exp(c/s_*)$ in the factor $e^{2T L}$ as $s_* \to 0$.
\end{remark}

\begin{remark}[Width requirement and sharpness]\label{rem:superpoly}
The width condition \eqref{eq:p_cell_ctrl_thm} provides a sufficient condition rather than a sharp complexity estimate. As noted in \Cref{rem:Cgeo_cell_ctrl}, the geometric constant $\mathfrak{C}$ incurs an exponential penalty as the minimum cell separation $s_*$ shrinks. Consequently, for the refined partitions used later to establish statistical consistency, this geometric bottleneck can force a super-polynomial threshold $p_N$ as the dataset size $N$ grows. This reflects the curse of dimensionality inherent in isolating many small regions. 

We emphasize that this scaling acts purely as an achievability guarantee. The specific bound arises from our chosen proof strategy---specifically, applying the uniform approximation results of \cite[Theorem 2]{Klusowski2018Approximation}. Employing alternative tools (e.g., \cite[Proposition 6]{Bach2017Breaking}) could alter the trade-off between the convergence rate in $p$ and the magnitude of $\mathfrak{C}$. Determining the sharp intrinsic width required for simultaneous cell  controllability remains an open quantitative question.
\end{remark}

\subsection{Autonomous models}\label{sss:control_auton}

Having established the \UIP and \SCC for \eqref{eq:saNODE}, it is natural to ask whether explicit time dependence is truly necessary. To answer this, we remove it entirely and analyze the autonomous model \eqref{eq:aNODE} and its two-layer variant \eqref{eq:2layerANODE}.

While \cite{AlvarezLopez2024Interplay} establishes the approximate \UIP for \eqref{eq:aNODE}---including explicit error decay rates as the width grows---bridging the gap from approximate to exact interpolation remains unresolved:

\vspace{0.5em}
\noindent\textbf{Open Problem.} \emph{Does the exact \UIP hold for the single-layer autonomous system \eqref{eq:aNODE}?}
\vspace{0.5em}

The difficulty is structural: unlike time-dependent models, which can concatenate flow stages by switching controls, autonomous single-layer fields are highly rigid. Specifically, they cannot be spatially localized. Because they are finite superpositions of ridge functions---whose Fourier transforms are supported on one-dimensional sets---such nontrivial vector fields can never belong to $L^2(\R^d)$. Though elementary, this fact seems absent from standard references, so we provide a proof.

\begin{lemma}
\label{lem:non.integ}
Let $d\ge2$ and let $\sigma\in \mathscr{C}^0(\R)$ have at most polynomial growth.
For any $p\ge1$, define
\[
g(\bfx)=\sum_{j=1}^p w_j\,\sigma(\bfa_j\cdot \bfx+b_j),
\qquad (\bfx\in\R^d),\qquad \text{with }\bfa_j\in\R^d,\quad w_j,b_j\in\R.
\]
If $g\not\equiv0$, then $g\notin L^2(\R^d)$. In particular, $g$ cannot have compact support unless $g\equiv0$.
\end{lemma}

\begin{proof}
For each $j$, the Fourier transform of the ridge function $\bfx\mapsto \sigma(\bfa_j\cdot \bfx+b_j)$ is supported, in the sense of tempered distributions, on the line $L_j=\{\lambda \bfa_j:\lambda\in\R\}\subset\R^d.$ By linearity, $\widehat g$ is then supported in the finite union $\bigcup_{j=1}^p L_j$, which has Lebesgue measure zero in $\R^d$ since $d\ge2$ by hypothesis. 

Suppose $g\in L^2(\R^d)$. Then, by Plancherel, $\widehat g\in L^2(\R^d)$ as well. But $\widehat g=0$ a.e. because an $L^2$--function supported on a null set must vanish almost everywhere, and therefore $g\equiv0$.
\end{proof}

While \Cref{lem:non.integ} does not strictly rule out exact interpolation, it explains why ``pointwise'' steering strategies are exceptionally hard to implement with \eqref{eq:aNODE} alone. A minimal way to restore localization without reintroducing time dependence is to add one compositional layer. By composing ReLU units, the two-layer field \eqref{eq:2layerANODE} can generate vector fields supported on convex polytopes (see \Cref{lem:compact_support} below), enabling the construction of localized controls.

For \eqref{eq:2layerANODE}, we require the stronger assumption that input-output segments are mutually disjoint:
\begin{equation}\label{eq:interp:segments_disjoint}
[\bfx_i,\bfy_i]\cap[\bfx_j,\bfy_j]=\varnothing \qquad\text{for all } i\neq j,
\end{equation}
where $[\bfa,\bfb]$ denotes the line segment between $\bfa$ and $\bfb$. Note that this naturally implies $\bfx_i\neq \bfy_j$ for $i\neq j$.

\begin{theorem}[Simultaneous point controllability for \eqref{eq:2layerANODE}]\label{thm:exact_2ANODE}
Let $N\ge 1$, $d\ge 2$, and $T>0$ be fixed.
Consider any dataset $\{(\bfx_i,\bfy_i)\}_{i=1}^N\subset\R^d\times\R^d$ admissible in the sense of \eqref{eq:interp:distinct_targets}, and additionally satisfying \eqref{eq:interp:segments_disjoint}. Then there exists a (constant) control $\theta$ such that the flow map $\Phi_T^\theta$ generated by \eqref{eq:2layerANODE}, with  $p_1=2d$ and $p_2=N$, satisfies
\[
\Phi_T^\theta(\bfx_i)=\bfy_i,\qquad\text{for all }i\in[N].
\]
\end{theorem}

The proof is deferred to \Cref{sss.2layers}. 
Although \eqref{eq:2layerANODE} achieves exact \UIP without any form of time dependence, its mechanism is less efficient than that of \eqref{eq:saNODE}: it requires $p_1=2d$ and $p_2=N$, hence \(p_1\cdot p_2=2dN\) first-layer units and $p_2=N$ second-layer units. By contrast, \eqref{eq:saNODE} uses \(p=2N\) units.

\begin{remark}[Dimensional constraints]\label{rem:genericity}
Condition \eqref{eq:interp:segments_disjoint} depends strongly on the dimension $d$. Because our current construction relies on vector fields localized strictly around straight segments $[\bfx_i, \bfy_i]$, these paths must not cross. In $d \ge 3$, linear segments generically do not intersect, so this condition is easily satisfied. In $d=2$, however, straight-line collisions are frequent, making the assumption highly restrictive (see \Cref{fig:condition}).

Nevertheless, the \UIP can theoretically be established for all $d \ge 2$ without requiring \eqref{eq:interp:segments_disjoint}. Instead of moving along straight lines, particles could be routed through intermediate waypoints to avoid collisions, using localized two-layer ``connectors'' to steer the flow around obstacles. While mathematically feasible, the technical length of such a construction places it beyond the scope of this paper, and we defer it to future work.
\end{remark}

\begin{figure}[t]
    \centering
    \includegraphics[width=0.7\linewidth]{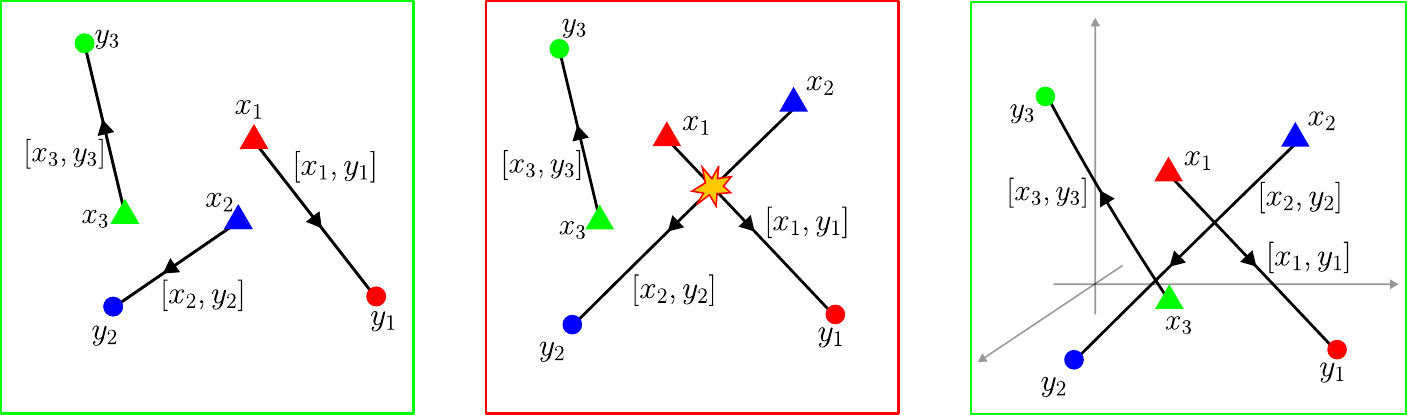}
    \caption{Illustration of the nondegeneracy condition \eqref{eq:interp:segments_disjoint}. \textbf{Left \& Right:} Admissible datasets for $d=2$ and $d=3$. \textbf{Center:} A degenerate data configuration. As discussed in \Cref{rem:genericity}, condition \eqref{eq:interp:segments_disjoint} is generic for $d \geq 3$.}
    \label{fig:condition}
\end{figure}

At this point, the distinction between point interpolation and cell routing becomes essential. While autonomous systems can map finite point clouds arbitrarily close to their targets by letting isolated trajectories bypass each other in $d \ge 2$ (\cite[Theorem 7]{AlvarezLopez2024Interplay}), routing macroscopic cells $A_k$ introduces structural obstructions. Since autonomous trajectories cannot cross, models like \eqref{eq:aNODE} or \eqref{eq:2layerANODE} face a topological barrier. For example, if the target of $A_1$ lies behind $A_2$, the flow would need to push $A_1$ directly through the space initially occupied by $A_2$, which violates uniqueness. The explicit time dependence in \eqref{eq:saNODE} circumvents this by allowing trajectories to cross safely in space-time.

Furthermore, \Cref{def:cell_ctrl} strictly requires $\Phi_T^\theta(A_k) \subset B(\bfr_k, \eta)$. A weaker $L^2$-a.e.\ control might bypass the obstruction by sacrificing small-measure subsets of the cells: failing to route these small regions to their targets would free up physical space to create ``corridors'' for other cells to pass through, potentially accommodating autonomous systems.

In summary, while autonomous systems can interpolate points---especially with added depth as in \eqref{eq:2layerANODE}---their lack of a time coordinate to schedule motion makes them structurally ill-suited for cell routing. Thus, among the architectures considered here, \eqref{eq:saNODE} provides the simplest setting where exact interpolation and generalization coexist.

\section{Nonparametric estimation}\label{s:nonparam}

In this section, we recall the construction and fundamental properties of piecewise-constant nonparametric estimators. Our analysis focuses on two canonical approaches:
\begin{enumerate}
    \item \textbf{Histogram estimators}, which partition the input domain using a fixed uniform grid.
    \item \textbf{Nearest-neighbor estimators}, which rely on Voronoi tessellations \cite{Aurenhammer1991VoronoiSurvey, BiauDevroye2015NearestNeighbor}.
\end{enumerate}
While the results in this section are standard, proofs are collected in \Cref{sec:proof_nonparam}, as the precise statements we need do not appear verbatim in the literature. We refer the interested reader to \cite{Gyorfi2002DistributionFreeRegression, Tsybakov2008Introduction} and references therein for a deeper dive into the theory of nonparametric estimators.

\subsection{Histogram estimator}\label{ss:histogram}

We first analyze the histogram approach. Fix $R>0$, a resolution $h \in (0, 2R]$ and let 
\begin{equation}\label{eq:index.set}
\mathcal{K} \coloneqq \left\{0, \;\dots,\; \left\lceil 2R/h \right\rceil - 1 \right\}^d
\end{equation}
be the corresponding multi-index set. For each $k \in \mathcal{K}$, we define the cube
\begin{equation}\label{eq:cubes}
Q_k \coloneqq   I_R\cap\prod_{j=1}^d \left[ -R + k_j h, \, -R + (k_j+1)h \right),
\end{equation}
ensuring that the collection $\{Q_k\}_{k \in \mathcal{K}}$ forms a partition of $I_R$ up to a $\mu$-null boundary set. The total number of cells is given by
\begin{equation}\label{eq:number.cells}
    K_h \coloneqq |\mathcal{K}| = \left\lceil \frac{2R}{h} \right\rceil^d \lesssim h^{-d},
\end{equation}
and the diameter of each cell satisfies $\diam(Q_k) \leq \sqrt{d}\,h$. We approximate $y(\cdot)$ by averaging its values within each cell. This leads to the definition of both a population-level and an empirical estimator:
\begin{itemize}[leftmargin=2em]
    \item The population average $y_h\colon I_R\to\R^d$ is the best piecewise-constant approximation of $y$ in $L^2(\mu)$:
\begin{equation}\label{eq:pop-avg-def}
    y_h(\bfx) \coloneqq \sum_{k \in \mathcal{K}} \bfs_k\,\1_{Q_k}(\bfx),
\end{equation}
where the coefficients $\bfs_k\in\R^d$ are defined by
\begin{equation}\label{eq:pop-avg-coeffs}
    \bfs_k \coloneqq
    \begin{cases}
    \displaystyle \frac{1}{\mu(Q_k)}\int_{Q_k} y(\bfx) \diff\mu(\bfx), \qquad & \text{if } \mu(Q_k) > 0, \\[8pt]
    0, \qquad & \text{if } \mu(Q_k) = 0.
    \end{cases}   
\end{equation}
Note that $y_h(\cdot)$ is in fact the best piecewise-constant approximation of $y(\cdot)$ in $L^2(\mu)$, as each $\bfs_k$ minimizes the local $L^2(\mu)$-error over $Q_k$.

\item The empirical average $y_{N,h}:I_R\to\R^d$, instead, is computed directly from the dataset:
\begin{equation}\label{eq:emp-avg-def}
    y_{N,h}(\bfx) \coloneqq \sum_{k \in \mathcal{K}} \bfS_k\,\mathbf{1}_{Q_k}(\bfx).
\end{equation}
Letting $N_k \coloneqq \sum_{i=1}^N \mathbf{1}_{Q_k}(\bfx_i)$ denote the number of samples falling into $Q_k$, the coefficients $\bfS_k\in\R^d$ are now defined as
\begin{equation}\label{eq:emp-avg-coeffs}
    \bfS_k \coloneqq
    \begin{cases}
    \displaystyle \frac{1}{N_k}\sum_{i : \bfx_i \in Q_k} y(\bfx_i), \qquad & \text{if } N_k \geq 1, \\[8pt]
    0,\qquad  & \text{if } N_k = 0.
    \end{cases}
\end{equation}
\end{itemize}
A schematic representation of $y_{N,h}(\cdot)$ when $d=2$ can be seen in \Cref{fig:main_nonparam} (central column). We now establish the convergence rate for the empirical histogram estimator.

\begin{proposition}\label{prop:hist-rate}
Let $\mathcal{D}_N = \{(\bfx_i, y(\bfx_i))\}_{i=1}^N$ with $\bfx_1,\dots,\bfx_N$ drawn i.i.d.\ from $\mu\in\Pac(I_R)$, and $y \in \mathscr{C}^0(I_R; \R^d)$. Then, for any $h \in (0,2R]$, the estimator $y_{N,h}(\cdot)$ defined in \eqref{eq:emp-avg-def} satisfies
\begin{equation}\label{eq:hist-rate}
\mathbb{E}_{\mathcal{D}_N}\left[\left\|y - y_{N,h}\right\|_{L^2(\mu)}^2\right]
\;\lesssim\;
\omega_y(\sqrt{d}\,h)^2
\;+\;
\frac{\|y\|_{L^\infty(I_R)}^2}{N h^d},
\end{equation}
where $\omega_y(t) \coloneqq \sup\left\{ \|y(\bfx)-y(\bfx')\| \colon \bfx,\bfx'\in I_R, \, \|\bfx-\bfx'\|\le t \right\}$ is the modulus of continuity of $y(\cdot)$.
\end{proposition}

\begin{figure}[t!]
    \centering
    \includegraphics[width=\linewidth]{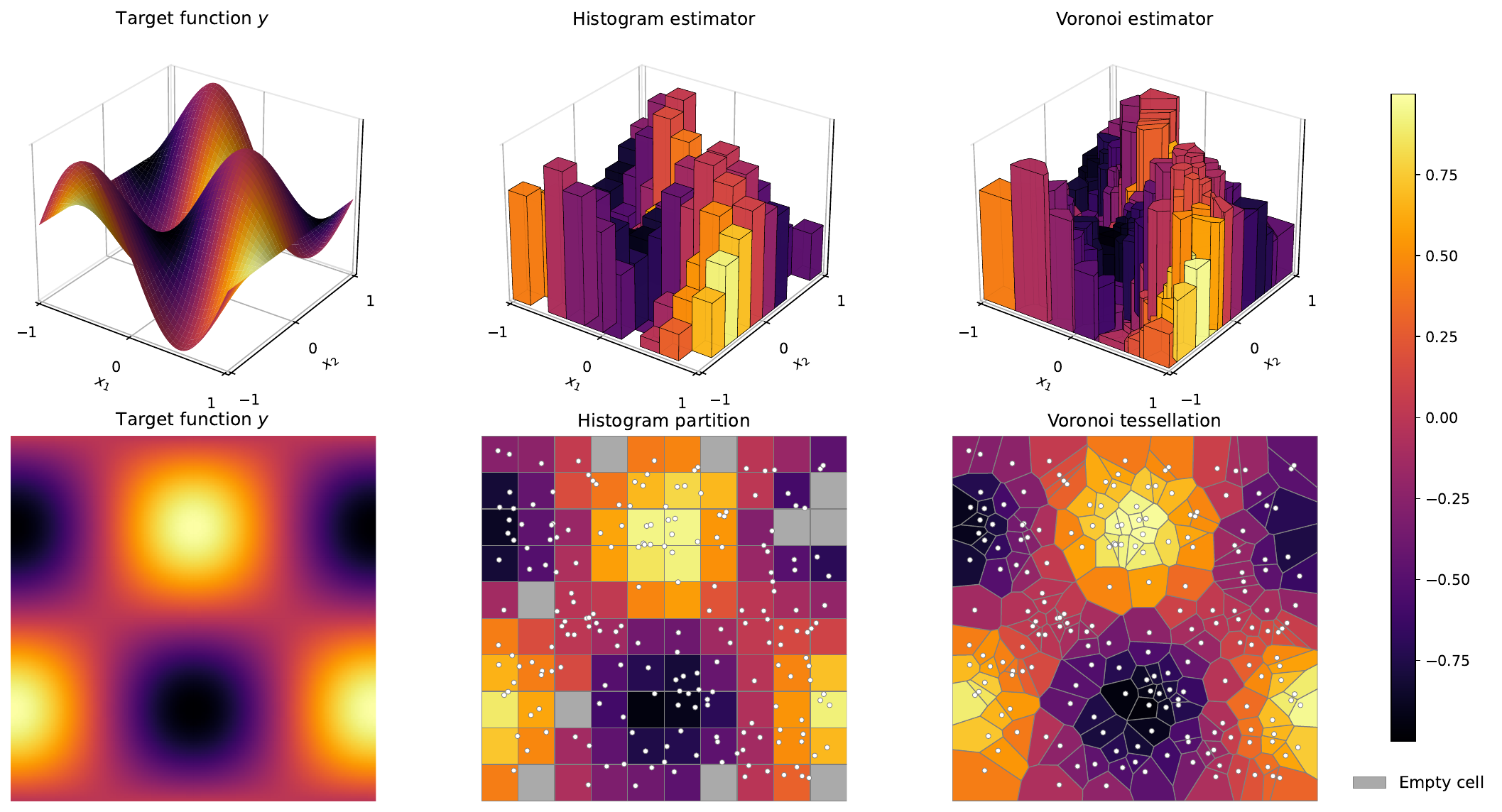}
    \caption{Histogram and Voronoi approximations of a function $y:\R^2\to\R^2$, represented in both 3D (height and color) and 2D (color only). \textbf{Left:} The target function. \textbf{Center:} The histogram estimator, which assigns each cell the average value of its enclosed training points; \textcolor{gray}{gray} cells indicate regions with no training data. \textbf{Right:} The Voronoi estimator, which assigns each cell the value of its nearest training point.}
    \label{fig:main_nonparam}
\end{figure}

The bound \eqref{eq:hist-rate} captures the standard bias-variance trade-off in nonparametric estimation \cite{Gyorfi2002DistributionFreeRegression}. The bias term $\omega_y(\sqrt{d}\,h)^2$ measures the spatial approximation error and vanishes as $h \to 0$. Conversely, the statistical variance scales as $1 / (N h^d)$; as the grid becomes finer, fewer samples fall into each cell. Driving the total error to zero therefore requires the dataset size $N$ to grow strictly faster than the number of cells.

The continuity assumption on $y$ can be relaxed at the cost of losing an explicit algebraic decay rate for the bias. If $y \in L^\infty(I_R; \R^d)$, the bound \eqref{eq:hist-rate} holds with the bias replaced by $\|y - y_h\|_{L^2(\mu)}^2$, which still vanishes as $h \to 0$ (e.g., via standard martingale convergence on nested grids \cite[Theorem~4.4.6]{Durrett2019Probability}). For $y \in L^2(\mu; \R^d)$, the bias similarly vanishes, though extending the $(Nh^d)^{-1}$ variance bound generally requires further assumptions. We do not pursue this level of generality here.

    

Conversely, by imposing stronger regularity on $y(\cdot)$ one can quantify the modulus of continuity. We focus on the Hölder case, characterized by
\begin{equation*}
\omega_y(t) \leq L_y t^\alpha \qquad \text{for some } \alpha \in (0,1] \text{ and } L_y > 0.  
\end{equation*}
Optimizing the grid resolution $h$ yields the following minimal statistical rate.

\begin{corollary}\label{cor:holder}
Suppose $y \in \mathscr{C}^{0,\alpha}(I_R; \R^d)$ for some $\alpha \in (0,1]$. Then, the estimator $y_h$ from \eqref{eq:pop-avg-def} satisfies
\begin{equation}\label{eq:riem-rate-holder}
\|y - y_h\|_{L^2(\mu)}^2 \;\lesssim\; h^{2\alpha}.
\end{equation}
Moreover, $h= (2R) \wedge (N^{-\frac{1}{2\alpha+d}})\eqqcolon h_N$ yields the minimal convergence rate in \eqref{eq:hist-rate},
\begin{equation}\label{eq:hist-rate-holder}
\mathbb{E}_{\mathcal{D}_N}\left[\left\|y - y_{N,h_N}\right\|_{L^2(\mu)}^2\right]
\;\lesssim\;
N^{-\frac{2\alpha}{2\alpha+d}}.
\end{equation}
\end{corollary}


\begin{remark}[Sample complexity]
The approximation by Riemann sums on a fixed grid of size $h$ yields \eqref{eq:riem-rate-holder}. To ensure a squared error below $\varepsilon > 0$, the number of cells must scale as $K_h \gtrsim \varepsilon^{-\frac{d}{2\alpha}}$. Consequently, the sample complexity to guarantee an expected risk below $\varepsilon$ in \eqref{eq:hist-rate-holder} scales as:
\begin{equation}\label{eq:samp.complexity}
N \;\gtrsim\; \varepsilon^{-\frac{2\alpha+d}{2\alpha}} \;=\; \varepsilon^{-1} \cdot \varepsilon^{-\frac{d}{2\alpha}}.
\end{equation}
We observe that the sample complexity naturally decomposes into the product of the Monte Carlo rate $\varepsilon^{-1}$ (associated with estimating a mean with fixed variance) and a geometric factor $\varepsilon^{-d/(2\alpha)}$. This reflects the curse of dimensionality: whereas numerical integration computes a single global average, function recovery requires estimating $K_h$ distinct local averages simultaneously to resolve the spatial structure of $y(\cdot)$ in $\R^d$.
\end{remark}

\begin{remark}[Minimax optimality]\label{rem:minimax_discussion}
In terms of the number of cells $K_h \asymp h^{-d}$, the bias estimate \eqref{eq:riem-rate-holder} yields $$\|y - y_h\|_{L^2(\mu)} \lesssim K_h^{-\alpha/d}.$$ This matches the optimal rate among all \emph{linear} approximation spaces of dimension $K_h$, as captured by the {Kolmogorov $n$-width}
\[
d_n(\mathcal{F}) \;\coloneqq\; \inf_{\dim(V)=n}\ \sup_{f\in\mathcal{F}}\ \inf_{g\in V}\|f-g\|_{L^2(\mu)}.
\]
For the isotropic $\mathscr{C}^{0,\alpha}$-Hölder ball $\mathcal{F}$, one has $d_n(\mathcal{F}) \asymp n^{-\alpha/d}$ \cite{Pinkus1985NWidths}, an order attained by uniform piecewise-constant grids such as $y_h$.  Moreover, this exponent cannot be improved even by \emph{adaptive} (nonlinear) methods with continuous parameter selection: the corresponding manifold $n$-widths exhibit the exact same decay rate on Besov (and hence Hölder) balls \cite[Eq.~(9.4)]{DeVore1998Nonlinear}. Adaptivity can only yield strictly faster rates on smaller, spatially inhomogeneous target classes \cite{Cohen2001Tree}.
\end{remark}



The rate $N^{-\frac{2\alpha}{2\alpha+d}}$ matches the minimax lower bound for nonparametric regression under additive label noise \cite{Stone1982Optimal}. When $\bfy_i = y(\bfx_i) + \text{noise}$, individual data points are unreliable, making local spatial averaging the optimal strategy to cancel out stochastic fluctuations.

In contrast, our setting is strictly \emph{noiseless} because $\bfy_i = y(\bfx_i)$. Here, spatial averaging is overly conservative and wastes the precision of perfect data. Enforcing \emph{strict interpolation} bypasses the statistical variance penalty, yielding faster rates. For instance, assuming $y\in \mathscr{C}^{k,\alpha}(I_R;\R^d)$ with $k\ge0$ integer and $\alpha \in (0,1]$, spline interpolants $m_N$ achieve \cite{Bauer2017Nonparametric}
\begin{equation}\label{eq:bauer_limit}
\|y - m_N\|_{L^2(\mu)}^2 \;\lesssim\; \|y - m_N\|_{L^\infty(I_R)}^2 \;\lesssim\; \left(\frac{\log N}{N}\right)^{\frac{2(k+\alpha)}{d}}.
\end{equation}
To exploit the lack of noise, we turn to the simplest interpolation scheme: the nearest-neighbor estimator.

\subsection{Nearest-neighbor estimator}\label{ss:voronoi}

We now partition $I_R$ by assigning each point to its nearest neighbor within $\{\bfx_i\}_{i=1}^N$. Unlike the histogram approach, where the grid is fixed and independent of the data, this partition naturally adapts to the local density of $\mathcal{D}_N$: dense regions produce smaller cells, while sparse regions yield larger ones. This adaptivity is exactly what yields faster approximation rates.

Formally, let $\{V_i\}_{i=1}^N$ denote the covering of $I_R$ by the closed Voronoi cells generated by $\{\mathbf{x}_i\}_{i=1}^N$:
\begin{equation}\label{eq:voronoi.cells}
V_i \coloneqq \{\bfx \in I_R : \|\bfx-\bfx_i\| \le \|\bfx-\bfx_j\| \text{ for all } j \ne i\}.    
\end{equation}
For any $\mathbf{x} \in I_R$, we define its nearest neighbor within the dataset by
\begin{equation}\label{eq:nearest_neighbor}
x_{\mathrm{NN}}(\bfx) \coloneqq \argmin_{\bfx_i \in \{\bfx_1, \dots, \bfx_N\}} \|\bfx - \bfx_i\|.
\end{equation}
where any ties on the cell boundaries are broken arbitrarily. Using this mapping, we define the nearest-neighbor interpolant of $y$, denoted by $y_N^V$, as
\begin{equation}\label{eq:voronoi-estimator}
    y_N^V(\bfx) \coloneqq y\left(x_{\mathrm{NN}}(\bfx)\right) = \sum_{i=1}^N y(\bfx_i)\,\1_{V_i}(\bfx).
\end{equation}
A schematic representation of $y_N^V$ for $d=2$ can be seen in the rightmost column of \Cref{fig:main_nonparam}. The accuracy of $y_N^V$ is governed purely by the covering radius of $\mathcal{D}_N$, defined by
\begin{equation}\label{eq:cov.radius}
R_N \coloneqq \sup_{\bfx \in I_R} \min_{i\in[N]} \|\bfx-\bfx_i\| = \sup_{\bfx \in I_R} \|\bfx - x_{\mathrm{NN}}(\bfx)\|.
\end{equation}
The smaller $R_N$, the closer every point in $I_R$ is to some input $\bfx_i$. For i.i.d.\ samples drawn from a density bounded below on $I_R$, the expected scaling of $R_N$ is characterized by \cite[Theorem~2.1]{ReznikovSaff2016CoveringRadius}, which yields
\begin{equation}\label{eq:cov_radius_rate}
\mathbb{E}_{\mathcal{D}_N}[R_N^q] \;\lesssim\; \left( \frac{\log N}{N} \right)^{q/d} \quad \text{for any } q > 0.
\end{equation}
The following proposition uses this rate to quantify the $L^2$-error for Hölder targets:

\begin{proposition}\label{prop:voronoi-rate}
Let $N\ge2$ and $\mathcal{D}_N = \{(\bfx_i, y(\bfx_i))\}_{i=1}^N$ with 
$\bfx_1,\dots,\bfx_N$ drawn i.i.d.\ from $\mu \in \mathcal{P}_{\mathrm{ac}}(I_R)$.
Assume that $\mu$ admits a density $\rho$ with
$\inf_{\bfx\in I_R}\rho(\bfx)>0$, and let
$y\in\mathscr{C}^{0,\alpha}(I_R;\R^d)$ for some $\alpha\in(0,1]$.
Then the nearest-neighbor interpolant $y_N^V$ defined in
\eqref{eq:voronoi-estimator} satisfies
\begin{equation}\label{eq:voronoi-rate}
\mathbb E_{\mathcal D_N}
\Big[
\|y-y_N^V\|_{L^2(\mu)}^2
\Big]
\lesssim
\left(\frac{\log N}{N}\right)^{\frac{2\alpha}{d}} .
\end{equation}
\end{proposition}
Note that the rate in \eqref{eq:voronoi-rate} is strictly faster than the corresponding rate of $O\left(N^{-\frac{2\alpha}{2\alpha+d}}\right)$ derived in \eqref{eq:hist-rate-holder}.



\section{Generalization via simultaneous cell controllability}
\label{s:generalization}

The previous two sections provide the ingredients needed for our subsequent analysis. On the one hand, \Cref{s:expressivity} shows that \eqref{eq:saNODE} can map prescribed convex cells into small target balls. On the other hand, \Cref{s:nonparam} provides nonparametric estimators with explicit statistical risks. We now combine both ingredients to derive generalization bounds.


\subsection{General bounds}

To derive quantitative rates, we assume that $\bfx_1,\dots,\bfx_N$ are drawn i.i.d.\ from $\mu$.\footnote{The deterministic core of the argument below, however, is purely geometric and applies to any fixed dataset.} This induces the expected population risk over all realizations of the training sample:
\begin{equation}\label{eq:expected_risk}
\mathbb E_{\mathcal D_N}\left[\mathcal R(\theta_N)\right]
\coloneqq
\int_{I_R}\cdots\int_{I_R}\mathcal R\left(\theta_N(\mathcal D_N)\right)\diff\mu(\bfx_1)\cdots\diff\mu(\bfx_N).
\end{equation}
Our approach proceeds in two steps. First, we construct from $\mathcal D_N$ a piecewise-constant proxy $y_N(\cdot)$ that estimates $y(\cdot)$. Second, we design a control $\theta_N$ such that $\Phi_T^{\theta_N}$ approximates $y_N$ on most of the domain.  

Let $\{P_k\}_{k=1}^K$ be a partition of $I_R$ into convex cells, let $\{\bfr_k\}_{k=1}^K\subset\R^d$ be target values computed from $\mathcal D_N$ (e.g., via cell averages), and define the piecewise-constant estimator $$y_N(\bfx)\coloneqq \sum_{k=1}^K \bfr_k\,\mathbf 1_{P_k}(\bfx).$$ Because $\Phi_T^\theta$ is a homeomorphism, it cannot uniformly approximate the jump discontinuities of $y_N$. We therefore restrict the accuracy requirements to regions safely away from the cell boundaries. For a margin $\delta > 0$, we define the trimmed cores $P_k^\delta$ and the boundary layer $\Omega_\delta$ (illustrated in \Cref{fig:trimmed}):
\begin{equation}\label{eq:trimmed-cores}
P_k^\delta \coloneqq \left\{ \bfx \in P_k \colon \dist(\bfx, \partial P_k) \geq \delta \right\},
\qquad
\Omega_\delta \coloneqq I_R \setminus \bigcup_{k=1}^K P_k^\delta.   
\end{equation}

\begin{figure}[t]
    \centering
    \includegraphics[width=\linewidth]{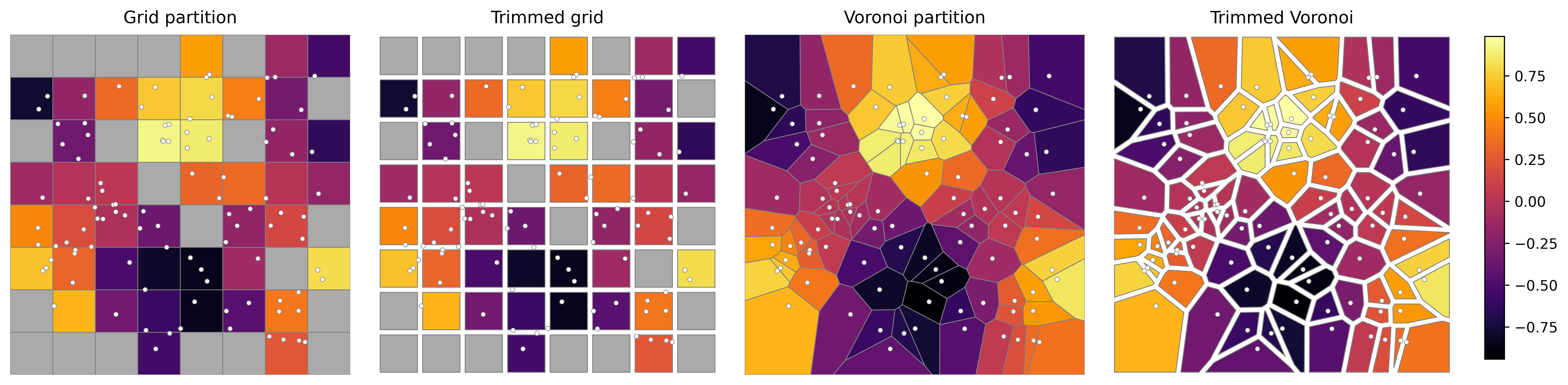}
\caption{Comparison of grid and Voronoi partitions on the same dataset of $N=100$ training points. \textbf{Left to right:} grid partition, trimmed grid partition, Voronoi partition, and trimmed Voronoi partition. In the trimmed cases, each cell is eroded inward by a margin $\delta>0$, meaning points within distance $\delta$ of the boundary are removed to produce disjoint compact cores. The white regions indicate these removed boundary layers.}
    \label{fig:trimmed}
\end{figure}
Since $y_N \equiv \bfr_k$ on each core $P_k^\delta$, any control $\theta$ satisfying condition \eqref{eq:cell-wise-eq} up to tolerance $\eta$ on these cores yields a natural decomposition of the error. This leads to the following model-agnostic bound.

\begin{theorem}[Template bound]\label{thm:template_bound}
Assume $d \geq 2$ and that the flow $\Phi_T^\theta$ of \eqref{eq:gen_model} is cell-wise controllable in the sense of \Cref{def:cell_ctrl}. Let $\mathcal{D}_N = \{(\bfx_i, y(\bfx_i))\}_{i=1}^N$ with $\bfx_1,\dots,\bfx_N$ drawn i.i.d.\ from $\mu\in\Pac(I_R)$, let $y \in \mathscr{C}^0(I_R; \R^d)$, and let $y_N$ be a piecewise-constant estimator on $\{P_k\}_{k=1}^K$. Then, for any margin $\delta > 0$ and tolerance $\eta > 0$, there exists a data-dependent control $\theta_N = \theta_N(\mathcal{D}_N)$ such that
\begin{equation}\label{eq:template_bound}
\mathbb{E}_{\mathcal{D}_N} \left[ \mathcal{R}(\theta_N) \right]
\;\lesssim\;
\mathbb{E}_{\mathcal{D}_N} \left[ \left\| y - y_N \right\|_{L^2(\mu)}^2 \right]
\;+\; \mathbb{E}_{\mathcal{D}_N} \left[ \mu(\Omega_\delta) \right]
\;+\; \eta^2,
\end{equation}
provided both $y_N$ and $\Phi_T^{\theta_N}$ remain uniformly bounded on $I_R$.
\end{theorem}

\begin{proof}
By the triangle inequality, for every control $\theta$ we have
$$\left\| \Phi_T^\theta - y \right\|_{L^2(\mu)}^2 \leq 2 \left\| y - y_N \right\|_{L^2(\mu)}^2 + 2 \left\| y_N - \Phi_T^\theta \right\|_{L^2(\mu)}^2.$$ To bound the realization error, we split the domain into the trimmed cores and the boundary layer:
\begin{align*}
\left\| y_N - \Phi_T^\theta \right\|_{L^2(\mu)}^2
&\leq \sum_{k=1}^K \mu(P_k^\delta) \sup_{\bfx\in P_k^\delta} \left\| \Phi_T^\theta(\bfx) - \bfr_k \right\|^2 
+ \int_{\Omega_\delta} \left\| y_N(\bfx) - \Phi_T^\theta(\bfx) \right\|^2 \diff\mu(\bfx) \\
&\leq \eta^2 
+ 2\mu(\Omega_\delta) \left( \|y_N\|_{L^\infty(I_R)}^2 + \|\Phi_T^\theta\|_{L^\infty(I_R)}^2 \right),
\end{align*}
where the first term exploits the condition of \Cref{def:cell_ctrl}. By the uniform boundedness of both the estimator and the flow, taking the expectation over $\mathcal{D}_N$ yields the result.
\end{proof}

The template bound \eqref{eq:template_bound} cleanly isolates the three sources of error. The statistical error depends entirely on the chosen proxy $y_N$, the control error is dictated by the network capacity ($\eta \to 0$ as width grows), and the geometric error depends on the volume of the boundary layer $\Omega_\delta$. We treat the two proxy partitions of interest separately.

\begin{itemize}[leftmargin=2em]
    \item \textbf{Uniform grid.} The boundary layer $\Omega_\delta$ is simply a union of $\delta$-strips along the fixed cubic cell faces.
    
    \item \textbf{Voronoi.} The boundary layer $\Omega_\delta$ is the $\delta$-neighborhood of the Voronoi skeleton $\Sigma_N \coloneqq \bigcup_{i=1}^N \partial V_i$. Because the cells $V_i$ are intersections of half-spaces, they are convex polytopes, ensuring that the trimmed cores $V_i^\delta$ are pairwise disjoint compact convex sets. This allows \Cref{def:cell_ctrl} to be applied directly, even though $\Sigma_N$ depends on the random spatial configuration of the sample points.
\end{itemize}

The next two lemmas establish quantitative bounds on the expected measure of $\Omega_\delta$ for each type.

\begin{lemma}\label{lem:grid-layer}
For $\mu \in \Pac(I_R)$ with density bounded above, any partition $\{Q_k\}$ as in \eqref{eq:cubes} with $h \in (0, 2R]$, and any margin $\delta \in (0, h/2)$, the boundary layer $\Omega_\delta$ satisfies
\begin{equation}\label{eq:grid_boundary}
\mu(\Omega_\delta) \;\lesssim\; \frac{\delta}{h}.
\end{equation}
\end{lemma}

\begin{lemma}\label{lem:voronoi-layer}
For any integers $d \ge 2$ and $N\ge1$, any $\mu \in \Pac(I_R)$ with density bounded above and below by positive constants, and any sufficiently small $\delta > 0$, the boundary layer $\Omega_\delta$ generated by $\{\bfx_i\}_{i=1}^N$ satisfies
\begin{equation}\label{eq:layer-explicit}
\mathbb{E}_{\mathcal{D}_N} \left[ \mu(\Omega_\delta) \right] \;\lesssim\; \delta N^{1/d}.
\end{equation}
\end{lemma}

Substituting these geometric bounds and the statistical risks from \Cref{s:nonparam} into the decomposition \eqref{eq:template_bound} yields explicit rates for any system possessing the \SCC property.

\begin{corollary}\label{cor:rates}
Under the hypotheses of \Cref{thm:template_bound}, suppose that $y \in \mathscr{C}^{0,\alpha}(I_R;\R^d)$ for some $\alpha \in (0,1]$. Furthermore, for the Voronoi estimator, assume that $\mu$ admits a strictly positive density on $I_R$. Then,
\begin{equation}\label{eq:non-param}
\mathbb{E}_{\mathcal{D}_N} \left[ \mathcal{R}(\theta_N) \right]
\;\lesssim\;
\begin{cases}
h^{2\alpha} + \dfrac{1}{N h^d} + \dfrac{\delta}{h} + \eta^2
& \qquad \text{\normalfont (Histogram)}, \\[10pt]
\left(\dfrac{\log N}{N}\right)^{2\alpha/d} + \delta N^{1/d} + \eta^2
& \qquad \text{\normalfont (Voronoi)}.
\end{cases}
\end{equation}
\end{corollary}


\subsection{Nonparametric rates with semi-autonomous neural ODEs}

We now specialize these bounds to \eqref{eq:saNODE}. By \Cref{thm:cell_ctrl}, this model satisfies the \SCC and admits realizing flows that remain uniformly bounded on $I_R$. To match the baseline rates, it suffices to balance the statistical, geometric, and control errors in \eqref{eq:template_bound} by appropriately scaling the margin $\delta$, the tolerance $\eta$, and the width $p_N$. We first state the result for the histogram partition.

\begin{proposition}[Histogram rate]\label{prop:sanode-rate}
Let $d \ge 2$, let $y \in \mathscr{C}^{0,\alpha}(I_R;\R^d)$ for some $\alpha \in (0,1]$, and let $\bfx_1,\dots,\bfx_N$ be drawn i.i.d.\ from $\mu \in \mathcal{P}_{\mathrm{ac}}(I_R)$ with bounded density. Define
\begin{equation}\label{eq:param-choices}
h_N \coloneqq N^{-\frac{1}{2\alpha+d}}, \qquad
\delta_N \coloneqq h_N^{1+2\alpha},
\end{equation}
and let $p_N$ satisfy condition \eqref{eq:p_cell_ctrl_thm} for the histogram partition with cells $\{Q_k\}$ of side $h_N$ and margin $\delta_N$. Then, for all sufficiently large $N$, there exists a data-dependent control $\theta_N$ of width $p_N$ for \eqref{eq:saNODE} such that
\[
\mathbb{E}_{\mathcal{D}_N}\left[\mathcal{R}(\theta_N)\right]
\;\lesssim\;
N^{-\frac{2\alpha}{2\alpha+d}}.
\]
\end{proposition}

Next, we establish the analogous result for the faster Voronoi partition.

\begin{proposition}[Voronoi rate]\label{prop:sanode-voronoi}
Let $d \ge 2$, let $y \in \mathscr{C}^{0,\alpha}(I_R;\R^d)$ for some $\alpha \in (0,1]$, and let $\bfx_1,\dots,\bfx_N$ be drawn i.i.d.\ from $\mu \in \mathcal{P}_{\mathrm{ac}}(I_R)$ with density bounded above and below on $I_R$. Define
\begin{equation}\label{eq:voronoi-param-choices}
\delta_N \coloneqq (\log N)^{\frac{2\alpha}{d}}\, N^{-\frac{2\alpha+1}{d}},
\end{equation}
and let $p_N$ satisfy condition \eqref{eq:p_cell_ctrl_thm} for the Voronoi partition generated by $\{\bfx_i\}_{i=1}^N$ with margin $\delta_N$. Then, for all sufficiently large $N$, there exists a data-dependent control $\theta_N$ of width $p_N$ for \eqref{eq:saNODE} such that
\[
\mathbb{E}_{\mathcal{D}_N}\left[\mathcal{R}(\theta_N)\right]
\;\lesssim\;
\left(\frac{\log N}{N}\right)^{\frac{2\alpha}{d}}.
\]
\end{proposition}

The proofs of both propositions are provided in \Cref{sec:proof_gen}. In both cases, the width $p_N$ ensures the existence of the required flow for each $N$. This required width scaling is directly tied to the geometry of the chosen partition---specifically, the minimum cell separation and maximum cell diameter. As partitions refine, these worst-case geometric constraints drive the super-polynomial growth discussed in \Cref{rem:Cgeo_cell_ctrl}.


\section{Numerical experiments}
\label{s:num}

We complement the theoretical analysis of \Cref{s:generalization} with numerical experiments on \eqref{eq:saNODE}. The first experiment investigates whether the trained model can match nonparametric estimator baselines at a network width well below the theoretical prescription. This tests the sharpness of the sufficient width scaling. The second demonstrates that time-dependence of the vector field is necessary for topologically non-trivial tasks, providing empirical support for the autonomous/non-autonomous separation discussed in \Cref{sss:control_auton}. The code to reproduce these experiments is available in the GitHub repository \href{https://github.com/ALive95/ExGen}{ExGen}.

Throughout, training points are sampled uniformly from $I_R = [-2,2]^d$, and the population risk is estimated on a fixed test set using the squared $\ell^2$-loss. All models are trained with the Adam optimizer at a fixed learning rate of $10^{-4}$. Training ends after a maximum number of epochs (specified per experiment), or earlier if the loss falls below $10^{-6}$ or fails to improve by more than $10^{-8}$ over $5 \times 10^3$ consecutive steps.

\subsection{Width scaling}

The width prescriptions derived in \Cref{s:generalization} scale explosively with the
dimension $d$, much like the storage costs of the nonparametric estimators they are
calibrated against (see also \Cref{rem:superpoly}). In practice, however, neural ODEs are inherently nonlinear function
approximators and may exploit this structure to achieve the same risk with far fewer
parameters. We investigate to what extent this is the case by fixing the dataset size $N$
and sweeping the network width $p$ to determine the minimal width at which
\eqref{eq:saNODE} matches classical nonparametric baselines.

We consider two target functions $f, g \colon \mathbb{R}^d \to \mathbb{R}^d$:
\begin{itemize}[leftmargin=*]
\item \textbf{Smooth target} ($\alpha = 1$): $f(\bfx)^{(2k-1)} = \sin(x^{(2k-1)})\cos(x^{(2k)})$ and $f(\bfx)^{(2k)}   = \cos(x^{(2k-1)})\sin(x^{(2k)})$ for $k=1,\dots,\lfloor d/2 \rfloor,$ with $f(\bfx)^{(d)} = x^{(d)}$ if $d$ is odd.
\item \textbf{Hölder-$1/2$ target} ($\alpha=1/2$): $g(\bfx)^{(i)} = \operatorname{sgn}(x^{(i)})\sqrt{|x^{(i)}|}$ for $i \in[d].$

\end{itemize}
We study two regimes: a low-dimensional setting with $d = 3$ and $N = 500$ training points, and a high-dimensional one where $d = 8$ with
$N = 5000$ training points. The larger dataset for $d = 8$ compensates for the
increased difficulty of the estimation problem in high dimension.
The width is varied over a logarithmic grid
$p \in \{5, 11, 23, 51, 109, 237, 512, 1024\}$.
Training runs for at most $3 \times 10^4$ gradient steps,
and the population risk is estimated on a fixed test set of $2 \times 10^3$ points.
For reference, the histogram and Voronoi estimator errors are computed at the same
$(N, d)$ and displayed as horizontal baselines.  
The histogram estimator uses the optimal bandwidth $h_N = N^{-1/(2\alpha+d)}$, prescribed by \Cref{cor:holder}.

In order to compare the complexity of \eqref{eq:saNODE} with that of the nonparametric
estimators, we report for each model the number of stored scalar values. For
\eqref{eq:saNODE} this coincides with the number of trainable parameters. For the
histogram estimator this is at most $Nd$ (at most $N$ cells can be occupied, each
storing a label in $\mathbb{R}^d$), and for the Voronoi estimator it is $2Nd$ ($N$
input--label pairs, each in $\mathbb{R}^d$). 

\Cref{fig:width_vs_error} reports the test risk as a function of $p$ for both targets
and both dimensions. The quantitative results are
collected in \Cref{tab_width}.

\begin{figure}[h!]
    \centering
    \includegraphics[width=0.9\linewidth]{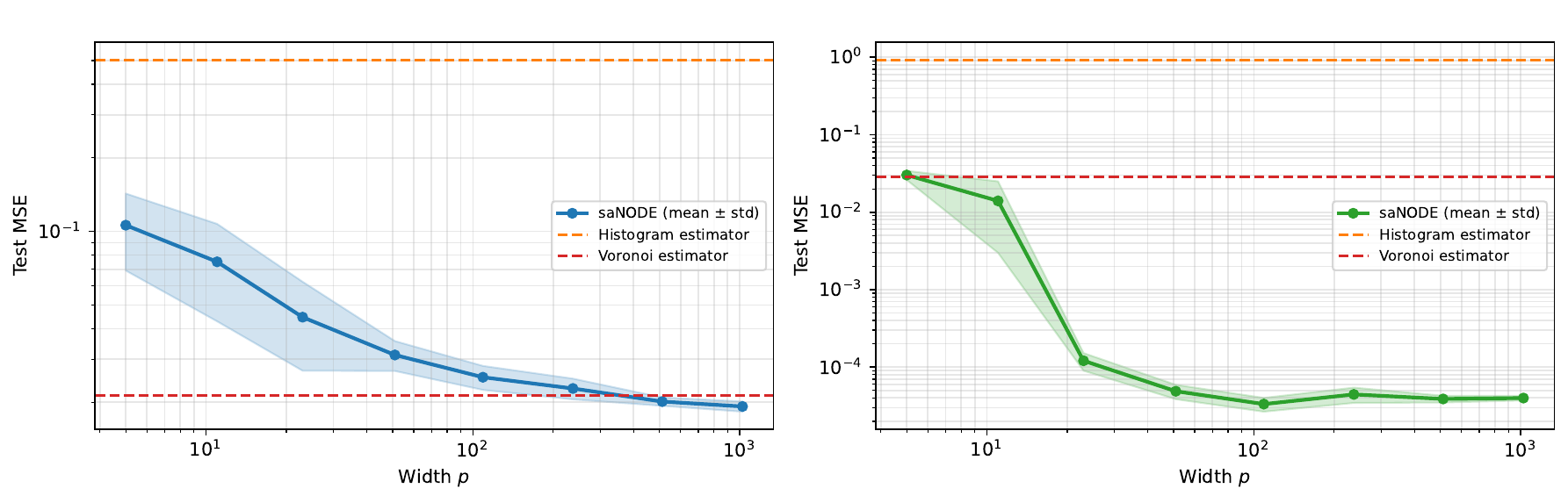}\\
    \vspace{-0.1cm}
    \includegraphics[width=0.9\linewidth]{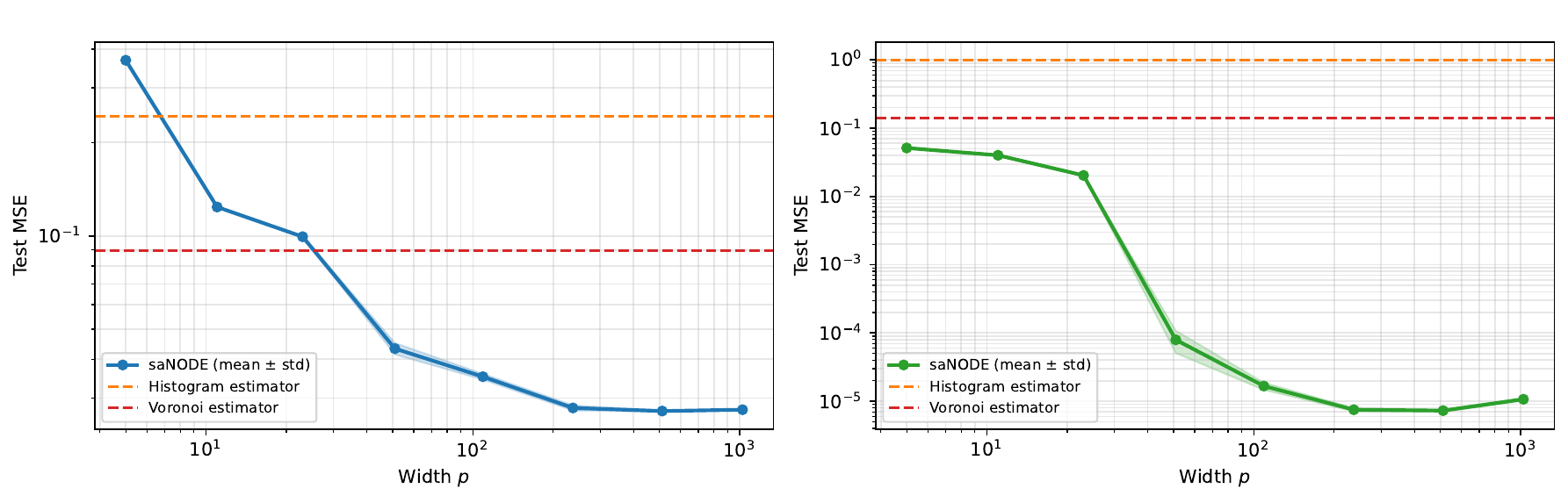}
   \caption{Test risk of \eqref{eq:saNODE} as a function of network width $p$ for the
smooth target $f$ (\textbf{left}) and the H\"older-$1/2$ target $g$ (\textbf{right}),
at $d=3$, $N=500$ (\textbf{top}) and $d=8$, $N=5000$ (\textbf{bottom}). Solid curves
show the mean over 3 independent seeds and shaded bands indicate $\pm 1$ standard
deviation. Horizontal dashed lines mark the histogram and Voronoi estimator errors.}
    \label{fig:width_vs_error}
\end{figure}

\paragraph{Case $d=3$.}
For the smooth target, the \eqref{eq:saNODE} already surpasses the histogram estimator at the
smallest tested width ($p=5$, $43$ parameters) and converges to Voronoi-level error
around $p = 512$ ($4099$ parameters), compared to the $3000$ values stored by
the Voronoi estimator. The crossover thus occurs at a parameter count comparable to
the baseline, well below the theoretical prescription.
The behavior on the H\"older-$1/2$ target is markedly different. Already at $p=23$
($187$ parameters), the \eqref{eq:saNODE} achieves a test MSE of $1.2 \times 10^{-4}$, more than
two orders of magnitude below the Voronoi baseline ($2.8 \times 10^{-2}$), and the
error plateaus around $p = 50$. The histogram estimator,
with $500$ stored values, achieves a test MSE far worse than the \eqref{eq:saNODE}
even at $p=5$. This gap suggests that the \eqref{eq:saNODE} exploits the component-wise,
sign-symmetric structure of $g$, which the histogram partition cannot capture, and that
the conservative worst-case width prescription is particularly loose for structured
targets.

\paragraph{Case $d=8$.} The \eqref{eq:saNODE} seems to do even better as the dimension increases, surpassing the histogram baseline ($2.42\times10^{-1}$)
around $p=51$ ($926$ parameters) and approaching Voronoi-level error ($8.95\times10^{-2}$)
around $p=237$ ($4274$ parameters), with almost an order of magnitude fewer parameters.

The picture for the Hölder-$1/2$ target is similar to the one in the low-dimensional setting. Despite the much higher
dimension, the \eqref{eq:saNODE} beats the histogram estimator ($9.98\times10^{-1}$) already at
$p=5$ and drops below the Voronoi baseline ($1.40\times10^{-1}$) between $p=23$ and
$p=51$, reaching a plateau of around $10^{-5}$ from $p=109$ onward, four orders of
magnitude below the Voronoi baseline. This
further confirms that the \eqref{eq:saNODE} exploits the component-wise structure of $g$ in a way
that is largely insensitive to dimension, while the classical estimators degrade rapidly.

\begin{table}[h!]
\centering
\caption{Test MSE of \eqref{eq:saNODE} versus $p$, averaged over 3 seeds. Complexity denotes the number of scalar degrees of freedom: trainable parameters for \eqref{eq:saNODE}, and $Nd$ (histogram) or $2Nd$ (Voronoi) stored scalars for the baselines.}
\label{tab_width}
\footnotesize
\setlength{\tabcolsep}{4pt}
\begin{minipage}[t]{0.49\textwidth}
\centering
\textbf{$d=3$, $N=500$}\\[2pt]
\begin{tabular}{l r r r r}
\toprule
Method & $p$ & Compl. & Smooth & H\"older-$1/2$ \\
\midrule
Histogram & --   & 1500 & $5.02\!\times\!10^{-1}$ & $9.22\!\times\!10^{-1}$ \\
Voronoi   & --   & 3000 & $2.14\!\times\!10^{-2}$ & $2.83\!\times\!10^{-2}$ \\
\midrule
\multirow{8}{*}{\eqref{eq:saNODE}}
& 5    & 43   & $1.06\!\times\!10^{-1}$ & $3.02\!\times\!10^{-2}$ \\
& 11   & 91   & $7.51\!\times\!10^{-2}$ & $1.40\!\times\!10^{-2}$ \\
& 23   & 187  & $4.46\!\times\!10^{-2}$ & $1.21\!\times\!10^{-4}$ \\
& 51   & 411  & $3.12\!\times\!10^{-2}$ & $4.88\!\times\!10^{-5}$ \\
& 109  & 875  & $2.53\!\times\!10^{-2}$ & $3.33\!\times\!10^{-5}$ \\
& 237  & 1899 & $2.28\!\times\!10^{-2}$ & $4.44\!\times\!10^{-5}$ \\
& 512  & 4099 & $2.01\!\times\!10^{-2}$ & $3.88\!\times\!10^{-5}$ \\
& 1024 & 8195 & $1.92\!\times\!10^{-2}$ & $3.99\!\times\!10^{-5}$ \\
\bottomrule
\end{tabular}
\end{minipage}
\hfill
\begin{minipage}[t]{0.49\textwidth}
\centering
\textbf{$d=8$, $N=5000$}\\[2pt]
\begin{tabular}{l r r r r}
\toprule
Method & $p$ & Compl. & Smooth & H\"older-$1/2$ \\
\midrule
Histogram & --   & 40000 & $2.42\!\times\!10^{-1}$ & $9.98\!\times\!10^{-1}$ \\
Voronoi   & --   & 80000 & $8.95\!\times\!10^{-2}$ & $1.40\!\times\!10^{-1}$ \\
\midrule
\multirow{8}{*}{\eqref{eq:saNODE}}
& 5    & 98    & $3.68\!\times\!10^{-1}$ & $5.13\!\times\!10^{-2}$ \\
& 11   & 206   & $1.24\!\times\!10^{-1}$ & $4.01\!\times\!10^{-2}$ \\
& 23   & 422   & $9.94\!\times\!10^{-2}$ & $2.04\!\times\!10^{-2}$ \\
& 51   & 926   & $4.33\!\times\!10^{-2}$ & $7.94\!\times\!10^{-5}$ \\
& 109  & 1970  & $3.52\!\times\!10^{-2}$ & $1.67\!\times\!10^{-5}$ \\
& 237  & 4274  & $2.79\!\times\!10^{-2}$ & $7.51\!\times\!10^{-6}$ \\
& 512  & 9224  & $2.72\!\times\!10^{-2}$ & $7.29\!\times\!10^{-6}$ \\
& 1024 & 18440 & $2.75\!\times\!10^{-2}$ & $1.07\!\times\!10^{-5}$ \\
\bottomrule
\end{tabular}
\end{minipage}
\end{table}

\subsection{Necessity of time-dependence}\label{ss:necessity.timedep}

As discussed in \Cref{sss:control_auton}, because macroscopic trajectories cannot cross, autonomous flows face structural obstructions to cell routing whenever the target configuration is topologically incompatible with a continuous deformation of the input. We now provide direct empirical evidence for this limitation.

We consider a checkerboard sorting problem: given a $K \times K$ partition of $[-1,1]^2$, the target assigns each point to $(+S,+S)$ if its cell indices $(i,j)$ have an even sum, and to $(-S,-S)$ otherwise, setting $S = 0.7$. For $K \geq 2$, the alternating initial regions are heavily interleaved, forcing paths to cross to reach their respective targets. Time-dependent models like \eqref{eq:saNODE} circumvent this spatial bottleneck by decoupling the motion in the extended space-time domain.

We compare \eqref{eq:saNODE} against its autonomous counterpart \eqref{eq:aNODE}. To ensure a fair comparison, the width of \eqref{eq:aNODE}  is chosen so that the total parameter count matches the \eqref{eq:saNODE}:
\[
\begin{cases}
\text{\eqref{eq:saNODE} params} = (2d+2)p + d \\    \text{\eqref{eq:aNODE} params} = (2d+1)p + d
\end{cases}
\quad \Longrightarrow \quad p_{\mathrm{aNODE}} = \left\lceil \frac{2d+2}{2d+1}\, p_{\mathrm{saNODE}} \right\rceil.
\]
Both models are trained on $N = 1000$ points with identical optimizer settings for at most $5 \times 10^4$ gradient steps, and evaluated on a separate test set of $10^3$ points. We sweep the resolution $K \in \{2, 3, 4\}$.

\begin{figure}[h!]
    \centering
    \includegraphics[width=0.42\linewidth]{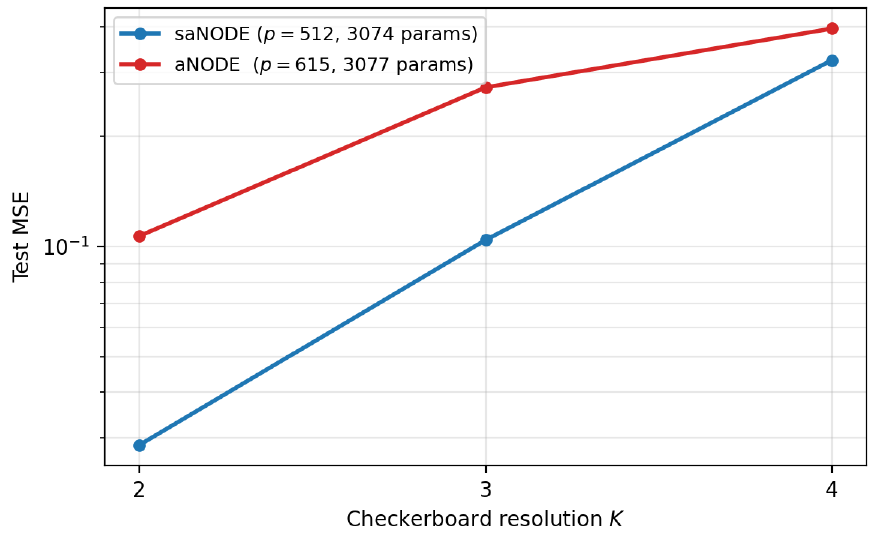}
    \caption{Test risk of \eqref{eq:saNODE} and parameter-matched \eqref{eq:aNODE} on the checkerboard sorting task with respect to $K$.}
    \label{fig:checkerboard_sweep}
\end{figure}

\Cref{fig:checkerboard_sweep} reports the test risk as a function of the  grid resolution $K$. As expected, we observe a growing gap in test error between the two models. This behavior is consistent with the theoretical obstruction: the autonomous flow cannot cleanly resolve the increasingly fine interleaving of the cells.

\begin{figure}[h!]
    \centering
\includegraphics[width=0.95\linewidth]{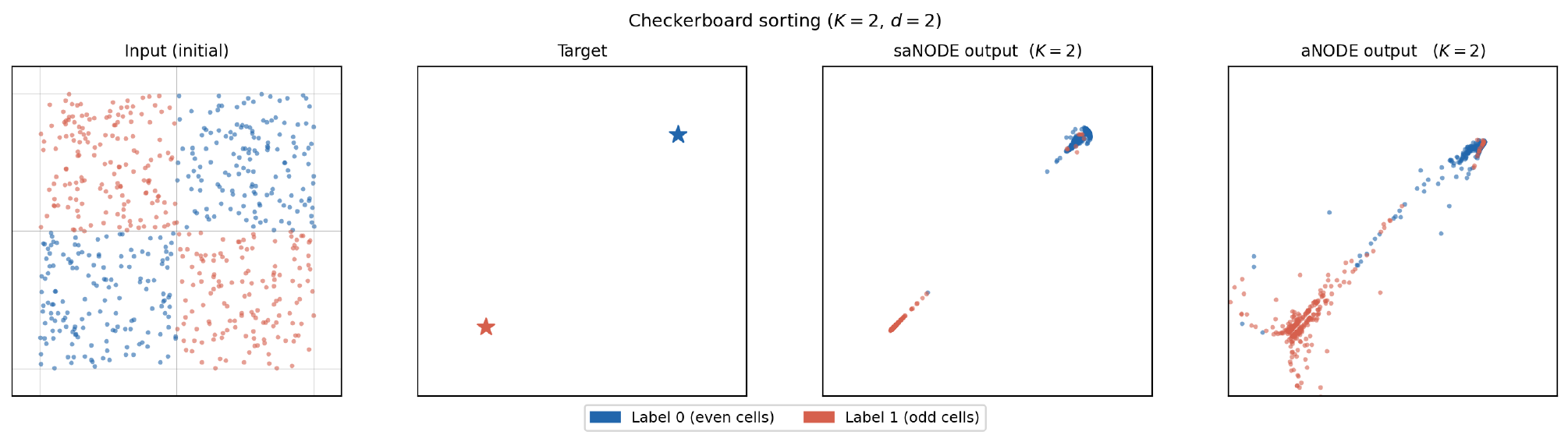}
    \includegraphics[width=0.95\linewidth]{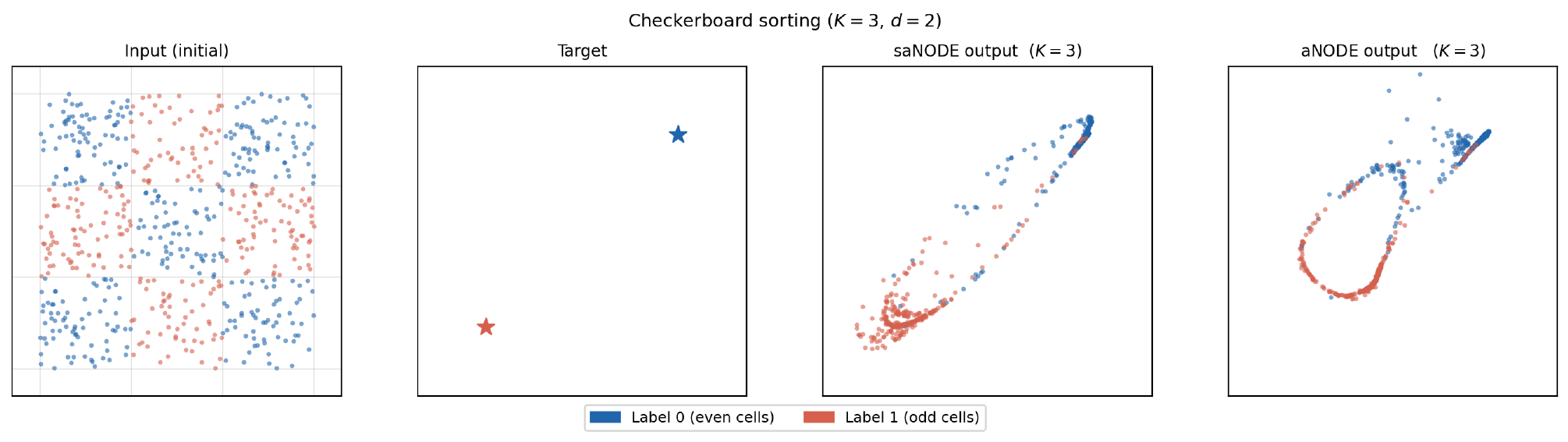}
    \includegraphics[width=0.95\linewidth]{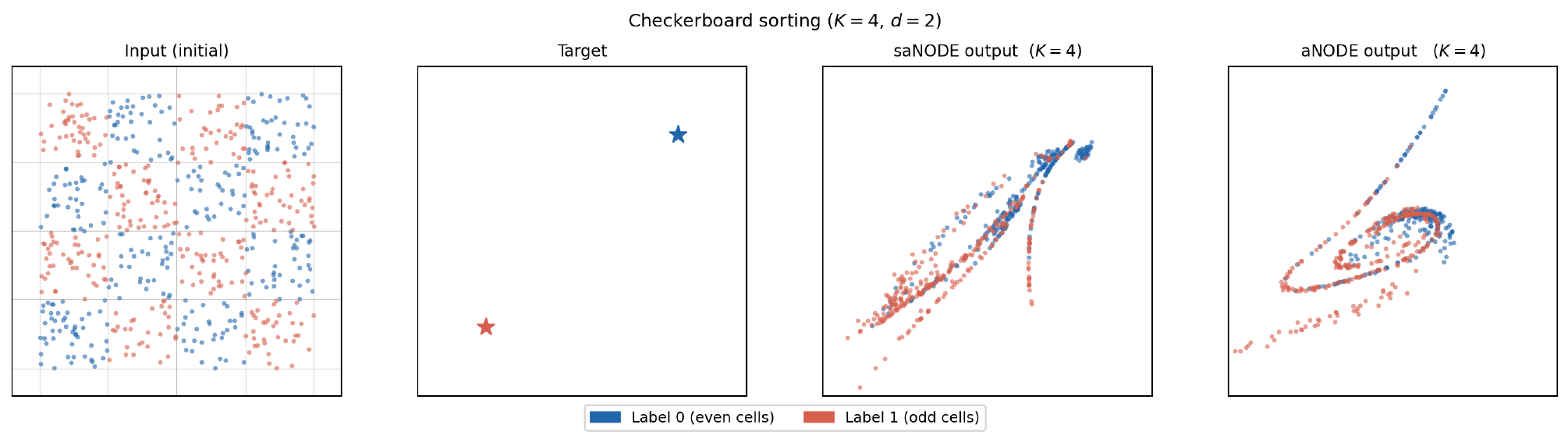}
    
    \caption{Checkerboard sorting for $d=2$ at $K=2$ (\textbf{top}), $K=3$ (\textbf{middle}), and $K=4$ (\textbf{bottom}). Each row shows, from left to right: input points, target assignment, \eqref{eq:saNODE} output, and \eqref{eq:aNODE} output. Points are colored by cell parity (even cells in \textcolor{blue}{blue}, odd cells in \textcolor{red}{red}). The \eqref{eq:saNODE} is able to separate the two classes at all resolutions, losing some precision only when $K=4$. By contrast, the \eqref{eq:aNODE} output becomes increasingly entangled as $K$ grows.}
    \label{fig:checkerboard_viz}
\end{figure}

\Cref{fig:checkerboard_viz} provides a direct visualization of the learned mappings. The \eqref{eq:saNODE} successfully routes the two classes into distinct clusters (showing only a minor loss of precision at $K=4$). By contrast, the \eqref{eq:aNODE} output remains  entangled, with points from opposite classes overlapping. As the resolution increases, this degradation becomes progressively more severe, confirming that the topological obstruction fundamentally limits autonomous models in region-routing tasks.


\section{Conclusions and perspectives}

\label{s:conclusions}

We have developed a controllability-based framework to study generalization in neural ODEs. A central insight of our work is that point interpolation, while necessary for expressivity, cannot fully explain generalization. To bridge this gap, we introduced \emph{simultaneous cell controllability}---the ability of the flow to compress entire input regions toward prescribed targets. We proved that the semi-autonomous model \eqref{eq:saNODE} satisfies this property, enabling it to approximate piecewise-constant nonparametric estimators.

This connection yields explicit population-risk bounds. By constructing a single, data-dependent flow that realizes a statistically meaningful proxy, we demonstrate that exact interpolation and quantitative generalization can coexist. Our results thus establish that overparameterized neural ODEs can emulate nonparametric procedures, providing an achievability guarantee rather than an efficient scaling law. This approach departs from standard uniform generalization bounds, which evaluate the worst-case error over an entire hypothesis class and force a rigid trade-off between expressivity and global complexity. 

Finally, we find that explicit time dependence is essential for our generalization mechanism. Without it, topological obstructions break simultaneous cell controllability, preventing autonomous architectures from routing macroscopic regions even when point interpolation is preserved.

\smallskip
We conclude by outlining natural directions opened for future work.
\paragraph{Saturation of regularity and higher-order flows.} 
The bounds established here rely on piecewise-constant (order-$0$) approximations, whose statistical bias saturates at Lipschitz continuity ($\alpha=1$). We deliberately favor this baseline because it naturally extends the exact \UIP. Nevertheless, if the ground truth is smoother ($y \in \mathscr{C}^{m,\alpha}$), replacing the order-$0$ proxy with a degree-$m$ cell-wise polynomial would yield faster statistical rates. Transferring this higher-order rate to $\Phi_T^\theta$, however, would require a complex analogue of \Cref{def:cell_ctrl} capable of realizing local higher-degree polynomials up to tolerance $\eta$. Formulating and proving such higher-order controllability remains an interesting open problem.



\paragraph{Quantitative sharpness.} 
While our proofs establish the vanishing of population risk, the required network width grows super-polynomially with $N$ due to worst-case geometric bottlenecks in the partition. Closing the gap between these theoretical thresholds and the much smaller widths observed in practice requires sharper quantitative controllability estimates. 

\paragraph{From static maps to trajectories.} 
Finally, our controllability viewpoint suggests a natural extension from static regression to dynamical-system learning. While this work uses neural ODEs to approximate fixed terminal maps, applications often involve continuous sequences. Extending simultaneous cell controllability to steer entire trajectory tubes---rather than just terminal cells---would provide a framework for trajectory-level generalization bounds in sequence modeling.

\section{Proofs}
\label{appendix}
\subsection{Proofs of \texorpdfstring{\Cref{s:expressivity}}{Section \ref*{s:expressivity}}}

\subsubsection{Proof of Theorem \ref{thm:SANODE_exact}}\label{ss:saNODE}

We show that \eqref{eq:saNODE} is simultaneously controllable, and we give an explicit construction. In particular, to steer $N$ points to prescribed targets we use at most $p=2N$ neurons. The proof requires several intermediate results.

A first key observation, that we will use throughout this section, is that
each triple of parameters $(\bfa_j,b_j,c_j)$ defines a moving hyperplane
\begin{equation}\label{eq.movinghyp}
\mathsf{H}_j(t) \coloneqq \{x \in \R^d : \bfa_j\cdot \bfx + b_j t + c_j = 0\},
\end{equation}
which translates with constant normal velocity $-b_j\bfa_j/\|\bfa_j\|^2$. 

The following preliminary lemma provides exact controllability of $d-1$ coordinates while keeping one coordinate fixed.

\begin{lemma}\label{lem:exactcontrol}
Let $N\ge1$, $d\ge2$, and let $\{(\bfx_i,\bfy_i)\}_{i=1}^N\subset\R^d\times\R^d$ be any admissible finite dataset in the sense of \eqref{eq:interp:distinct_targets} and such that $x_i^{(1)}\neq x_j^{(1)}$ for all $i\neq j$. For any $T>0$ and any $R > \max_i x_i^{(1)}$, there exists
\[
\theta_R=(\bfw_i,\bfa_i,b_i,c_i)_{i=1}^N\in(\R^d\times\R^d\times\R\times\R)^N
\]
such that the flow map $\Phi_T^{\theta_R}$ of \eqref{eq:saNODE} satisfies, for each $i\in[N]$,
\begin{align*}
\left[\Phi_T^{\theta_R}(\bfx_i)\right]^{(1)}=x_i^{(1)} \qquad \text{and} \qquad 
\left[\Phi_T^{\theta_R}(\bfx_i)\right]^{(k)}=y_i^{(k)} \text{ for } k=2,\dots,d
\end{align*}
and, for each $i\in[N]$, we have $(\bfa_i \cdot \bfx + b_i t + c_i)_+ = 0$ for all $t \ge T$ and any $\bfx \in \R^d$ such that $x^{(1)} \le R$.

Furthermore, for every $R_0>\max_i x_i^{(1)}$, the family of controls $\{\theta_R\}_{R\ge R_0}$ can be chosen so that
\begin{equation}
  \sup_{R\ge R_0}
\max_{i\in[N]}\max_{t\in[0,T]}\max_{k=2,\dots,d}
\left|\left[\Phi_t^{\theta_R}(\bfx_i)\right]^{(k)}\right|<\infty .  
\end{equation}

\end{lemma}
\begin{proof}
Up to relabeling the points, we may assume
\begin{equation}\label{eq:order}
x_1^{(1)}<x_2^{(1)}<\cdots<x_N^{(1)}.
\end{equation}
We set $\bfa_i=\bfe_1$, and $w_i^{(1)}=0$, for all $i\in[N]$.
Then \eqref{eq:saNODE} becomes
\[
\dot \bfx(t)=\sum_{i=1}^{N} \bfw_i\,(x^{(1)}(t)+b_it + c_i)_+,\qquad t\in[0,T],
\]
so, in components,
\begin{equation}\label{eq:component}
\dot x^{(1)}(t)=0,\qquad
\dot x^{(k)}(t)=\sum_{i=1}^{N} w_i^{(k)}\,(x^{(1)}(t)+b_it + c_i)_+,\qquad k=2,\dots,d.
\end{equation}
In particular, $\left[\Phi_t^{\theta_R}(\bfx_i)\right]^{(1)}=x_i^{(1)}$ for all $t\in[0,T],$ so the first coordinate is frozen and \eqref{eq:order} is preserved.

\smallskip
\noindent\textbf{Step 1.} Set
\[
b_1=\frac{x_1^{(1)}-1-R}{T},
\quad
c_1=1-x_1^{(1)},\qquad\text{and}\qquad b_i=\frac{x_{i-1}^{(1)}-R}{T},
\quad
c_i=-x_{i-1}^{(1)}\quad (i=2,\dots,N).
\]
Then each moving hyperplane $\mathsf{H}_i(t) \coloneqq \{ x \in \R^d \,:\, x^{(1)}+b_it + c_i=0\}$
satisfies:
\begin{itemize}[leftmargin=2em]
\item[(i)] $\mathsf{H}_i(t)$ is orthogonal to $\bfe_1$ for all $t$.
\item[(ii)] For $i=1$, $\mathsf{H}_1(0): x^{(1)}=x_1^{(1)}-1$ and $\mathsf{H}_1(T): x^{(1)}=R$.
\item[(iii)] For $i\ge2$, $\mathsf{H}_i(0): x^{(1)}=x_{i-1}^{(1)}$ and $\mathsf{H}_i(T): x^{(1)}=R$.
\end{itemize}
By the choice of $(b_i,c_i)$ we have $b_i T + c_i = -R$ for all $i \in [N]$. Furthermore, the hyperplanes do not intersect on $[0,T)$.
\smallskip

\noindent\textbf{Step 2.} Fix $i\in[N]$. Since $x^{(1)}(t)\equiv x_i^{(1)}$, for each $k\ge2$ we have
\[
\frac{\diff}{\diff t}\left[\Phi_t^{\theta_R}(\bfx_i)\right]^{(k)}
=\sum_{j=1}^N w_j^{(k)}\,(x_i^{(1)}+b_j t + c_j)_+.
\]
With the above choice of $(b_j,c_j)$ and \eqref{eq:order}, if $j>i$ (hence $j\ge 2$), the argument $t \mapsto x_i^{(1)}+b_j t + c_j$ is equal to $x_i^{(1)}-x_{j-1}^{(1)}\le 0$  at $t=0$, and equal to $x_i^{(1)}-R < 0$  at $t=T$. By linearity, 
\[
x_i^{(1)}+b_j t + c_j\le 0\qquad\text{for all }t\in[0,T],
\]
so neuron $j$ is never active along the trajectory from $\bfx_i$. Therefore,
\[
\left[\Phi_T^{\theta_R}(\bfx_i)\right]^{(k)}
= x_i^{(k)}+\sum_{j=1}^{i} w_j^{(k)}\int_0^T (x_i^{(1)}+b_j t + c_j)_+\diff t \eqqcolon  x_i^{(k)}+\sum_{j=1}^{i} w_j^{(k)} M_{i,j}.
\]
Conversely, if $j=1$, the argument at $t=0$ is $x_i^{(1)}+c_1=x_i^{(1)}+1-x_1^{(1)}>0,$ while for $2\le j\le i$ it is
\[
x_i^{(1)}+c_j=x_i^{(1)}-x_{j-1}^{(1)}
\ge x_i^{(1)}-x_{i-1}^{(1)}>0.
\]
Thus, for all $1\le j\le i$, the integrand is strictly positive at $t=0$ and continuous, so $M_{i,j}>0$.

\smallskip
\noindent\textbf{Step 3.} Imposing $\left[\Phi_T^{\theta_R}(\bfx_i)\right]^{(k)}=y_i^{(k)}$ for $k=2,\dots,d$ yields, for each fixed $k$,
\[
y_i^{(k)}-x_i^{(k)}=\sum_{j=1}^{i} w_j^{(k)}\,M_{i,j},\qquad i\in[N],
\]
which is a lower-triangular linear system in $(w_1^{(k)},\dots,w_N^{(k)})$ with strictly positive diagonal. Hence it has a unique solution for each $k=2,\dots,d$.

Let $t \ge T$ and $\bfx \in \R^d$ with $x^{(1)} \le R$. Since $R > \max_i x_i^{(1)}$, Step 1 yields $b_i < 0$ and $b_i T + c_i = -R$ for all $i\in[N]$. Thus, for $t \ge T$ it holds that $b_i t + c_i \le -R$, which implies  $(\bfa_i \cdot \bfx + b_i t + c_i)_+ = 0$ because
\[
\bfa_i \cdot \bfx + b_i t + c_i \le x^{(1)} - R \le 0.
\]

Finally, fix $R_0>\max_i x_i^{(1)}$. The uniform boundedness follows from the explicit triangular construction. Indeed, the entries $M_{i,j}$ are of order $R^{-1}$ as $R\to\infty$, while the corresponding weights grow at most linearly in $R$. Hence 
\[
w_j^{(k)}\int_0^t (x_i^{(1)}+b_js+c_j)_+\,\diff s
\]
remain uniformly bounded for $R\ge R_0$, uniformly in $i,j,k$ and $t\in[0,T]$. Since the same quantities depend continuously on $R$ and no $M_{i,i}$ vanishes for $R\ge R_0$, the bound extends to all $R\ge R_0$.
\end{proof}

\begin{proof}[Proof of \Cref{thm:SANODE_exact}]
Since the dataset is finite and admissible, we can apply an arbitrarily small generic rotation to ensure that
\[
x_i^{(1)}\neq x_j^{(1)} \quad\text{and}\quad y_i^{(2)}\neq y_j^{(2)}\quad\text{for }i\neq j.
\]
Up to relabeling the pairs, we assume
\begin{equation}\label{eq:order-y2}
y_1^{(2)}<y_2^{(2)}<\cdots<y_N^{(2)}.
\end{equation}
We build $\theta\in(\R^d\times\R^d\times\R\times\R)^{2N}$ by concatenating two sets of $N$ neurons which are active during the two phases given by the time intervals $[0,T/2]$ and $[T/2,T]$.

\smallskip
\noindent\textbf{Step 1.} Fix $R_0>\max_i x_i^{(1)}$. By the last assertion of \Cref{lem:exactcontrol}, applied on the interval $[0,T/2]$, there exists $C_0>0$ such that, for every $R_1\ge R_0$,
\[
\max_{i\in[N]}\max_{t\in[0,T/2]}
\left|\left[\Phi_t^{\theta_{R_1}}(\bfx_i)\right]^{(2)}\right|
\le C_0 .
\]
Choose
\begin{equation}\label{eq:m}
R_2<(-C_0)\wedge \min_{i\in[N]}y_i^{(2)}.
\end{equation}

\smallskip
\noindent\textbf{Step 2.} Now, for $\ell=N+1,\dots,2N$, set $\bfa_\ell=-\bfe_2$ and $\bfw_\ell^{(k)}=0$ for all $k=2,\dots,d$. Choose $(b_\ell,c_\ell)$ so that the hyperplanes $\mathsf{H}_\ell(t)\coloneqq \{x \in \R^d \,:\, -x^{(2)}+b_\ell t+c_\ell=0\}$ all satisfy $\mathsf{H}_\ell(T/2): x^{(2)}=R_2$. At $t=T$, they must reach the following prescribed levels: for $\ell=N+1,\dots,2N-1$, impose
\[
b_\ell\,\frac{T}{2}+c_\ell=R_2\qquad\text{and}
\qquad
b_\ell\,T+c_\ell=y^{(2)}_{\ell+1-N},
\]
and for $\ell=2N$, impose
\[
b_{2N}\,\frac{T}{2}+c_{2N}=R_2\qquad\text{and}
\qquad
b_{2N}\,T+c_{2N}=y_N^{(2)}+1.
\]

Now consider the auxiliary dynamics on $[T/2,T]$ generated only by the second block, starting from
\[
\bfz_i=(x_i^{(1)},y_i^{(2)},\dots,y_i^{(d)}).
\]
Let $\Psi_i(t)$ denote the corresponding trajectory. Since $w_\ell^{(2)} = 0$ for all $\ell \ge N+1$ we have $[\Psi_i(t)]^{(2)}\equiv y_i^{(2)}$ on this interval, and the first coordinate satisfies
\[
\frac{\diff}{\diff t}\left[\Psi_i(t)\right]^{(1)}
=\sum_{\ell=N+1}^{2N} w_\ell^{(1)}\,(-y_i^{(2)}+b_\ell t+c_\ell)_+.
\]
Using \eqref{eq:order-y2} and the fact that $b_\ell t+c_\ell\le b_\ell T+c_\ell$ for $t\in[T/2,T]$, we deduce that for $\ell=N+1,\dots,N+i-1$,
\[
b_\ell t+c_\ell \le b_\ell T+c_\ell = y^{(2)}_{\ell+1-N} \le y_i^{(2)},
\]
meaning $(-y_i^{(2)}+b_\ell t+c_\ell)_+\equiv 0$ on $[T/2,T]$. Thus, only the neurons with $\ell\ge N+i$ contribute, yielding
\[
\left[\Psi_i(T)\right]^{(1)}
=x_i^{(1)}+\sum_{j=i}^{N} w_{N+j}^{(1)}\int_{T/2}^{T}(-y_i^{(2)}+b_{N+j}t+c_{N+j})_+\diff t \eqqcolon x_i^{(1)}+\sum_{j=i}^{N} w_{N+j}^{(1)}
\widetilde M_{i,j}. 
\]
By construction, $\widetilde M_{i,j}>0$ for all $1\le i\le j\le N$. Imposing $\left[\Psi_i(T)\right]^{(1)}=y_i^{(1)}$ yields the upper-triangular linear system
\[
y_i^{(1)}-x_i^{(1)}=\sum_{j=i}^{N} w_{N+j}^{(1)}\,\widetilde M_{i,j},\qquad i\in[N],
\]
which has a unique solution since the diagonal entries $\widetilde M_{i,i}$ are strictly positive.  Let
\[
B\coloneqq 1+\max_{i\in[N]}\max_{t\in[T/2,T]}[\Psi_i(t)]^{(1)}.
\]

\smallskip
\noindent\textbf{Step 3.} Choose $R_1>\max\{R_0,B\}$ and apply \Cref{lem:exactcontrol} on $[0,T/2]$ with target level $R_1$ to construct the first block. Since $R_1\ge R_0$, the bound defining $C_0$ applies. Hence, along the first-stage trajectories,
\[
\left[\Phi_t^{\theta_{R_1}}(\bfx_i)\right]^{(2)}\ge -C_0
\qquad\text{for }t\in[0,T/2].
\]
Moreover, since $b_\ell T/2+c_\ell=R_2$, we have
$b_\ell t+c_\ell\le R_2$ for $t\in[0,T/2]$.  Therefore, by $R_2<-C_0$,
\[
-\left[\Phi_t^{\theta_{R_1}}(\bfx_i)\right]^{(2)}+b_\ell t+c_\ell
\le C_0+R_2<0,
\]
so the second block is inactive on $[0,T/2]$. Since $R_1>B$, the first block is inactive on $[T/2,T]$ along the auxiliary trajectories $\Psi_i$. By uniqueness, the full trajectories coincide with the concatenation of the first-stage trajectories and the auxiliary second-stage trajectories. Therefore
\[
\Phi_T^\theta(\bfx_i)=\bfy_i,\qquad i\in[N].
\]

\end{proof}

\begin{remark}
\Cref{thm:SANODE_exact} extends to any neural ODE of the form
\begin{equation}\label{eq:sanodegen}
\dot \bfx(t)=\sum_{j=1}^{p} \bfw_j\left(\bfa_j\cdot \bfx(t) + f_j(t) + c_j\right)_+,
\qquad t\in[0,T],
\end{equation}
provided each $f_j$ is strictly monotone. Geometrically, the hyperplane velocities then vary over time rather than remain constant. We emphasize, however, that the linear choice $f_j(t)=b_j t$ is the simplest.

For instance, one may take $f_j(t)=d_j\,\tanh(b_j t)$, where $b_j>0$ rescales time and $d_j\in\R$ bounds the hyperplane speed. In this case, the evolution splits into two phases:

\textbf{1. Control:} For small \(b_j\, t\) one has \(\tanh(b_j\, t)<1\) and the exact-control objective is attained in this phase.

\textbf{2. Stationary:} As \(\tanh(b_i\, t)\to1\), the system smoothly transitions to an autonomous regime.

\end{remark}

\subsubsection{Proof of Theorem \ref{thm:exact_2ANODE}}\label{sss.2layers}


The proof of \Cref{thm:exact_2ANODE} requires the following preliminary lemma.

\begin{lemma}\label{lem:compact_support}
Fix $d\ge 2$. Let $f_1:\R^d\to\R^d$ be of the form \eqref{eq:aNODE}, let $\bfu_1, \dots, \bfu_d \in \R^d$ be linearly independent, and let $\alpha_i < \beta_i$ for each $i \in [d]$. For any $\delta_i > 0$ such that $\alpha_i + \delta_i < \beta_i - \delta_i$, define the sets
\[
K \coloneqq \{\bfx\in\R^d \ :\ \alpha_i \le \bfu_i\cdot\bfx \le \beta_i \ \forall i\in[d]\}\quad\text{and}\quad
\tilde K \coloneqq \{\bfx\in\R^d \ :\ \alpha_i+\delta_i \le \bfu_i\cdot\bfx \le \beta_i-\delta_i \ \forall i\in[d]\}.
\]
Then, there exists $f_2:\R^d\to\R^d$ of the form \eqref{eq:2layerANODE} such that $f_2 \equiv f_1$ on $\tilde K$ and $f_2 \equiv 0$ on $\R^d\setminus K$.
\end{lemma}

\begin{proof}[Proof of \Cref{lem:compact_support}]
Let $\delta_{\min} \coloneqq \min_{i\in[d]} \delta_i > 0$. We define the barrier function $B: \R^d \to \R$ by
\[
B(\bfx) \coloneqq \sum_{i=1}^{d} \left( (\alpha_i + \delta_i - \bfu_i\cdot\bfx)_+ + (\bfu_i\cdot\bfx - (\beta_i - \delta_i))_+ \right).
\]
By definition, if $\bfx \in \tilde K$, all arguments inside the ReLUs are non-positive, yielding $B(\bfx) = 0$. Conversely, if $\bfx \notin K$, then for some index $j$, either $\bfu_j\cdot\bfx < \alpha_j$ or $\bfu_j\cdot\bfx > \beta_j$. In either case, the corresponding ReLU term is strictly greater than $\delta_j$, ensuring that $B(\bfx) > \delta_{\min} > 0$ on $\R^d \setminus K$.

Let $f_1(\bfx)=\sum_{m=1}^{p_2} \bfv_m L_m(\bfx)$, where $L_m(\bfx)\coloneqq (\bfc_m\cdot \bfx+d_m)_+$ and $\bfv_m \in \R^d$. 
Since $\bfu_1, \dots, \bfu_d$ span $\R^d$, then $B(\bfx)$ grows linearly as $\|\bfx\| \to \infty$. Because each $L_m$ has at most linear growth, the ratio $L_m/B$ is bounded on $\R^d \setminus K$, allowing us to choose $\kappa > 0$ such that
\[
L_m(\bfx) \le \kappa B(\bfx) \quad \text{for all } \bfx \notin K \text{ and } m \in [p_2].
\]
Finally, define
\[
f_2(\bfx) \coloneqq \sum_{m=1}^{p_2} \bfv_m\,\left(L_m(\bfx)-\kappa B(\bfx)\right)_+.
\]
If $\bfx\in\tilde K$, then $B(\bfx)=0$, which implies $f_2(\bfx) = \sum_m \bfv_m L_m(\bfx) = f_1(\bfx)$. 
If $\bfx\notin K$, the bound $L_m(\bfx)\le \kappa B(\bfx)$ guarantees that every ReLU is inactive, yielding $f_2(\bfx)=0$. Because $B(\bfx)$ is a linear combination of ReLUs of affine functions, $f_2$ is representable in the two-layer form \eqref{eq:2layerANODE}.
\end{proof}

\begin{remark}
This result extends immediately to any pair of strictly nested convex polytopes: one replaces the defining inequalities of $K$ in the construction of $B$ with the affine constraints defining its facets.
\end{remark}

We are now ready to prove \Cref{thm:exact_2ANODE}.

\begin{proof}[Proof of \Cref{thm:exact_2ANODE}]
Fix $T>0$ and an admissible dataset $\{(\bfx_i,\bfy_i)\}_{i=1}^N$ satisfying \eqref{eq:interp:segments_disjoint}. Without loss of generality, we may assume $\bfx_i\neq\bfy_i$ for all $i$. Let $\bfu_i \coloneqq (\bfy_i-\bfx_i)/T$ for each $i \in [N]$. The curve $\gamma_i(t) \coloneqq \bfx_i + t\bfu_i$ for $t \in [0,T]$ traces exactly the segment $S_i \coloneqq [\bfx_i,\bfy_i]$. 
Since all $S_i$ are compact and pairwise disjoint, the minimum distance between them is positive:
\[
\delta \coloneqq \min_{i\neq j} \dist(S_i,S_j) > 0.
\]
We construct disjoint neighborhoods around each $S_i$ as follows. For each $i$, fix an orthonormal basis $(\bfe_{i,1},\dots,\bfe_{i,d})$ with $\bfe_{i,1}=\tfrac{\bfy_i-\bfx_i}{\|\bfy_i-\bfx_i\|}$. We define for each $\rho>0$,
\[
K_{i,\rho} \coloneqq \left\{\bfx\in\R^d : -\rho\le (\bfx-\bfx_i)\cdot \bfe_{i,1}\le \|\bfy_i-\bfx_i\|+ \rho \quad\text{and}\quad |(\bfx-\bfx_i)\cdot \bfe_{i,k}|\le \rho \quad\text{for all } k\ge2\right\}.
\]
Set $r \coloneqq \delta/(4\sqrt{d})$. By definition, every $\bfx\in S_i$ satisfies
\[
0\le (\bfx-\bfx_i)\cdot \bfe_{i,1}\le \|\bfy_i-\bfx_i\|,
\qquad
(\bfx-\bfx_i)\cdot \bfe_{i,k}=0 \quad\text{for all } k\ge2.
\]
Hence $S_i \subset K_{i,r/2} \subset \mathrm{int}(K_{i,r})$. Moreover, the maximum distance from any $\bfx \in K_{i,r}$ to $S_i$ is $\leq \sqrt{r^2 + (d-1)r^2} = r\sqrt{d} = \delta/4$, so $K_{i,r}$ lies entirely within the $\delta/4$-neighborhood of $S_i$. 

By the triangle inequality, for any $i \neq j$,
\[
\dist(K_{i,r},K_{j,r}) \ge \dist(S_i,S_j) - 2(\delta/4) \ge \delta - \delta/2 = \delta/2 > 0.
\]
Hence, $K_{1,r},\dots,K_{N,r}$ are pairwise disjoint.

We now construct the vector field. For each $i \in [N]$, we apply \Cref{lem:compact_support} to the constant field $g_i(\bfx)\equiv \bfu_i$---which is trivially of the form \eqref{eq:aNODE}---setting the inner region to $K_{i,r/2}$ and the outer region to $K_{i,r}$. This yields a localized field $v_i$ of the form \eqref{eq:2layerANODE} such that
\[
v_i(\bfx)=\bfu_i \quad\text{for all } \bfx\in K_{i,r/2},
\qquad
v_i(\bfx)=0 \quad\text{for all } \bfx\in \R^d\setminus K_{i,r}.
\]
Define the global vector field $v_\theta(\bfx)\coloneqq \sum_{i=1}^N v_i(\bfx)$.
Since the neighborhoods $K_{1,r}, \dots, K_{N,r}$ are pairwise disjoint, for any fixed $i \in [N]$ we have
\[
v_\theta(\bfx) = v_i(\bfx) + \sum_{j\neq i} \mathbf{0} = \bfu_i \qquad \text{for all } \bfx\in K_{i,r/2}.
\]
Recall that the trace of the trajectory is $\gamma_i([0,T]) = S_i \subset K_{i,r/2}$. Therefore, the ODE dynamics exactly follow the constant velocity: $\dot\gamma_i(t) = v_\theta(\gamma_i(t)) = \bfu_i$. Integrating this yields $\Phi_T^\theta(\bfx_i) = \gamma_i(T) = \bfy_i$. Since this holds for all $i\in[N]$, we deduce that $\Phi_T^\theta$ perfectly interpolates the dataset.

Finally, we determine the architecture. By the explicit construction in \Cref{lem:compact_support}, since $g_i(\bfx) \equiv \bfu_i$ relies on a single constant term, the localized field takes the form
\[
v_i(\bfx) = \bfu_i \left( 1 - \kappa_i B_i(\bfx) \right)_+ ,
\]
where $B_i(\bfx)$ is the barrier function. Because $K_{i,r}$ is defined by bounding the $d$ projections onto the basis $\bfe_{i,k}$, the barrier $B_i$ is a sum of exactly $2d$ ReLU terms (one for each upper and lower bound). Consequently, the global field $v_\theta$ exactly matches the two-layer architecture \eqref{eq:2layerANODE} with $p_1=2d$ and $p_2=N$.
\end{proof}

\subsubsection{Proof of Theorem \ref{thm:cell_ctrl}}\label{ss:cell.ctrl}

We first state some auxiliary lemmas. The first establishes that we can construct a smooth, compactly supported, time-varying vector field that steers each $A_k$ to its target ball. We then give an explicit estimate of its Barron norm as a function of $\eta$ and of the geometry of the sets and connecting trajectories.  To this end, for a compactly supported vector field $v\in \mathscr C_c^\infty(\R\times\R^d;\R^d)$, we define the Barron norm 
as
\begin{equation}\label{eq:barron.norm}
\|v\|_{\mathcal B_2^{d+1}} \coloneqq \sum_{\ell=1}^d \int_{\R\times\R^d} \|(\tau,\omega)\|_1^2\,|\widehat{v^{(\ell)}}(\tau,\omega)|\diff\tau\diff\omega,
\end{equation}
where $\hat{v}$ is the Fourier transform of $v$.
This definition and the subsequent lemma build upon the framework of \cite[Theorem 2]{Klusowski2018Approximation}, which we adapt to the setting of time-dependent vector fields.

\begin{figure}[t!]
    \centering
    \includegraphics[width=0.3\linewidth]{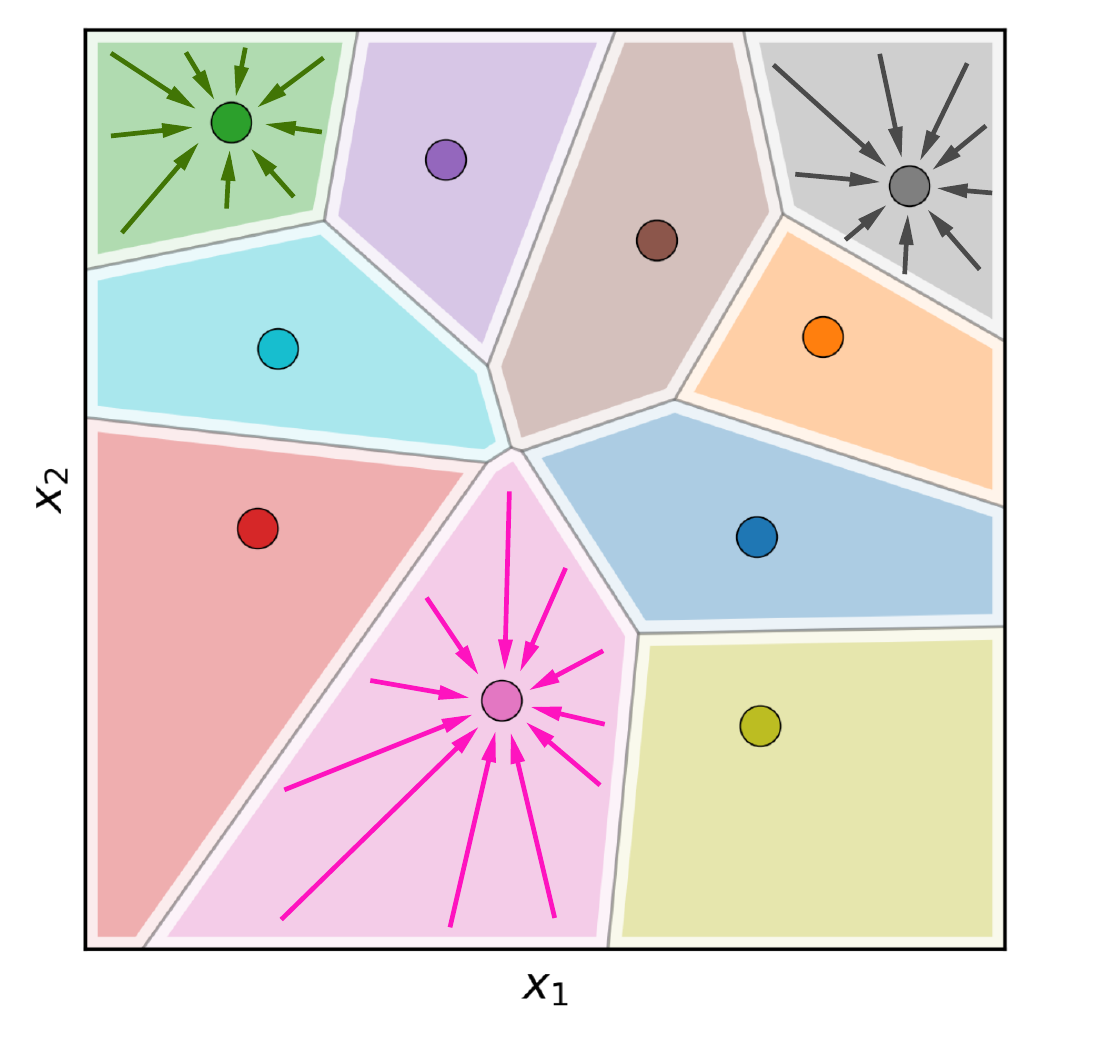}
    \qquad \qquad 
    \includegraphics[width=0.3\linewidth]{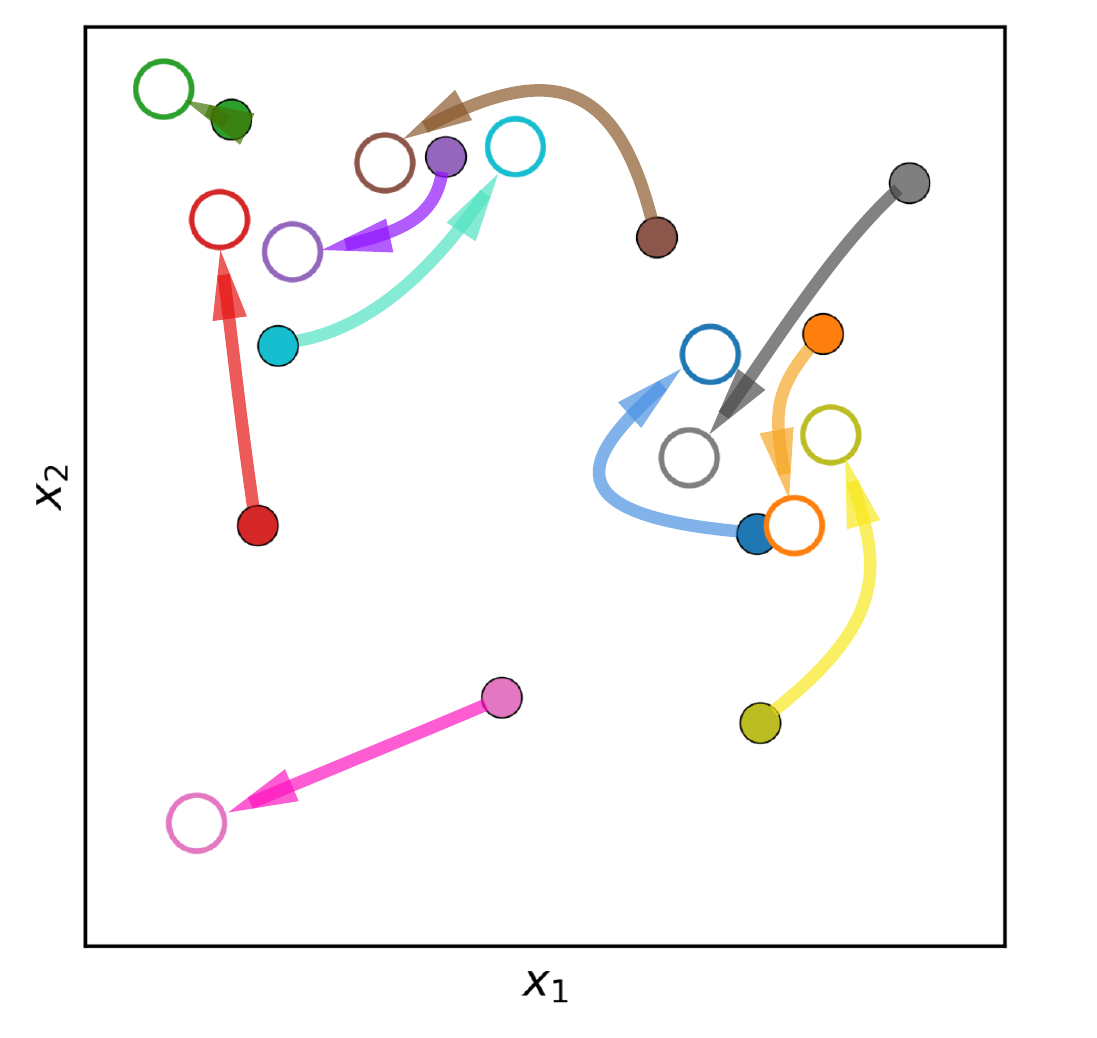}
    
    \caption{Two-step construction of the steering vector field. \textbf{Left:} Step 1 shows the initial pairwise disjoint convex sets $A_k$. The flow exponentially contracts these sets into small balls of radius $\delta$ centered around $\bfx_k$. \textbf{Right:} Step 2 depicts the rigid transport of these contracted balls along smooth, collision-free paths $\gamma_k(t)$ towards the target balls.}
    \label{fig:steering_corridors}
\end{figure}

\begin{lemma}[Steering vector field]\label{lem:steering-corridors}
Let $d\ge2$, $K\ge2$ and $T>0$. Let $\{A_k\}_{k=1}^K$ be a family of
pairwise disjoint compact convex subsets of $\R^d$, and let
$\{\bfr_k\}_{k=1}^K\subset\R^d$ be pairwise distinct. Then, for every $\eta>0$, there exists a vector field $v\in \mathscr C_c^\infty(\R\times\R^d;\R^d)$ whose flow map $\Phi_T:\R^d\to\R^d$ from time $0$ to time $T$ satisfies
\begin{equation}\label{eq:final-inclusion-claim}
 \Phi_T(A_k)\subset B(\bfr_k,\eta)
\qquad\text{for all }k\in[K].   
\end{equation}
\end{lemma}

\begin{proof}[Proof of \Cref{lem:steering-corridors}]
We first assume that the target points $\bfr_1,\dots,\bfr_K$ are pairwise
distinct. Set
\begin{equation}\label{eq:steering-geom-data}
D_*\coloneqq \max_{k\in[K]}\diam(A_k),
\qquad
s_*\coloneqq \min_{i\neq j}\dist(A_i,A_j)>0.
\end{equation}
By definition of $s_*$, the open neighborhoods 
\begin{equation}\label{eq:Vk-def}
V_k\coloneqq\left\{\bfx\in\R^d:\dist(\bfx,A_k)<\frac{s_*}{8}\right\}=A_k+\left(\frac{s_*}{8}\right) B(0,1)
\end{equation}
are pairwise disjoint. Since $d\ge2$, we may choose arbitrary points $\bfx_k\in A_k$ and connect each $\bfx_k$ to its target $\bfr_k$ via pairwise disjoint polygonal arcs. This is a standard topological consequence of working in dimensions $d \ge 2$, where finitely many paths can avoid intersections by introducing small perturbations. By smoothing the corners of these arcs and parameterizing them appropriately over time, we obtain smooth, pairwise disjoint curves $\gamma_k:[0,T]\to\R^d$ such that
\begin{equation}\label{eq:traj}
 \gamma_k(0)=\bfx_k,\quad \gamma_k\text{ is constant on }[0, 2T/3],\quad \gamma_k(T)=\bfr_k   
\end{equation}
and, by ensuring the polygonal arcs are contained within a reasonably tight bounding ball before smoothing, we also satisfy
\begin{equation}\label{eq:bd.traj}
\max_{t\in[0,T]}\|\gamma_k(t)\|< 1+\max_{k\in[K]}\|\mathbf{r}_k\|+\max_{k\in[K]}\max_{\bfx\in A_k}\|\bfx\|.    
\end{equation}

Our proof is constructive. In particular, the final vector field will be the combination of a \emph{compressing} vector field, collapsing each $A_k$ to a small ball, and a \emph{transporting} vector field, driving these balls to the targets $\bfr_k$. \Cref{fig:steering_corridors} summarizes these steps visually. Accordingly, we divide this proof into three parts.

\paragraph{Compression.} We fix a standard nonnegative radial mollifier $\psi \in \mathscr{C}_c^\infty(\R^d)$, with support in $B(0,1)$ and $\int\psi=1$, and set
\begin{equation}\label{eq:scaled-mollifier}
\psi_{s_*/32}(\bfx)\coloneqq \left(\frac{32}{s_*}\right)^{d}\psi\left(\frac{32}{s_*}\bfx\right).
\end{equation}
For each $k\in[K]$, define the smooth cutoff
\begin{equation}\label{eq:chi-k-def}
\chi_k\coloneqq \mathbf 1_{A_k+(s_*/32) B(0,1)} * \psi_{s_*/32}
\end{equation}
so that $\chi_k\in \mathscr C_c^\infty(\R^d;[0,1])$ with $\chi_k\equiv 1$ on $A_k$ and $\supp(\chi_k)\subset A_k+(s_*/16)B(0,1) \subset V_k$. Since the supports of $\chi_k$ are pairwise disjoint, the vector field
\begin{equation}\label{eq:u-def}
u(\bfx)\coloneqq\sum_{k=1}^K \chi_k(\bfx)(\bfx_k-\bfx)
\end{equation}
is well-defined and satisfies $u\in \mathscr C_c^\infty(\R^d;\R^d)$. On each $A_k$, \eqref{eq:u-def} simplifies exactly to $u(\bfx)=\bfx_k-\bfx$. 


\paragraph{Transport.} Let $m$ be defined as 
\[m\coloneqq \min_{t\in[0,T]}\min_{i\neq j}\|\gamma_i(t)-\gamma_j(t)\|,\] and denote 
\begin{equation}\label{eq:lambdadelta-def}
 \delta\coloneqq\min\left\{\eta,\frac{s_*}{8},\frac{m}{8}\right\},\qquad \lambda\coloneqq\frac{3}{T}\log_+\left(\frac{D_*}{\delta}\right)\ge0.
\end{equation}
Furthermore, let $\xi\in\mathscr C_c^\infty(\R^d;[0,1])$ be a smooth radial cutoff with $\xi\equiv 1$ on $B(0,1)$ and $\supp \xi\subset B(0,2)$, and define
\begin{equation}\label{eq:vT-def}
w(t,\bfx)\coloneqq \sum_{k=1}^K \dot{\gamma}_k(t)\,
\xi\left(\frac{8(\bfx-\gamma_k(t))}{m}\right),
\qquad (t,\bfx)\in [0,T]\times\R^d .
\end{equation}
For each $t\in[0,T]$, the spatial support of the $k$-th summand is contained in $B(\gamma_k(t),m/4)$, and these balls are pairwise disjoint. Moreover, $w(t,\bfx)=\dot{\gamma}_k(t)$ whenever $\|\bfx-\gamma_k(t)\|\le \frac{m}{8}.$

\paragraph{Concatenation.} 
To concatenate compression and transport, choose a nondecreasing $\varphi\in\mathscr{C}^\infty([0,T])$ with $\varphi(t)=0$ on $[0,T/3]$ and $\varphi(t)=1$ on $[2T/3,T]$, and define the time-dependent vector field
\begin{equation}\label{eq:v-tilde-def}
v_{[0,T]}(t,\bfx)\coloneqq (1-\varphi(t))\,\lambda u(\bfx)+\varphi(t)\,w(t,\bfx),
\qquad (t,\bfx)\in [0,T]\times\R^d .
\end{equation}
Then $v_{[0,T]}\in \mathscr C_c^\infty([0,T]\times\R^d;\R^d)$, and thus \eqref{eq:v-tilde-def} generates a unique flow $\Phi_t$ on $[0,T]$.

We verify that $\Phi_t$ satisfies the condition \eqref{eq:final-inclusion-claim}. 
Since $\gamma_k$ is constant on $[0,2T/3]$, we have $w\equiv 0$ on $[0,2T/3]\times\R^d$, and therefore \eqref{eq:v-tilde-def} reduces there to $v_{[0,T]}(t,\bfx)=\lambda(1-\varphi(t))u(\bfx).$ 
Besides, solving $\dot\bfx(t)=\lambda(1-\varphi(t))u(\bfx)$, we see that the trajectory of any initial point $\bfx(0)\in A_k$ satisfies
\begin{equation}\label{eq:compression.traj}
\Phi_t(\bfx(0))=\bfx_k+\exp\left(-\lambda\int_0^t (1-\varphi(s))\diff s\right)\left(\bfx(0)-\bfx_k\right)\in A_k \qquad \text{for } t\in[0,2T/3].    
\end{equation}
Since $\varphi(s)=0$ on $[0,T/3]$ and $\varphi\le 1$ everywhere, the integral satisfies $\int_0^{2T/3}(1-\varphi(s))\diff s \ge T/3$. Recalling our choice of $\lambda$ in \eqref{eq:lambdadelta-def}, this guarantees that at $t=2T/3$,
$$
\|\Phi_{2T/3}(\bfx(0))-\bfx_k\| \le e^{-\lambda(T/3)}\|\bfx(0)-\bfx_k\| \le e^{-\lambda(T/3)}D_* \le \delta.
$$
Consequently, we obtain
\begin{equation}\label{eq:state-at-2T3}
\Phi_{2T/3}(A_k)\subset B(\bfx_k,\delta).
\end{equation}
On $[2T/3,T]$, we have $\varphi\equiv 1$, so $v_{[0,T]}=w$. Let
$z_0\in B(\bfx_k,\delta)$, and write $z_0=\bfx_k+\bfy=\gamma_k(2T/3)+\bfy$ with
$\|\bfy\|\le\delta\le m/8$. The path $z(t)\coloneqq\gamma_k(t)+\bfy$ stays in the
region where $w(t,\cdot)=\dot\gamma_k(t)$, hence solves the ODE; by uniqueness, the
flow starting from $z_0$ equals $z(t)$, and in particular reaches $\bfr_k+\bfy$ at
$t=T$. Combined with \eqref{eq:state-at-2T3}, this gives
\[\Phi_T(A_k)\subset B(\bfr_k,\delta)\subset B(\bfr_k,\eta).\]
To conclude, by a standard smooth extension-and-cutoff argument in the time variable, there exists a global field
\begin{equation}\label{eq:global.field}
 v\in \mathscr C_c^\infty(\R\times\R^d;\R^d)\qquad\text{such that }v=v_{[0,T]}
\qquad\text{on }[0,T]\times\R^d,   
\end{equation}
Since $v=v_{[0,T]}$ on $[0,T]\times\R^d$, the flow of $v$ on $[0,T]$ coincides with $\Phi_t$ as above. In particular,~\eqref{eq:final-inclusion-claim}~holds.

\end{proof}

\begin{lemma}[Quantitative estimates on the steering field]\label{lem:steering-estimates}

Let \(v=v_\eta\) be the vector field constructed in the proof of \Cref{lem:steering-corridors} with tolerance \(\eta\) when $\{\bfr_k\}_{k=1}^K\subset\R^d$ are pairwise distinct (see \eqref{eq:global.field}) and let $\Phi_t$ be its flow on $[0,T]$. Define  the minimum path separation $m$ and the maximum path derivative $G$ associated with the trajectories $\gamma_k$ as 
\begin{equation}\label{eq:steering-mGxi}
m\coloneqq \min_{t\in[0,T]}\min_{i\neq j}\|\gamma_i(t)-\gamma_j(t)\|,
\qquad
G\coloneqq 1+\max_{1\le j\le d+5}\max_{k\in[K]}\left\|\frac{\diff^j \gamma_k}{\diff t^j}\right\|_{L^\infty(0,T)},
\end{equation}
and let $\delta$ be as in \eqref{eq:lambdadelta-def}. Then there exists a constant $C_{d,T}>0$, independent of $\eta$ and $(A_k,\gamma_k)_k$, such that
\begin{align} \label{eq:lip.v.bound}
\sup_{t\in[0,T]}\Lip_\bfx(v(t,\cdot))
\le
C_{d,T}\left[
\frac{G}{m}+\frac{1+D_*/s_*}{T}\log_+\left(\frac{D_*}{\delta}\right)
\right].
\end{align}
Moreover, the Barron norm of $v$, as defined in \eqref{eq:barron.norm}, satisfies
\begin{align} \label{eq:steering-barron-bound}
\|v\|_{\mathcal B_2^{d+1}}
\le
C_{d,T}\,K\left[
T G^{d+5}(m^d+m^{-4})
+
\frac{(D_*+s_*)^{d+1}(1+s_*^{-(d+4)})}{T} \log_+\left(\frac{D_*}{\delta}\right)
\right].
\end{align}
Finally, for $R>0$ large enough so that $[-R,R]^d\supset \bigcup_{k\in[K]}A_k$, we also have:
\begin{align} \label{eq:Linftybd}
\max_{(t,\bfx)\in[0,T]\times[-R,R]^d}\|\Phi_t(\bfx)\|
\le1+\max_{k\in[K]}\|\bfr_k\|+\frac32\sqrt{d}\,R.
\end{align}
\end{lemma}

\begin{proof}
Throughout the proof, $C_d$ and $C_{d,T}$ denote positive constants that may change from line to line and depend only on $d$, $T$, and on the fixed smooth profiles $\psi$, $\xi$, and $\varphi$.

\paragraph{Lipschitz estimate \eqref{eq:lip.v.bound}.} Using the decomposition \eqref{eq:v-tilde-def} and the bound $0\le \varphi\le 1$, we obtain
\begin{equation}\label{eq:Lip.decomp}
\sup_{t\in[0,T]}\Lip_\bfx(v(t,\cdot))
\le
\lambda \|Du\|_{L^\infty(\R^d)}   +
\sup_{t\in[0,T]}\|D_\bfx w(t,\cdot)\|_{L^\infty(\R^d)}. 
\end{equation}
Recall the definition of $u$ in \eqref{eq:u-def}. Writing $\chi_k=\mathbf 1_{E_k}*\psi_{s_*/32}$ with $E_k\coloneqq A_k+(s_*/32)B(0,1)$, by the classical Young's convolution inequality, we have 
\begin{equation}
\label{temp}
\|\nabla\chi_k\|_{L^\infty(\R^d)} = \|\mathbf{1}_{E_k} * \nabla\psi_{s_*/32}\|_{L^\infty(\R^d)}\le\|\mathbf{1}_{E_k}\|_{L^\infty}\|\nabla\psi_{s_*/32}\|_{L^1} \le C_d\,s_*^{-1}.
\end{equation}
Since the functions $\chi_k$ have pairwise disjoint supports, and $\|\bfx - \bfx_k\| \le D_* + s_*/16$ for  $\bfx\in\supp(\chi_k)$, we obtain
\begin{equation}\label{eq:Duinfty}
\|Du\|_{L^\infty(\R^d)}\le C_d\left(1+\frac{D_*}{s_*}\right).    
\end{equation}
For the second term, differentiating \eqref{eq:vT-def} gives:
\[
\partial_{x_j} w^{(\ell)}(t,\bfx)
=
\sum_{k=1}^K \dot\gamma_k^{(\ell)}(t)\,\frac{8}{m}\,
\partial_j \xi\left(\frac{8(\bfx-\gamma_k(t))}{m}\right), \quad \text{ for all }j,\ell\in[d].
\]
For any fixed $t$, the support of the $k$-th summand is contained in $B(\gamma_k(t),m/4)$. These balls are pairwise disjoint for all $k$, meaning at most one summand is nonzero at any given $(t,\bfx)$. Since $|\dot\gamma_k(t)|\le G$ and $\|\nabla\xi\|_{L^\infty}<\infty$, we obtain
\begin{equation}\label{eq:Dwinfty}
\sup_{t\in[0,T]}\|D_\bfx w(t,\cdot)\|_{L^\infty(\R^d)}\le C_d\,\frac{G}{m}.    
\end{equation}
Substituting \eqref{eq:Duinfty}, \eqref{eq:Dwinfty}, and \eqref{eq:lambdadelta-def} into \eqref{eq:Lip.decomp} yields the desired bound \eqref{eq:lip.v.bound}.

\paragraph{Barron estimate.}  
Fix $n=d+4$ and $g\in\mathscr C_c^\infty(\R^{d+1})$. Since
$\frac{\|(\tau,\omega)\|_1^2}{(1+\|(\tau,\omega)\|_1)^n}$ is integrable on $\R^{d+1}$,
$$
\int_{\R^{d+1}}\|(\tau,\omega)\|_1^2\,|\widehat g(\tau,\omega)|\diff\tau\diff\omega
\le
C_d
\sup_{(\tau,\omega)\in\R^{d+1}}(1+\|(\tau,\omega)\|_1)^n\,|\widehat g(\tau,\omega)|.
$$
Expanding the weight as a sum of monomials and using the identity
$(\tau,\omega)^\alpha\widehat g=c_\alpha\widehat{\partial^\alpha g}$ together with
$\|\widehat h\|_{L^\infty}\le\|h\|_{L^1}$, we obtain
\begin{equation}\label{eq:fourier-sobolev}
\int_{\R^{d+1}}\|(\tau,\omega)\|_1^2\,|\widehat g(\tau,\omega)|\diff\tau\diff\omega
\le
C_d\sum_{|\alpha|\le n}\|\partial^\alpha g\|_{L^1(\R^{d+1})}.
\end{equation}
Moreover, for every component $\ell\in[d]$, the extension operator and the cutoff yield
\begin{equation}\label{eq:extension-L1}
\sum_{a+|\beta|\le n}\|\partial_t^a\partial_\bfx^\beta v^{(\ell)}\|_{L^1(\R\times\R^d)}
\le
C_{d,T}
\sum_{a+|\beta|\le n}\|\partial_t^a\partial_\bfx^\beta v_{[0,T]}^{(\ell)}\|_{L^1((0,T)\times\R^d)}.
\end{equation}
Applying \eqref{eq:fourier-sobolev} componentwise to $v$ and using \eqref{eq:extension-L1}, we obtain
\begin{equation}\label{eq:barron-reduction}
\|v\|_{\mathcal B_2^{d+1}}
\le
C_{d,T}
\sum_{\ell=1}^d \sum_{a+|\beta|\le n}
\|\partial_t^a\partial_\bfx^\beta v_{[0,T]}^{(\ell)}\|_{L^1((0,T)\times\R^d)}.
\end{equation}
Thus, it suffices to bound the $L^1$ derivatives of $v_{[0,T]}=v^C+v^T$, where
\[
v^C(t,\bfx)\coloneqq(1-\varphi(t))\,\lambda u(\bfx),
\qquad
v^T(t,\bfx)\coloneqq\varphi(t)\,w(t,\bfx).
\]
Since $\varphi$ is a fixed smooth profile, its derivatives up to order $n$ are bounded
in $L^\infty(0,T)$ and in $L^1(0,T)$ by a constant $C_{d,T}>0$; we will use this
repeatedly below without further mention.

\smallskip
\noindent
\emph{Compression.} Recalling the definition of $\chi_k$ at \eqref{eq:chi-k-def} and of $E_k$ above, Young's convolution inequality, combined with the isodiametric bound $|E_k|\le C_d(D_*+s_*)^d$ and the scaling $\|\partial^\gamma\psi_{s_*/32}\|_{L^1}\le C_d s_*^{-|\gamma|}$ yields, for every $|\gamma|\le n$,
\begin{equation}\label{eq:chi-L1}
\|\partial^\gamma\chi_k\|_{L^1(\R^d)}\le C_d\,(D_*+s_*)^d\,s_*^{-|\gamma|}.
\end{equation}
Applying the product rule to $u^{(\ell)}(\bfx)=\sum_k\chi_k(\bfx)(x_k^{(\ell)}-x^{(\ell)})$, and, using  $\|\bfx_k-\bfx\|\le D_*+s_*$ for $\bfx\in \supp(\chi_k)$, \eqref{eq:chi-L1} gives
\begin{equation}\label{eq:u-L1}
\sum_{|\beta|\le n}\|\partial^\beta u^{(\ell)}\|_{L^1(\R^d)}\le C_d\,K\,(D_*+s_*)^{d+1}\left(1+s_*^{-n}\right).
\end{equation}
Multiplying by the temporal factor $\lambda(1-\varphi)$ and using the uniform bound on the $L^1$-derivatives of $\varphi$, 
\begin{equation}\label{eq:compression-L1}
\sum_{a+|\beta|\le n}\|\partial_t^a\partial_x^\beta v^{C,(\ell)}\|_{L^1((0,T)\times\R^d)}\le C_{d,T}\,\lambda K\,(D_*+s_*)^{d+1}\left(1+s_*^{-n}\right).
\end{equation}

\smallskip
\noindent
\emph{Transport.} Differentiating $w^{(\ell)}(t,\bfx)=\sum_k\dot\gamma_k^{(\ell)}(t)\,\xi(8(\bfx-\gamma_k(t))/m)$ via the Leibniz and chain rules, every spatial derivative extracts a factor $m^{-1}$, and every time derivative either hits $\xi$ (extracting another $m^{-1}$) or lands on the prefactor $\dot\gamma_k^{(\ell)}$. Furthermore, we note that
\[\supp\xi(8(\,\cdot\,-\gamma_k(t))/m) \subset B(\gamma_k(t),m/4),\]
and that the latter are pairwise disjoint at each $t$. Thus, at most one summand contributes at any $(t,\bfx)$, so for all $a+|\beta|\le n$,
\begin{equation}\label{eq:pointwise-vT-summed}
\left|\partial_t^a\partial_x^\beta w^{(\ell)}(t,\bfx)\right|
\le
C_{d}\,G^{n+1}\sum_{r=0}^a m^{-(|\beta|+r)}\sum_{k=1}^K\mathbf 1_{B(\gamma_k(t),m/4)}(\bfx).
\end{equation}
Integrating over $(0,T)\times\R^d$, where each ball has volume $\lesssim m^d$, yields
\begin{equation}\label{eq:vT-L1}
\|\partial_t^a\partial_x^\beta w^{(\ell)}\|_{L^1((0,T)\times\R^d)}
\le C_{d,T}\,K\,G^{n+1}\sum_{r=0}^a m^{d-|\beta|-r}.
\end{equation}
Since $|\beta|+r\le n=d+4$, each exponent $d-|\beta|-r$ lies in $[-4,d]$, hence 
$m^{d-|\beta|-r}\le m^d+m^{-4}.$ Summing over $a+|\beta|\le n$ and absorbing the $\varphi$-factor from $v^T=\varphi\,w$, we obtain
\begin{equation}\label{eq:transport-L1}
\sum_{a+|\beta|\le n}\|\partial_t^a\partial_x^\beta v^{T,(\ell)}\|_{L^1((0,T)\times\R^d)}
\le C_{d,T}\,K\,G^{n+1}(m^d+m^{-4}).
\end{equation}
Combining \eqref{eq:compression-L1}, \eqref{eq:transport-L1}, and \eqref{eq:barron-reduction} with $n=d+4$, and setting $\lambda\le (3/T)\log_+(D_*/\delta)$ from \eqref{eq:lambdadelta-def}, proves \eqref{eq:steering-barron-bound}.

\paragraph{Uniform bound.}
Let $\bfx_0 \in [-R,R]^d$. For $t \in [0,2T/3]$,
we have $v(t,\bfx)=\lambda(1-\varphi(t))u(\bfx)$, since $w\equiv 0$ on
$[0,2T/3]\times\R^d$. The flow either fixes $\bfx_0$ or moves it along a straight segment within the convex set $[-R,R]^d$. Therefore, $\|\Phi_t(\bfx_0)\|\le \sqrt d\,R$ for all $t\in[0,2T/3]$.

For $t\in[2T/3,T]$ we have $v=w$, which vanishes outside
$\bigcup_{k=1}^K B(\gamma_k(t),m/4)$. Therefore, whenever $\Phi_t(\bfx_0)$ is not stationary, it lies in $B(\gamma_k(t),m/4)$ for some $k$, yielding
$$\|\Phi_t(\bfx_0)\| \le \|\gamma_k(t)\| + \frac{m}{4}.$$
By \eqref{eq:bd.traj} we have $\|\gamma_k(t)\| < 1+\max_{j\in[K]}\|\bfr_j\|+\sqrt d\,R$. Additionally, since $\gamma_k(0)\in A_k\subset[-R,R]^d$, we trivially have $m \le 2\sqrt d\,R$. Combining these estimates gives
$$\|\Phi_t(\bfx_0)\| \le 1+\max_{j\in[K]}\|\bfr_j\| + \frac{3}{2}\sqrt d\,R.$$
Since this bound dominates the compression phase and holds whenever the trajectory moves, we conclude
$$\max_{(t,\bfx)\in[0,T]\times[-R,R]^d}\|\Phi_t(\bfx)\| \le 1+\max_{k\in[K]}\|\bfr_k\|+\frac{3}{2}\sqrt d\,R,$$
which proves \eqref{eq:Linftybd}.
\end{proof}

The next lemma establishes an $O(p^{-1/2})$ approximation bound for the flow generated by \eqref{eq:saNODE}. The result relies on \cite[Theorem~2]{Klusowski2018Approximation}, which was subsequently adapted for dynamical systems in \cite{Li2024Universal}. Here, we extend \cite[Theorem~2.3]{Li2024Universal} by making the dependence on the Barron norm of the reference vector field explicit.  

\begin{lemma}\label{lem:barron_approximation}
Fix $d\ge2$ and $T>0$, let 
$v\in \mathscr C_c^\infty(\R\times\R^d;\R^d)$ with
$\|v\|_{\mathcal B_2^{d+1}}<\infty$, and let $\Omega\subset\R^d$ be compact.
Let $(\Psi_t)_{t\in[0,T]}$ be the flow induced on $[0,T]$ by
$\dot{\bfx}=v(t,\bfx)$, and assume that for some $R_\star\ge1$,
\begin{equation}\label{eq:barron_containment}
[0,T]\times \bigcup_{t\in[0,T]}\Psi_t(\Omega) \subset [-(R_\star-1),R_\star-1]^{d+1}.
\end{equation}
Then there exists a constant $C_d>0$, depending only on $d$, such that for every integer $p\ge3$ satisfying
\begin{equation}\label{eq:barron_p_cond}
p>C_d^2\,R_\star^4\,T^2\,\|v\|_{\mathcal B_2^{d+1}}^2
\exp\left(2T\,\sup_{t\in[0,T]}\Lip_\bfx (v(t,\cdot))\right),
\end{equation}
there exists a control $\theta$ of width $p$ for the field $v_\theta$ of the form \eqref{eq:saNODE} such that the flow $\Phi_t^\theta$ satisfies
\begin{equation}\label{eq:barron_bound}
\sup_{\bfx\in\Omega}\sup_{t\in[0,T]}
\|\Psi_t(\bfx)-\Phi_t^\theta(\bfx)\|
\le
\frac{C_d\,R_\star^2\,T\,\|v\|_{\mathcal B_2^{d+1}}}{\sqrt p}
\exp\left(T\,\sup_{t\in[0,T]}\Lip_\bfx (v(t,\cdot))\right).
\end{equation}
\end{lemma}

\begin{proof}
Denote $L \coloneqq \sup_{t\in[0,T]}\Lip_\bfx(v(t,\cdot))$.
By \cite[Theorem 2]{Klusowski2018Approximation}, there exists a control $\theta$
of width $p$ and a constant $C_d>0$ such that\footnote{\cite[Theorem 2]{Klusowski2018Approximation} is scalar-valued, but applying it componentwise and distributing the hidden units among the \(d\) components only changes the dimensional constant \(C_d\).}
\[
\sup_{(t,\bfx)\in[-R_\star,R_\star]^{d+1}}\|v(t,\bfx)-v_\theta(t,\bfx)\|
\le \frac{C_d\,R_\star^2\,\|v\|_{\mathcal B_2^{d+1}}}{\sqrt{p}} \eqqcolon \varepsilon.
\]
For any $\bfx\in\Omega$, let $\tau\le T$ be the maximal time such that
$\Phi_t^\theta(\bfx)\in[-R_\star,R_\star]^d$ for all $t\in[0,\tau]$.
On $[0,\tau]$, adding and subtracting $v(t,\Phi_t^\theta(\bfx))$ gives
\[
\|\Psi_t(\bfx)-\Phi_t^\theta(\bfx)\|
\le \int_0^t L\,\|\Psi_s(\bfx)-\Phi_s^\theta(\bfx)\|\diff s + \varepsilon t,
\]
so Grönwall's lemma yields
\[
\sup_{t\in[0,\tau]}\|\Psi_t(\bfx)-\Phi_t^\theta(\bfx)\|
\le \varepsilon T e^{LT} < 1,
\]
where the last inequality follows from \eqref{eq:barron_p_cond}.
Hence $\Phi_t^\theta(\bfx)$ stays within distance $1$ of $\Psi_t(\bfx)$, which
by \eqref{eq:barron_containment} lies in $[-(R_\star-1),R_\star-1]^d$. Thus
$\Phi_t^\theta(\bfx)\in(-R_\star,R_\star)^d$ on $[0,\tau]$, forcing $\tau=T$
by continuity. Taking the supremum over $\bfx\in\Omega$ gives \eqref{eq:barron_bound}.
\end{proof}

We are now in a position to prove the theorem.

\begin{proof}[Proof of \Cref{thm:cell_ctrl}]
Assume first that the targets $\bfr_k$ are pairwise distinct. Since $d\ge2$, we can choose smooth, pairwise disjoint $\gamma_k:[0,T]\to\R^d$ satisfying \eqref{eq:traj}--\eqref{eq:bd.traj}. We use the construction in the proof of \Cref{lem:steering-corridors} with these fixed
paths and tolerance \(\eta/4\). Let \(v\) be the resulting field and \((\Psi_t)_{t\in[0,T]}\) its flow.  Let $s_*$ be defined by \eqref{eq:steering-geom-data} and $(m,G)$ by \eqref{eq:steering-mGxi}, set \(\eta_0\coloneqq \min\{1,s_*/2,m/2\}\), and fix \(\eta\in(0,\eta_0]\). Therefore,
\begin{equation}\label{eq:psi}
\Psi_T(A_k)\subset B(\bfr_k,\eta/4) \qquad\text{for all }k\in[K].
\end{equation}
For $\mathfrak M>0$ given by \eqref{eq:defMp}, it follows that:
\begin{equation}\label{eq:psi_bdd_proof_cell}
[0,T]\times\bigcup_{t\in[0,T]}\Psi_t(I_R)\subset [-(\mathfrak M-1),\mathfrak M-1]^{d+1}.
\end{equation}
Moreover, if $\eta>0$ is small enough that $2\eta<s_*\wedge m$, then $\delta=\min\left\{\frac{\eta}{4},\frac{s_*}{8},\frac{m}{8}\right\}=\frac{\eta}{4}.$ \Cref{lem:steering-estimates} then yields:
\[
\sup_{t\in[0,T]}\Lip_\bfx(v(t,\cdot))
\le L\left(1-\log\eta\right),
\qquad
\|v\|_{\mathcal B_2^{d+1}}
\le B\left(1-\log\eta\right),
\]Increasing \(\mathfrak C\), if necessary, we also ensure that the lower bound
\(\eqref{eq:p_cell_ctrl_thm}\) implies the size condition
\(\eqref{eq:barron_p_cond}\) for every \(\eta\in(0,\eta_0]\), for constants $L,B\ge0$ that depend only on
$d,T,R,(\bfr_k)_k,(A_k)_k$. Applying \Cref{lem:barron_approximation} with $\Omega=I_R$ and $R_\star=\mathfrak M$, we obtain that for any sufficiently large $p\ge3$:
\begin{align*}
\sup_{\bfx\in I_R}\sup_{t\in[0,T]}\|\Psi_t(\bfx)-\Phi_t^\theta(\bfx)\|
&\le
\frac{C_d\mathfrak M^2 TB(1-\log\eta)}{\sqrt p}
\exp\left(LT(1-\log\eta)\right)\\
&\le
\frac{C_d\mathfrak M^2 TBe^{LT}(1-\log\eta)}{\sqrt p}\eta^{-LT}.
\end{align*}
Now, choose $\mathfrak{C}$ sufficiently large so that $\mathfrak{C}>2LT$ and $\sqrt{\mathfrak{C}}\ge 2C_d\mathfrak M^2 TBe^{LT}$. If $p$ satisfies the bound in \eqref{eq:p_cell_ctrl_thm}, we deduce:
\[
\sup_{\bfx\in I_R}\sup_{t\in[0,T]}\|\Psi_t(\bfx)-\Phi_t^\theta(\bfx)\|
\le
\frac{C_d\mathfrak M^2 TBe^{LT}}{\sqrt{\mathfrak{C}}}\,
\eta^{1+\mathfrak{C}/2-LT}
\le \frac{\eta}{2}.
\]
For any $\bfx\in A_k$, combining this bound with \eqref{eq:psi} allows us to verify (i) via the triangle inequality:
\[
\|\Phi_T^\theta(\bfx)-\bfr_k\|
\le
\|\Phi_T^\theta(\bfx)-\Psi_T(\bfx)\|
+
\|\Psi_T(\bfx)-\bfr_k\|
\le
\frac{\eta}{2}+\frac{\eta}{4}
<\eta.
\]
Furthermore, using \eqref{eq:psi_bdd_proof_cell}, we can easily verify (ii):
\[
\|\Phi_T^\theta\|_{L^\infty(I_R)}
\le
\|\Psi_T\|_{L^\infty(I_R)}
+
\|\Phi_T^\theta-\Psi_T\|_{L^\infty(I_R)}
\le
(\mathfrak M-1)+\frac{\eta}{2}
<
\mathfrak M,
\]
which holds because $\eta\le1$.

Finally, consider the case where the targets $\bfr_k$ are not necessarily pairwise distinct. We can slightly perturb them to choose auxiliary targets $\tilde{\bfr}_k$ that are pairwise distinct and satisfy:
\[
\|\tilde{\bfr}_k-\bfr_k\|<\eta/4
\qquad\text{for all }k.
\]
Applying the previous argument to these auxiliary targets $\tilde{\bfr}_1,\dots,\tilde{\bfr}_K$ with a tighter tolerance of $\eta/2$, we obtain, for sufficiently large $p$, a control $\theta$ such that $\Phi_T^\theta(A_k)\subset B(\tilde{\bfr}_k,\eta/2)$. By the triangle inequality, it immediately follows that $\Phi_T^\theta(A_k)\subset B(\bfr_k,\eta)$. Since the mutual separation of the auxiliary targets may be of order \(\eta\),
the path-separation parameter \(m\), and hence the constants produced by the previous argument, need not remain uniform as \(\eta\downarrow0\).
\end{proof}

\subsection{Proofs of \texorpdfstring{\Cref{s:nonparam}}{Section \ref*{s:nonparam}}}
\label{sec:proof_nonparam}

\begin{proof}[Proof of \Cref{prop:hist-rate}]
By the squared-sum inequality, the expected error can be decomposed into a bias and a variance component:
\[
\mathbb{E}_{\mathcal{D}_N}\left[\left\|y - y_{N,h}\right\|_{L^2(\mu)}^2\right]
\le
2 \left\|y - y_h\right\|_{L^2(\mu)}^2
+
2 \, \mathbb{E}_{\mathcal{D}_N}\left[\left\|y_h - y_{N,h}\right\|_{L^2(\mu)}^2\right],
\]
where $y_h$ and $y_{N,h}$ are the population and empirical estimators defined in \eqref{eq:pop-avg-def} and \eqref{eq:emp-avg-def}, respectively. We bound each term separately.

\noindent\textbf{Bias.} Fix a cell $Q_k$ with $p_k \coloneqq \mu(Q_k) > 0$. The population coefficient $\bfs_k$ is the conditional expectation of $y(\cdot)$ on $Q_k$. By Jensen's inequality, for any $\bfx \in Q_k$:
\[
\|y(\bfx) - \bfs_k\|
=
\Biggl\|\frac{1}{p_k}\int_{Q_k}\left(y(\bfx)-y(\bfx')\right)\diff\mu(\bfx')\Biggr\|
\le
\frac{1}{p_k}\int_{Q_k}\|y(\bfx)-y(\bfx')\|\diff\mu(\bfx').
\]
For any $\bfx, \bfx' \in Q_k$, the distance satisfies: $\|\bfx - \bfx'\| \le \diam(Q_k) \le \sqrt{d}\,h.$
Consequently, $\|y(\bfx) - y(\bfx')\| \le \omega_y(\sqrt{d}\,h)$, which implies $\|y(\bfx) - \bfs_k\| \le \omega_y(\sqrt{d}\,h)$. Since cells with $p_k = 0$ do not contribute to the $L^2(\mu)$ norm, integrating over the partition yields:
\begin{equation}\label{eq:bias_analysis}
\|y - y_h\|_{L^2(\mu)}^2
=
\sum_{k \in \mathcal{K}} \int_{Q_k}\|y(\bfx) - \bfs_k\|^2\diff\mu(\bfx)
\le
\omega_y(\sqrt{d}\,h)^2 \sum_{k \in \mathcal{K}} \mu(Q_k)
=
\omega_y(\sqrt{d}\,h)^2,
\end{equation}
where $\mathcal{K}$ is the index set defined in \eqref{eq:index.set}.

\noindent\textbf{Variance.} Since $y_h|_{Q_k} = \bfs_k$ and $y_{N,h}|_{Q_k} = \bfS_k$, the variance term becomes:
\begin{equation}\label{eq:variance_sum_revised}
\mathbb{E}_{\mathcal{D}_N}\left[\left\|y_h - y_{N,h}\right\|_{L^2(\mu)}^2\right]
=
\sum_{k \in \mathcal{K}} p_k\,\mathbb{E}_{\mathcal{D}_N}\left[\|\bfs_k - \bfS_k\|^2\right].
\end{equation}
Fix an index $k$ and let $N_k \coloneqq \sum_{i=1}^N \mathbf{1}_{Q_k}(\bfx_i)$ denote the number of samples falling into $Q_k$. Because the inputs are drawn i.i.d. from $\mu$, $N_k$ follows a binomial distribution, $N_k \sim \mathrm{Binom}(N, p_k)$. 

Conditionally on $N_k = n \ge 1$, the coefficient $\bfS_k$ is the empirical average of $n$ i.i.d. random variables $\bfY_1, \dots, \bfY_n$ distributed as $y(\bfX)$ for $\bfX \sim \mu(\cdot \mid Q_k)$. By definition, $\mathbb{E}[\bfY_1] = \bfs_k$. The variance of this empirical average is exactly:
\[
\mathbb{E}\left[\|\bfS_k - \bfs_k\|^2 \mid N_k = n\right]
=
\mathbb{E}\Biggl[\Biggl\|\frac{1}{n}\sum_{j=1}^n\left(\bfY_j-\bfs_k\right)\Biggr\|^2\Biggr]
=
\frac{1}{n}\,\mathbb{E}\left[\|\bfY_1-\bfs_k\|^2\right]
\le
\frac{\|y\|_{L^\infty(I_R)}^2}{n}.
\]
If $N_k = 0$, we defined $\bfS_k = 0$, which gives $\|\bfS_k - \bfs_k\|^2 = \|\bfs_k\|^2 \le \|y\|_{L^\infty(I_R)}^2$. Taking the unconditional expectation over the binomial variable $N_k$, we obtain:
\[
\mathbb{E}_{\mathcal{D}_N}\left[\|\bfS_k - \bfs_k\|^2\right]
\le
\|y\|_{L^\infty(I_R)}^2\left(\mathbb{E}\left[\frac{\mathbf{1}_{\{N_k>0\}}}{N_k}\right] + \mathbb{P}(N_k = 0)\right).
\]
We bound the two terms in the parenthesis. First, $\mathbb{P}(N_k = 0) = (1 - p_k)^N \le e^{-N p_k}$. Second, utilizing the elementary inequality $\frac{1}{n} \le \frac{2}{n+1}$ for $n \ge 1$, we evaluate the expectation via an integral identity for the binomial generating function:
\[
\mathbb{E}\left[\frac{\mathbf{1}_{\{N_k>0\}}}{N_k}\right]
\le
2\,\mathbb{E}\left[\frac{1}{N_k+1}\right]
=
2\int_0^1 \mathbb{E}[t^{N_k}]\diff t
=
2\int_0^1 (1 - p_k + p_k t)^N\diff t
\le
\frac{2}{p_k(N+1)}.
\]
Substituting these estimates into \eqref{eq:variance_sum_revised} yields:
\[
\mathbb{E}_{\mathcal{D}_N}\left[\left\|y_h - y_{N,h}\right\|_{L^2(\mu)}^2\right]
\le
\|y\|_{L^\infty(I_R)}^2 \sum_{k \in \mathcal{K}} \left(\frac{2}{N+1} + p_k e^{-N p_k}\right).
\]
Recall from \eqref{eq:number.cells} that the total number of cells is $K_h \lesssim h^{-d}$. Moreover, using $\sup_{p \ge 0} p e^{-N p} = (eN)^{-1}$, we can bound the sum uniformly:
\[
\mathbb{E}_{\mathcal{D}_N}\left[\left\|y_h - y_{N,h}\right\|_{L^2(\mu)}^2\right]
\le
\|y\|_{L^\infty(I_R)}^2\left(\frac{2 K_h}{N+1} + \frac{K_h}{eN}\right)
\lesssim
\frac{\|y\|_{L^\infty(I_R)}^2 K_h}{N}
\lesssim
\frac{\|y\|_{L^\infty(I_R)}^2}{N h^d}.
\]
Combining the bias and variance upper bounds directly gives \eqref{eq:hist-rate}.
\end{proof}

\begin{proof}[Proof of \Cref{cor:holder}]
Hölder regularity gives \(\omega_y(t)\leq L_y t^\alpha\) where $L_y > 0$ is the Hölder constant of $y(\cdot)$ on $I_R$. Hence \eqref{eq:bias_analysis} gives
\[
\|y-y_h\|_{L^2(\mu)}^2\le \omega_y(\sqrt d\,h)^2=L_y^2 (\sqrt{d}\,h)^{2\alpha} = L_y^2 d^\alpha h^{2\alpha}\lesssim h^{2\alpha}.
\]
Substituting this into \eqref{eq:hist-rate},
\[
\mathbb E_{\mathcal D_N}\|y-y_{N,h}\|_{L^2(\mu)}^2
\lesssim h^{2\alpha}+(Nh^d)^{-1}.
\]
The choice \(h_N=(2R)\wedge N^{-1/(2\alpha+d)}\) balances both terms and gives
\[
\mathbb E_{\mathcal D_N}\|y-y_{N,h_N}\|_{L^2(\mu)}^2
\lesssim N^{-2\alpha/(2\alpha+d)}. \qedhere
\]
\end{proof}

\begin{proof}[Proof of \Cref{prop:voronoi-rate}]
If $y(\cdot)$ is $\alpha$-Hölder with constant $L_y>0$, then
\begin{equation}\label{eq:voronoi_holder}
\|y(\bfx)-y_N^V(\bfx)\|
=
\|y(\bfx)-y(x_{\mathrm{NN}}(\bfx))\|
\le
L_y\|\bfx-x_{\mathrm{NN}}(\bfx)\|^\alpha
\le
L_y R_N^\alpha,
\end{equation}
where $R_N$ is the covering radius defined in \eqref{eq:cov.radius}. Therefore,
\[
\|y-y_N^V\|_{L^2(\mu)}^2
\le
\|y-y_N^V\|_{L^\infty(I_R)}^2
\le
L_y^2 R_N^{2\alpha}.
\]
Since $\rho$ is bounded below on the cube $I_R$, there exists $c>0$ such that $\mu(B(\bfx,r)\cap I_R)\ge c r^d$ for all $\bfx\in I_R$ and all sufficiently small $r>0$. Combining this estimate with \eqref{eq:cov_radius_rate} directly proves \eqref{eq:voronoi-rate}.
\end{proof}

\subsection{Proofs of \texorpdfstring{\Cref{s:generalization}}{Section \ref*{s:generalization}}}
\label{sec:proof_gen}

\begin{proof}[Proof of \Cref{lem:grid-layer}]
For a single cell $Q_k$ of side length $h$, the volume of its trimmed boundary rim is $ h^d - (h-2\delta)^d \lesssim h^{d-1}\delta$. Summing over the $K_h \asymp h^{-d}$ cells comprising the partition, the total volume is bounded by $|\Omega_\delta| \lesssim (h^{-d})(h^{d-1}\delta) = \delta/h$. Because the probability density of $\mu$ is bounded from above by a constant, the measure directly satisfies $\mu(\Omega_\delta) \lesssim |\Omega_\delta| \lesssim \delta/h$.
\end{proof}

To prove \Cref{lem:voronoi-layer}, we first isolate a preliminary estimate controlling the probability that the gap between the first two nearest-neighbor distances is small.

\begin{lemma}\label{lem:gap-small}
Assume $N\ge 2$, fix $\bfx\in I_R$, and let $D_1(\bfx)\le D_2(\bfx)$ be the distances from $\bfx$ to its nearest and second-nearest points within $\{\bfx_1,\dots,\bfx_N\}$. For $r>0$, denote $F_{\bfx}(r)\coloneqq \mu(B(\bfx,r))$. Then, for every $\delta>0$,
\[
\mathbb P_{\mathcal D_N}\left(D_2(\bfx)-D_1(\bfx)\le 2\delta\right)
\le
N(N-1)\int_0^\infty
\left[F_{\bfx}(r+2\delta)-F_{\bfx}(r)\right]\left(1-F_{\bfx}(r)\right)^{N-2}\diff F_{\bfx}(r).
\]
\end{lemma}

\begin{proof}
For each $i\in[N]$, set $Z_i\coloneqq \|\bfx-\bfx_i\|$. Writing the order statistics as $Z_{[1]}\le Z_{[2]}\le \cdots \le Z_{[N]},$
we have $D_1(\bfx)=Z_{[1]}$ and $D_2(\bfx)=Z_{[2]}$. By symmetry,
\[
\mathbb P\left(D_2(\bfx)-D_1(\bfx)\le 2\delta\right)
\le
N\,\mathbb P\left(Z_1=Z_{[1]},\ Z_{[2]}\le Z_1+2\delta\right).
\]
Conditioning with respect to $Z_1$, we see that on the event $\{Z_1=r,\ Z_1=Z_{[1]},\ Z_{[2]}\le r+2\delta\}$, at least one of the remaining $N-1$ distances must belong to $[r,r+2\delta]$, while all the others must be at least $r$. For a fixed index $j\in\{2,\dots,N\}$,
\[
\mathbb P\left(r\le Z_j\le r+2\delta\right)=F_{\bfx}(r+2\delta)-F_{\bfx}(r),
\]
and for each $k\notin\{1,j\}$, $\mathbb P(Z_k\ge r)=1-F_{\bfx}(r).$
By independence, these probabilities multiply, and a union bound over the $N-1$ possible choices of $j$ yields
\[
\mathbb P\left(Z_1=Z_{[1]},\ Z_{[2]}\le r+2\delta \mid Z_1=r\right)
\le
(N-1)\left[F_{\bfx}(r+2\delta)-F_{\bfx}(r)\right]\left(1-F_{\bfx}(r)\right)^{N-2}.
\]
Integrating with respect to the law of $Z_1$, whose distribution function is $F_{\bfx}$, proves the claim.
\end{proof}

\begin{proof}[Proof of \Cref{lem:voronoi-layer}]
We first treat the case $N=1$. Then $V_1=I_R$, so $\Omega_\delta=\{\bfx\in I_R:\dist(\bfx,\partial I_R)<\delta\}$, and since $\mu$ has density bounded above, $\mathbb E_{\mathcal D_N}[\mu(\Omega_\delta)] = \mu(\Omega_\delta) \lesssim \delta \le \delta N^{1/d}$. 

We may therefore assume $N\ge 2$. Let $\bfx\sim\mu$ be an independent test point; by Fubini's theorem,
\[
\mathbb E_{\mathcal D_N}\left[\mu(\Omega_\delta)\right]
=
\mathbb P\left(\bfx\in\Omega_\delta(\mathcal D_N)\right),
\]
where the probability is taken jointly over $\mathcal D_N$ and $\bfx$. The bound is trivial whenever $\delta N^{1/d}\ge1$ since probabilities are at most $1$, so we restrict to $\delta N^{1/d}\le 1.$

Let $D_1(\bfx)$ and $D_2(\bfx)$ denote the distances from $\bfx$ to its nearest and second-nearest sample points among $\{\bfx_1,\dots,\bfx_N\}$, respectively. We claim that
\begin{equation}\label{eq:omega-gap-inclusion}
\{\bfx\in\Omega_\delta,\ \dist(\bfx,\partial I_R)>\delta\}
\subset
\{D_2(\bfx)-D_1(\bfx)\le 2\delta\}.
\end{equation}
Suppose that $\bfx\in\Omega_\delta$ and $\dist(\bfx,\partial I_R)>\delta$, and let $V_i$ be the Voronoi cell containing $\bfx$. Since $\bfx\notin V_i^\delta$, there exists $\bfz\in\partial V_i$ with $\|\bfx-\bfz\|\le\delta$. The constraint $\dist(\bfx,\partial I_R)>\delta$ excludes $\bfz\in\partial I_R$, so $\bfz$ lies on an internal Voronoi boundary and is equidistant from at least two sample points $\bfx_i,\bfx_j$. The triangle inequality gives \[\|\bfx-\bfx_j\|\le \|\bfx-\bfz\|+\|\bfz-\bfx_i\|\le \|\bfx-\bfx_i\|+2\|\bfx-\bfz\|.\] Taking $\bfx_i$ as a nearest neighbor and $\bfx_j$ as a second-nearest neighbor of $\bfx$, we obtain $D_2(\bfx)-D_1(\bfx)\le 2\|\bfx-\bfz\|\le 2\delta$, which proves \eqref{eq:omega-gap-inclusion}. Consequently,
\begin{equation}\label{eq:split-layer}
\mathbb P\left(\bfx\in\Omega_\delta(\mathcal D_N)\right)
\le
\mathbb P\left(\dist(\bfx,\partial I_R)\le\delta\right)
+
\mathbb P\left(D_2(\bfx)-D_1(\bfx)\le 2\delta\right).
\end{equation}
Since $\mu$ has density bounded above and the Lebesgue measure of the $\delta$-strip near $\partial I_R$ is $O(\delta)$,
\begin{equation}\label{eq:boundary-strip}
\mathbb P\left(\dist(\bfx,\partial I_R)\le\delta\right)\lesssim \delta.
\end{equation}
For the second term, fix $\bfx\in I_R$ and define $F_{\bfx}(r)\coloneqq \mu(B(\bfx,r))$. The density bound on $\mu$ controls the $\mu$-mass of a spherical annulus by its Euclidean volume:
\begin{equation}\label{eq:sphericalannulus}
F_{\bfx}(r+2\delta)-F_{\bfx}(r)
\lesssim
\delta\,(r+2\delta)^{d-1}
\lesssim
\delta\, r^{d-1}+\delta^d.
\end{equation}
Substituting \eqref{eq:sphericalannulus} into \Cref{lem:gap-small} and using the identity $(N-1)\int_0^\infty(1-F_\bfx)^{N-2}\diff F_\bfx=1$ (immediate from the change of variable $s=F_\bfx(r)$),
\begin{equation}\label{eq:gap-prob-reduced}
\mathbb P\left(D_2(\bfx)-D_1(\bfx)\le2\delta \mid \bfx\right)
\lesssim
N\delta\, I_{\bfx}+N\delta^d,
\qquad
I_{\bfx}
\coloneqq
(N-1)\int_0^\infty r^{d-1}\left(1-F_{\bfx}(r)\right)^{N-2}\diff F_{\bfx}(r).
\end{equation}
It remains to bound $I_{\bfx}$ uniformly in $\bfx\in I_R$. Stieltjes integration by parts, whose boundary terms vanish (at $r=0$ since $d\ge2$, at $r=\infty$ since $\supp\mu\subset I_R$), gives
\[
I_{\bfx}
=
(d-1)\int_0^\infty r^{d-2}(1-F_{\bfx}(r))^{N-1}\diff r.
\]
Let $\rho_{\min}>0$ be a uniform lower bound for the density of $\mu$. Since $I_R$ is a cube, there exist constants $c>0$ and $r_0>0$ such that $|B(\bfx,r)\cap I_R|\ge c\,r^d$ for all $\bfx\in I_R$ and $0<r\le r_0$, hence $F_{\bfx}(r)\ge c\,r^d$ on the same range. Using $1-z\le e^{-z}$, we get $(1-F_{\bfx}(r))^{N-1}\le e^{-c(N-1)r^d}$ for $r\le r_0$, while for $r\ge r_0$, $1-F_\bfx(r)\le 1-F_\bfx(r_0)$ is uniformly bounded away from $1$. Setting $D_R\coloneqq\diam(I_R)$ and splitting at $r_0$,
\[
I_{\bfx}
\lesssim
\int_0^{r_0} r^{d-2}e^{-c(N-1)r^d}\diff r
+
\int_{r_0}^{D_R} r^{d-2}(1-F_{\bfx}(r_0))^{N-1}\diff r.
\]
The change of variables $u=(N-1)r^d$ in the first integral yields a bound $\lesssim N^{-(d-1)/d}$, and the second is exponentially small in $N$ and absorbed into the same bound. Hence
\begin{equation}\label{eq:Ix-bound}
I_{\bfx}\lesssim N^{-(d-1)/d}
\qquad\text{uniformly in }\bfx\in I_R.
\end{equation}
Substituting \eqref{eq:Ix-bound} into \eqref{eq:gap-prob-reduced} yields
\[
\mathbb P\left(D_2(\bfx)-D_1(\bfx)\le2\delta \mid \bfx\right)
\lesssim
\delta N^{1/d}+N\delta^d
\lesssim
\delta N^{1/d},
\]
where the last step uses $N\delta^d=(\delta N^{1/d})^d\le \delta N^{1/d}$ and $d\ge 2$. Taking expectations and combining with \eqref{eq:split-layer} and \eqref{eq:boundary-strip}, we conclude $\mathbb E_{\mathcal D_N}\left[\mu(\Omega_\delta)\right]
\lesssim
\delta+\delta N^{1/d}
\lesssim
\delta N^{1/d},$ since $N^{1/d}\ge1$.
\end{proof}

\begin{proof}[Proof of \Cref{cor:rates}]
For the histogram partition, apply \Cref{thm:template_bound} with $y_N=y_{N,h}$. Then \Cref{prop:hist-rate} and \Cref{lem:grid-layer} give
\[
\mathbb E_{\mathcal D_N}\left[\|y-y_{N,h}\|_{L^2(\mu)}^2\right]
\lesssim
h^{2\alpha}+\frac{1}{Nh^d},\qquad \mathbb E_{\mathcal D_N}\left[\mu(\Omega_\delta)\right]\lesssim \frac{\delta}{h}.
\]
Substituting these bounds into \eqref{eq:template_bound} gives the first estimate.

\noindent For the Voronoi partition, apply \Cref{thm:template_bound} with $y_N=y_N^V$. The expectation bound in \Cref{prop:voronoi-rate} and the boundary-layer estimate in
\Cref{lem:voronoi-layer} give
\[
\mathbb E_{\mathcal D_N}
\left[
\|y-y_N^V\|_{L^2(\mu)}^2
\right]
\lesssim
\left(\frac{\log N}{N}\right)^{2\alpha/d},
\qquad
\mathbb E_{\mathcal D_N}[\mu(\Omega_\delta)]
\lesssim
\delta N^{1/d}.
\]
Substituting these bounds into \eqref{eq:template_bound} gives the Voronoi estimate.
\end{proof}

\subsubsection*{Proofs of \Cref{prop:sanode-rate,prop:sanode-voronoi}}

\begin{proof}[Proof of \Cref{prop:sanode-rate}]
Define $\eta_N^2\coloneqq N^{-\frac{2\alpha}{2\alpha+d}},$ and $h_N\coloneqq N^{-\frac{1}{2\alpha+d}},
\delta_N\coloneqq h_N^{1+2\alpha},$
and apply the histogram bound in \Cref{cor:rates} with $h=h_N$ and $\delta=\delta_N$.
By \Cref{thm:cell_ctrl}, there exists $p_N^*\in\N$
given by \eqref{eq:p_cell_ctrl_thm}, such that for every $p_N \ge p_N^*$
one can choose a control for \eqref{eq:saNODE} of width $p_N$ with
$\eta^2 \asymp N^{-\frac{2\alpha}{2\alpha+d}}$,
whose flow is uniformly bounded on $I_R$.
Moreover, \Cref{prop:hist-rate} and \Cref{cor:holder} give
\[
\mathbb{E}_{\mathcal{D}_N}\left[\|y-y_{N,h_N}\|_{L^2(\mu)}^2\right]
\lesssim
h_N^{2\alpha}+\frac{1}{Nh_N^d}
\asymp
N^{-\frac{2\alpha}{2\alpha+d}},
\]
and \Cref{lem:grid-layer} yields $\mathbb{E}_{\mathcal{D}_N}\left[\mu(\Omega_{\delta_N})\right]
\lesssim
\delta_N h_N^{-1}
=
h_N^{2\alpha}
=
N^{-\frac{2\alpha}{2\alpha+d}}.$ If $N > 2^{1+d/(2\alpha)}$ then $\delta_N < h_N/2$ and substituting these bounds into \eqref{eq:non-param} yields the claim.
\end{proof}

\begin{proof}[Proof of \Cref{prop:sanode-voronoi}]
 Let $p_N$ satisfy condition \eqref{eq:p_cell_ctrl_thm} for the Voronoi
partition generated by $\{\bfx_i\}_{i=1}^N$, with margin $\delta_N$ and tolerance $\eta_N$ defined by
\[
\eta_N^2
\coloneqq
\left(\frac{\log N}{N}\right)^{\frac{2\alpha}{d}},
\qquad
\delta_N
\coloneqq
(\log N)^{\frac{2\alpha}{d}}N^{-\frac{2\alpha+1}{d}}.
\]
Substituting the bounds \eqref{eq:voronoi-rate} from \Cref{prop:voronoi-rate} and \eqref{eq:layer-explicit} from \Cref{lem:voronoi-layer} into \eqref{eq:non-param} yields the conclusion.
\end{proof}


\bibliographystyle{abbrv} 
\bibliography{biblio.bib}

\end{document}